\documentclass[11pt,a4paper, parskip=half]{scrartcl} 
\usepackage[utf8]{inputenc}
\usepackage[english]{babel}
\usepackage[T1]{fontenc}
\usepackage{lmodern}
\usepackage[left=3cm,right=3cm,top=3cm,bottom=3cm]{geometry}
\usepackage[runin]{abstract}
    \abslabeldelim{.}

\usepackage{todonotes}

\usepackage{amsmath, amsfonts, amssymb, amsthm, amstext}
\usepackage{mathtools}
\usepackage{mathrsfs}
\usepackage{mathdots}
\usepackage{cases}

\usepackage{graphicx}
\usepackage{subcaption}
\usepackage{algorithm}

\font\mfett=cmmib10 at11pt
 at9pt
\def\bgamma{\hbox{\mfett\char013}}
\def\gamra{\hbox{\mfett\char013}}

\def\bphi{\hbox{\mfett\char030}}

\usepackage{enumitem}

\usepackage{hyperref} 
\usepackage[capitalise]{cleveref}
	
\newcounter{thm}
\numberwithin{thm}{section}
\numberwithin{equation}{section}

	\newtheoremstyle{myplain}		
			{}			
			{}			
			{\itshape}				
			{}				
			{\sffamily\bfseries}				
			{.}		
			{ }				
			{\thmname{#1}\thmnumber{ #2}\textnormal{\textsf{\thmnote{ (#3)}}}}			
    \newtheoremstyle{mybreak}
            {}{}{}{}{\sffamily\bfseries}{.}{\newline}
            {\thmname{#1}\thmnumber{ #2}\textnormal{\textsf{\thmnote{ (#3)}}}}
	\newtheoremstyle{mydef}
			{}{}{}{}{\sffamily\bfseries}{.}{ }
			{\thmname{#1}\thmnumber{ #2}}
	\newtheoremstyle{myrem}
			{}{}{}{}{\sffamily\itshape}{.}{ }
			{\thmname{#1}\thmnumber{ #2}}

\theoremstyle{myplain}
	\newtheorem{theorem}[thm]{Theorem}

\theoremstyle{mybreak}
\theoremstyle{mydef}
	
	\newtheorem{remark}[thm]{Remark}
\theoremstyle{mydef}
	\newtheorem{example}[thm]{Example}

	\newcommand{\cc}{\mathbb{C}}
		
		\newcommand{\nn}{\mathbb{N}}

\newcommand{\argmax}{\mathop{\mathrm{argmax}}}

\allowdisplaybreaks

\def\sumprime_#1^#2{
    \setbox0=\hbox{$\scriptstyle{#1}$}
    \setbox1=\hbox{$\scriptstyle{#2}$}
    \setbox2=\hbox{$\displaystyle{\sum}$}
    \setbox4=\hbox{${}^\prime\mathsurround=0pt$}
    \dimen0=.5\wd0 \advance\dimen0 by-.5\wd2
    \ifdim\dimen0>0pt
        \ifdim\dimen0>\wd4 \kern\wd4
        \else\kern\dimen0
        \ifdim\dimen1>\wd4 \kern\wd4
        \else\kern\dimen1
    \fi\fi\fi
\mathop{{\sum}^\prime}_{\kern-\wd4 #1}^{\kern-\wd4 #2}
}

\title{\Large From ESPRIT to ESPIRA:  Estimation of  Signal  Parameters  by Iterative Rational Approximation}
\author{Nadiia Derevianko\footnote{Institute for Numerical and Applied Mathematics, G\"ottingen University, Lotzestr.\ 16-18, 37083 G\"ottingen, Germany, \{n.derevianko,plonka,m.petz\}@math.uni-goettingen.de} \footnote{Corresponding author} \quad Gerlind Plonka$^{*}$ \quad Markus Petz$^{*}$}
\date{}

\begin{document}
	\let\oldproofname=\proofname
	\renewcommand{\proofname}{\itshape\sffamily{\oldproofname}}

\maketitle

\begin{abstract}
We introduce a new method for Estimation of Signal Parameters based on Iterative Rational Approximation (ESPIRA) for  sparse exponential sums. Our algorithm uses the AAA algorithm  for rational approximation of the discrete Fourier transform  of the given equidistant signal values.
We show that ESPIRA can be interpreted as  a matrix pencil method  applied to Loewner matrices.
These Loewner matrices  are closely connected  with the Hankel matrices which are usually employed for signal recovery. Due to the construction of the Loewner matrices via an adaptive selection of index sets, the matrix pencil method is stabilized. ESPIRA achieves similar  recovery results for exact data as ESPRIT and the matrix pencil method  but  with less computational  effort. Moreover, ESPIRA strongly
outperforms ESPRIT and the matrix pencil method  for noisy data and for signal approximation by short exponential sums.\\[1ex]
\textbf{Keywords:}  sparse exponential sums, matrix pencil method, ESPRIT, rational interpolation, AAA algorithm, Loewner matrices, Hankel matrices.\\
\textbf{AMS classification:}
41A20, 42A16, 42C15, 65D15, 94A12.
\end{abstract}

\section{Introduction}
\label{introduction}

We consider  exponential sums of the form 
\begin{equation}\label{1.1}
f(t) = \sum_{j=1}^{M} \gamma_{j} \, {\mathrm e}^{\phi_{j}t} = \sum_{j=1}^{M} \gamma_{j} \, z_{j}^{t},
\end{equation}
where $M \in {\mathbb N}$, $\gamma_{j} \in {\mathbb C}\setminus \{ 0\}$ and $z_{j}= {\mathrm e}^{\phi_{j}} \in {\mathbb C}\setminus \{ 0\}$ with $\phi_{j} \in {\mathbb C}$ are pairwise distinct.
  We are interested in the following reconstruction problem:
For given (possibly perturbed) function values $f_{k} \coloneqq f(k)$, $k=0, \ldots, 2N-1$, $N\geq M$,
the goal is to recover the parameters $M$, $\gamma_{j}$, $z_{j}$, $j=1, \ldots , M$, that determine $f$ in (\ref{1.1}).

The recovery of exponential sums of the form (\ref{1.1}) from a finite set of possibly corrupted signal samples plays an important role in many signal processing applications, see e.g. 
\cite{Beinert17, BM05, Bossmann12,  Antoulas14, Knaepkens20, Lobos03, Pereyra10, PT14, VMB02}, or \cite{PPST18}, Chapter 10. 
For example, such reconstruction methods for exponential sums have been successfully applied  in phase retrieval \cite{Beinert17},  signal approximation  \cite{BM05,BH05}, sparse deconvolution in nondestructive testing \cite{Bossmann12}, model reduction in system theory \cite{Antoulas14},  direction of arrival estimation \cite{Knaepkens20}, exponential data fitting \cite{Pereyra10}, or reconstruction of signals  with finite rate of innovation \cite{VMB02}.
Often, the exponential sums occur as Fourier transforms or higher order moments of discrete measures (or streams of Diracs) of the form $\sum_{j=1}^{M} \gamma_{j} \, \delta( \cdot - T_{j})$ with $T_{j} \in {\mathbb R}$, see e.g.\ \cite{CRT06, Duval15, Fern16}, which leads to the special case that $\phi_{j} = {\mathrm i} T_{j}$ in  (\ref{1.1}) is purely complex, i.e., $|z_{j}| = 1$. Other applications employ decaying exponential sums with  real knots $z_{j} \in (-1,1)$, see \cite{BM05, H2019}.

If instead of parameters $z_{j}$ the frequency parameters $\phi_{j}$ in (\ref{1.1}) need to be recovered, then $\phi_{j} = \ln(z_{j})$ only provides $\textrm{Im}(\phi_{j})$ modulo $2\pi$. In this case, a correct recovery of $\textrm{Im}(\phi_{j})$ can be achieved by considering $\tilde{z}_{j} = z_{j}^{h}$ with a suitable sampling step size $h$, where $h$ is taken such that $h \, \textrm{Im}(\phi_{j}) \in [-\pi, \pi)$ for all $j=1, \ldots ,M$. Therefore, to fix $h$,   one needs  a priori knowledge on an upper bound for $\max_{j} |\textrm{Im}(\phi_{j})|$. Then, instead of $f_{k} \coloneqq f(k)$,  one employs here the samples $f_{k} \coloneqq f(t_{0}+  h k)$, $k=0, \ldots, 2N-1$, $N \ge M$, with   $t_{0} \in {\mathbb R}$.  This case can be transferred to our setting by  choosing $\tilde{f}(t) = \sum_{j=1}^{M} \Tilde{\gamma}_{j} \, \tilde{z}_{j}^{t}$ with 
$\tilde{\gamma}_{j} = \gamma_{j} z_{j}^{t_{0}}$ and $\tilde{z}_{j} = z_{j}^{h}$, such that $\tilde{f}(k)= f(t_{0}+ hk)$.

To recover the parameters  $\gamma_{j}$ and $z_{j}$ in (\ref{1.1}), one often employs Prony's method, which is based on the observation  that (\ref{1.1})
can be seen  as a solution  to a linear difference equation  of order $M$ with constant coefficients. Let the characteristic polynomial (Prony polynomial)
$$
 p(z) \coloneqq \prod_{j=1}^{M} (z - z_{j})
 $$ 
 be defined by the zeros $z_{j}$, which are the pairwise distinct parameters in (\ref{1.1}), and consider the corresponding monomial representation $p(z) = z^{M} + \sum_{\ell=0}^{M-1} p_{\ell} \, z^{\ell}$. Then the coefficients $p_{\ell}$ satisfy 
\begin{eqnarray*}
\sum_{\ell=0}^{M-1} p_{\ell} \, f_{\ell+m} &=& \sum_{\ell=0}^{M-1} p_{\ell} \sum_{j=1}^{M} \gamma_{j} \, z_{j}^{\ell + m}
= \sum_{j=1}^{M} \gamma_{j} \, z_{j}^{m} \sum_{\ell=0}^{M-1} p_{\ell} \, z_{j}^{\ell} \\
&=&  \sum_{j=1}^{M} \gamma_{j} \, z_{j}^{m} (p(z_{j}) - z_{j}^{M}) = - \sum_{j=1}^{M} \gamma_{j} \, z_{j}^{m+M} = -f_{m+M}
\end{eqnarray*}
for $m=0, \ldots , 2N-M-1$, and we obtain  a structured linear system  to compute the coefficient vector ${\mathbf p} = (p_{\ell})_{\ell=0}^{M-1}$, 
\begin{equation}\label{1.2}
\left( f_{\ell+m} \right)_{m,\ell=0}^{2N-M-1,M-1} \, {\mathbf p} = \left( -f_{m+M} \right)_{m=0}^{2N-M-1}.
\end{equation}
Equivalently,
\begin{equation}\label{1.3}
 \left( f_{\ell+m} \right)_{m,\ell=0}^{2N-M-1,M} \, \begin{pmatrix} {\mathbf p} \\ 1 \end{pmatrix} = {\mathbf 0}, 
 \end{equation}
i.e., $({\mathbf p}^{T}, 1)^{T}$ is a kernel vector of the Hankel matrix $\left( f_{\ell+m} \right)_{m,\ell=0}^{2N-M-1,M}$.
If ${\mathbf p}$ is known, then  we build $p(z)$, compute the zeros $z_{j}$ of $p(z)$ and  find the parameters $\gamma_{j}$ in (\ref{1.1}) as a least squares solution to  the overdetermined system
$$ \sum_{j=1}^{M} \gamma_{j} \, z_{j}^{k} = f_{k}, \qquad k=0, \ldots, 2N-1.$$
The described procedure  is known as the classical Prony method. It can  be simply observed that the Hankel matrix
\begin{equation}\label{HH} {\mathbf H}_{2N-M,M+1} \coloneqq \left( f_{\ell+m} \right)_{m,\ell=0}^{2N-M-1,M} 
\end{equation}
has indeed rank $M$ since it admits the factorization 
\begin{equation}\label{1.4}
{\mathbf H}_{2N-M,M+1} = {\mathbf V}_{2N-M,M}({\mathbf z}) \, \mathrm{diag}  \Big(\left( \gamma_{j} \right)_{j=1}^{M}  \Big) \, {\mathbf V}_{M+1,M}({\mathbf z})^{T}, 
\end{equation}
where ${\mathbf z} \coloneqq (z_{1}, \ldots , z_{M})^{T}$ and  
\begin{equation}\label{V} {\mathbf V}_{L,M}({\mathbf z}) \coloneqq \left( z_{j}^{m} \right)_{m=0,j=1}^{L-1,M} 
\end{equation}
denotes the Vandermonde matrix  generated by the knots $z_{j}$. All three matrix factors  in (\ref{1.4})  have full rank $M$ by assumption.
Thus, ${\mathbf p}$  is uniquely determined by  (\ref{1.3})  and the parameters $z_{j}$, $\gamma_{j}$, $j=1, \ldots , M$, are uniquely defined for $N\ge M$.
\medskip

However, the classical Prony method is numerically 
unstable, since the Hankel matrix in (\ref{1.2}) has usually a very large condition number already for small $M$.
If the distance between different frequency parameters $\phi_{j_{1}}$ and $\phi_{j_{2}}$ (or equivalently between knots $z_{j_{1}}$ and $z_{j_{2}}$)  gets small, then the reconstruction problem  changes to an ill-posed setting \cite{Bat2018}, and the condition numbers of the involved Vandermonde matrices in (\ref{1.4}) 
grow exponentially. 
Therefore  one needs to choose a numerical procedure to solve the parameter estimation problem (\ref{1.1}) with care.
There have been several ideas  to improve the  stability  of the classical Prony method described above including  MUSIC \cite{Schmidt86},  ESPRIT \cite{Hua90}, and the matrix pencil method (MPM) \cite{RK89}.
Further, several modifications have been proposed  to achieve lower computational cost \cite{PT15} or to ensure consistency of the parameter estimation in case of noisy data \cite{BM86,Osborne95,ZP18}.
The two most frequently used approaches MPM and ESPRIT are based on the same concept, namely the solution of a Hankel matrix pencil problem to recover the parameters $z_{j}$ in (\ref{1.1}). For a survey and comparison of these two algorithms we refer to \cite{PT2013}.
\smallskip

In our recent work \cite{DP21, PPD21}, we have investigated a different approach for parameter reconstruction for exponential sums, where we employed a finite set of Fourier coefficients of the signal $f$ that occur in a Fourier expansion of $f$ on a finite interval $(0, P)$.
We have shown in \cite{DP21, PPD21} that the signal recovery problem can then be  rephrased  as a rational interpolation  problem, which in turn is solvable in a stable way  using  the AAA algorithm \cite{NST18}. The AAA algorithm is available in Chebfun, see \cite{DHT14}.
A related approach  that connects trigonometric rational approximation and Prony techniques has been recently proposed in \cite{WDT21}.
Further,  in \cite{BM09} and \cite{PP19} somewhat converse ideas have been applied for adaptive signal representation, where the Fourier coefficients  are approximated by short exponential sums.

However, in most applications, only the equidistant  signal values  of $f(t)$ in (\ref{1.1}) are available or can be computed. Therefore, in this paper we will study the following questions.
\begin{itemize}
\item How can the  adaptive rational approximation method based on the AAA algorithm  be applied to reconstruct $f$ from equidistant function values?
\item  How does the resulting algorithm relate to the known stable recovery algorithms MPM and ESPRIT?
\item How does the resulting algorithm perform compared  to MPM and ESPRIT?
\end{itemize}

We will introduce a new method for \textbf{E}stimation of \textbf{S}ignal \textbf{P}arameters based on \textbf{I}terative  \textbf{R}ational \textbf{A}pproximation (ESPIRA). In particular, we will  show that the new ESPIRA algorithm can be understood as  a matrix pencil algorithm, but for Loewner matrices instead of Hankel matrices. The Loewner matrix occurring in our method  can be factorized  into a product of three matrices, namely the Hankel matrix in (\ref{1.4}) and two  structured matrices, which can be interpreted as inverses of certain Vandermonde matrices (up to  multiplication with a diagonal matrix). The Loewner  matrices evolve  from the adaptive strategy of choosing interpolation points in the AAA algorithm, and this greedy index selection strongly stabilizes the consecutive  procedure to solve the matrix pencil problem.
The connection  of Hankel matrices and Loewner matrices used in our method  goes back to Fiedler \cite{Fiedler}.  
 The importance of Loewner matrices in scalar rational interpolation has been extensively investigated in \cite{AA86,Be70}.
The obtained ESPIRA algorithms work with similar exactness as MPM and ESPRIT for exact data but essentially outperform these methods in case of noisy input data and for signal approximation. We conjecture that ESPIRA-II is indeed  statistically consistent  while MPM and ESPRIT are not,  see \cite{Osborne95, ZP18}. Moreover, for larger data sets, the two ESPIRA algorithms require much less computational effort. While MPM and ESPRIT need ${\mathcal O}(N^{3})$ flops  to achieve optimal recovery results, the  ESPIRA algorithms have computational costs of  ${\mathcal O}(N (M^{3}+\log N))$ where $M$  denotes the length of the exponential sum in (\ref{1.1}) and is usually small and where $2N$ is the number of given signal samples.
\smallskip

\textbf{Outline of the paper.}
 In Section 2, we  survey the MPM and ESPRIT method, the currently most used methods for recovery of exponential sums. As already shown in \cite{PT2013}, these two methods are closely related and based on solving a matrix pencil problem for two closely related Hankel matrices.

Section 3 is devoted to the new ESPIRA approach which is based on rational approximation of the DFT coefficients of the given data vector ${\mathbf f}= (f_{k})_{k=0}^{2N-1}$. We introduce the ESPIRA-I algorithm by representing the parameter reconstruction problem as a rational interpolation problem at knots on the unit circle in Section 3.1. To solve this problem, we employ the AAA algorithm in \cite{NST18} for our setting in Section 3.2. The adaptively constructed  Loewner matrices occurring in the AAA algorithm  will play a crucial role in our further investigations. In a further step, the obtained rational interpolant needs to be transferred into a partial fraction decomposition, such that all wanted parameters of the signal can be determined, see Section 3.3. In the special case that some of the wanted knots $z_{j}$ in (\ref{1.1}) satisfy $z_{j}^{2N} = 1$, the fractional structure of the DFT coefficients is lost, and the recovery process needs to be modified as shown in Section 3.4. In Section 3.5 we prove that the iterative AAA algorithm used in ESPIRA terminates for exact input data after exactly $M+1$ iteration steps.

In Section 4, we investigate the relations between ESPIRA, ESPRIT and MPM.  Employing the results in \cite{Fiedler},
we show that there is a close relation between the Hankel matrix ${\mathbf H}_{2N-M,M+1}$ in (\ref{1.4}) and the Loewner matrix 
$$ {\mathbf L}_{2N-M-1,M+1} = \left( \frac{\omega_{2N}^{\ell} \, \hat{f}_{\ell} - \omega_{2N}^{k} \hat{f}_{k} }{\omega_{2N}^{-\ell} - \omega_{2N}^{-k}} \right)_{\ell\in \Gamma_{M+1}, k \in S_{M+1}} $$
where $S_{M+1} \cup \Gamma_{M+1} = \{0, \ldots , 2N-1\}$ is a partition of the index set, and where $S_{M+1}$ has $M+1$ components. This relation provides a simple proof that ${\mathbf L}_{2N-M-1,M+1}$  has rank $M$ (in case of exact data) and leads to the new algorithm ESPIRA-II, which is again a matrix pencil method, but for Loewner instead of Hankel matrices. ESPIRA-II uses only the adaptively found index set $S_{M+1}$ from the AAA algorithm but  no longer computes the rational interpolant itself. Therefore a special treatment of knots $z_{j}$ with $z_{j}^{2N} = 1$ is no longer needed.

The numerical experiments in Section 5 show that ESPIRA-I and ESPIRA-II provide very good recovery results for exact input data. 
In the case of noisy data the new algorithms strongly outperform MPM and ESPRIT. In particular, good recovery results are provided, where MPM and ESPRIT fail completely. 
Moreover,  our algorithms can also be used for approximation of functions by short exponential sums, as shown in Section 6. Using double precision arithmetic in \textsc{Matlab}, we achieve almost the same accuracy for approximation of $f(t) = \frac{1}{t+1}$ in $[0,1]$ as the Remez algorithm in \cite{BH05}, which has been performed in high precision arithmetic. Our approximation results for Bessel functions outperform earlier results in \cite{BM05} and in \cite{PT2013}. In a final example we approximate the Dirichlet function with high accuracy, where MPM and ESPRIT completely fail.

\smallskip

Throughout this paper, we will use the matrix notation ${\mathbf A}_{L,K}$ for complex matrices of size $L \times K$ and the submatrix notation
${\mathbf A}_{L,K}(m:n,k:\ell)$ to denote a  submatrix of ${\mathbf A}_{L,K}$ with rows indexed $m$ to $n$ 
and columns indexed $k$ to $\ell$, where (as in \textsc{Matlab}) the first row and first column has index $1$ (even though the row- and column indices for the definition of the matrix may start with $0$). For square matrices  we often use the short notation ${\mathbf A}_{N}$ instead of ${\mathbf A}_{N,N}$.

\section{Matrix Pencil Method and ESPRIT Algorithm}
\label{sec:MPM}

In order to improve the numerical stability  of the parameter estimation in exponential sums (\ref{1.1}), there have been several attempts  to stabilize  the estimation procedure.
Nowadays, the most frequently used methods are the matrix pencil method \cite{Hua90} and the ESPRIT method \cite{RK89}. 
As shown in \cite{PT2013},  these two methods  are very closely related  and are both based  on suitable factorizations of the  underlying Hankel matrix.
In this section we shortly summarize these two recovery methods, which are later compared to our new approach.
\medskip

Assume that $M$, the number  of terms in (\ref{1.1}),  is  unknown and we only have an upper bound  $L$ with $M \le L \le N$ and the given data 
$f_{k}= f(k)$, $k=0, \ldots , 2N-1$.  If the data $f_{k}$ are exact, then
\begin{equation}\label{1.5}
{\mathbf H}_{2N-L,L+1} = \left( f_{k+\ell}\right)_{k,\ell=0}^{2N-L-1,L}
\end{equation}
possesses exactly rank $M$, since the structure of $f_{k}$ in (\ref{1.1}) directly implies
$$ {\mathbf H}_{2N-L,L+1} = {\mathbf V}_{2N-L,M}({\mathbf z})  \, \mathrm{diag}  \Big( \left( \gamma_{j} \right)_{j=1}^{M} \Big) \, {\mathbf V}_{L+1,M}({\mathbf z})^{T}, 
$$
similarly as in (\ref{1.4})--(\ref{V}). Thus, theoretically, the number $M$ of terms in (\ref{1.1}) is given  as the rank of ${\mathbf H}_{2N-L,L+1}$.  Practically, if the measurements $f_{k}$ are slightly perturbed, we have to compute the numerical rank  of ${\mathbf H}_{2N-L,L+1}$.

Consider the two Hankel matrices 
$$ {\mathbf H}_{2N-L,L}(0) \coloneqq  \left( f_{k+\ell}\right)_{k,\ell=0}^{2N-L-1,L-1}, \qquad  {\mathbf H}_{2N-L,L}(1) \coloneqq  \left( f_{k+\ell+1}\right)_{k,\ell=0}^{2N-L-1,L-1},
$$
both of the same size, where the first matrix is obtained from ${\mathbf H}_{2N-L,L+1}$ in (\ref{1.5})  by removing the last column, and the second matrix  is found by removing the first column. Then these matrices 
satisfy the factorizations
\begin{eqnarray}\label{extra}
{\mathbf H}_{2N-L,L}(0) &=&  {\mathbf V}_{2N-L,M}({\mathbf z})  \, \mathrm{diag}  \Big(\left( \gamma_{j} \right)_{j=1}^{M} \Big)\, {\mathbf V}_{L,M}({\mathbf z})^{T}, \\
\label{extra1}
{\mathbf H}_{2N-L,L}(1) &=&  {\mathbf V}_{2N-L,M}({\mathbf z})  \, \mathrm{diag}  \Big( \left( \gamma_{j} \, z_{j} \right)_{j=1}^{M} \Big) \, {\mathbf V}_{L,M}({\mathbf z})^{T}.
\end{eqnarray}

\subsection{Matrix Pencil Method (MPM)}
The matrix pencil method (MPM) is essentially based  on the observation  that the rectangular matrix pencil
\begin{equation}\label{1.6}
z \, {\mathbf H}_{2N-L,L}(0) - {\mathbf H}_{2N-L,L}(1)
\end{equation}
possesses the wanted  parameters $z=z_{j}$ as generalized eigenvalues, i.e.,  for $z=z_{j}$, $j=1, \ldots, M$,  the  matrix 
$z_{j}\, {\mathbf H}_{2N-L,L}(0) - {\mathbf H}_{2N-L,L}(1)$  has a nontrivial kernel.
Indeed, using the factorizations of  ${\mathbf H}_{2N-L,L}(0) $ and ${\mathbf H}_{2N-L,L}(1)$ in (\ref{extra})--(\ref{extra1}),
it follows  for any $z_{j} \in \{z_{1}, \ldots , z_{M} \}$ that the matrix
$$ z_{j} \,  {\mathbf H}_{2N-L,L}(0) - {\mathbf H}_{2N-L,L}(1) = {\mathbf V}_{2N-L,M}({\mathbf z})  \, \mathrm{diag}  \Big( \left( (z_{j} - z_{k})\gamma_{k} \right)_{k=1}^{M} \Big) \, {\mathbf V}_{L,M}({\mathbf z})^{T}$$
possesses only rank $M-1$, and therefore $z_{j}$ is an eigenvalue of the matrix pencil in (\ref{1.6}).

For a stable recovery of ${\mathbf z}$ from the  generalized (matrix pencil) eigenvalue problem (\ref{1.6}), a QR decomposition  is applied to 
$({\mathbf H}_{2N-L,L}(0), {\mathbf H}_{2N-L,L}(1)) \in {\mathbb C}^{2N-L,2L}$ in \cite{Hua90}.
However, as proposed in \cite{PT2013}, it suffices to  compute a QR decomposition (with column pivot) of ${\mathbf H}_{2N-L,L+1}$ in (\ref{1.5}).
More precisely,  let 
\begin{equation}\label{QR} {\mathbf H}_{2N-L,L+1} \, {\mathbf P}_{L+1}  = {\mathbf Q}_{2N-L} \, {\mathbf R}_{2N-L,L+1} 
\end{equation}
be a QR decomposition of ${\mathbf H}_{2N-L,L+1}$  with column pivot, where ${\mathbf P}_{L+1}$ is a permutation matrix and ${\mathbf Q}_{2N-L}$  is a unitary square matrix.
If we have exact input data, then 
${\mathbf R}_{2N-L,L+1}$  is a trapezoidal  matrix of the form $\left( \begin{array}{c} {\mathbf R}_{M,L+1} \\ {\mathbf 0} \end{array} \right)$, where ${\mathbf R}_{M,L+1}\coloneqq{\mathbf R}_{2N-L,L+1}(1:M,1:L+1)$ is trapezoidal with full rank $M$. Thus, the parameter $M$ can be estimated as the numerical rank of ${\mathbf R}_{2N-L,L+1}$.
Using (\ref{QR}), the matrix pencil in (\ref{1.6}) can be rewritten as 
\begin{eqnarray*}
&&  z ({\mathbf Q}_{2N-L} {\mathbf R}_{2N-L,L+1} {\mathbf P}_{L+1}^{T})(1:2N-L,1:L)
\\
&& \hspace{5cm} - ({\mathbf Q}_{2N-L} {\mathbf R}_{2N-L,L+1} {\mathbf P}_{L+1}^{T})(1:2N-L,2:L+1), \end{eqnarray*}
which is equivalent to 
$$ z ({\mathbf R}_{2N-L,L+1} {\mathbf P}_{L+1}^{T})(1:2N-L,1:L)
- ({\mathbf R}_{2N-L,L+1} {\mathbf P}_{L+1}^{T})(1:2N-L,2:L+1). $$
Using the structure of ${\mathbf R}_{2N-L,L+1}$, we  can further simplify to
$$ z ({\mathbf R}_{2N-L,L+1} {\mathbf P}_{L+1}^{T})(1:M,1:L)
- ({\mathbf R}_{2N-L,L+1} {\mathbf P}_{L+1}^{T})(1:M,2:L+1), $$
see \cite{PT2013}. 
We summarize the MPM as given in \cite{PT2013} in Algorithm \ref{algMPM}.

\begin{algorithm}[ht]\caption{Matrix Pencil Method (MPM)}
\label{algMPM}
\small{
\textbf{Input:} ${\mathbf f} = (f_{k})_{k=0}^{2N-1} = \big(f(k)\big)_{k=0}^{2N-1}$ (equidistant sampling values of $f$ in (\ref{1.1}))\\
\phantom{\textbf{Input:}} $L$, upper bound of $M$ with $M \le L < N$\\
\phantom{\textbf{Input:}} accuracy $\epsilon >0$

\begin{enumerate}
\item Build the Hankel matrix ${\mathbf H}_{2N-L, L+1} \coloneqq \left( f_{\ell+m} \right)_{\ell,m=0}^{2N-L-1,L}$ and 
compute the QR decomposition (with column pivot)  of ${\mathbf H}_{2N-L, L+1}$ in (\ref{QR}).
Determine the numerical rank $M$ of ${\mathbf R}_{2N-L,L+1}$ by taking the smallest $M$ such that 
$$ {\mathbf R}_{2N-L,L+1}(M+1,M+1) < \epsilon \, {\mathbf R}_{2N-L,L+1}(1,1). $$
\item
Form
\begin{eqnarray*}
 {\mathbf S}_{M,L}(0) &\coloneqq& ({\mathbf R}_{2N-L,L+1} {\mathbf P}_{L+1}^{T})(1:M,1:L),\\
  {\mathbf S}_{M,L}(1) &\coloneqq& ({\mathbf R}_{2N-L,L+1} {\mathbf P}_{L+1}^{T})(1:M,2:L+1)
\end{eqnarray*}
and determine the vector of eigenvalues ${\mathbf z} =(z_{1}, \ldots , z_{M})^{T}$ of the generalized eigenvalue problem 
$z {\mathbf S}_{M,L}(0) - {\mathbf S}_{M,L}(1)$, or equivalently, the vector of eigenvalues of 
$\left(  {\mathbf S}_{M,L}(0)^{T} \right)^{\dagger}  {\mathbf S}_{M,L}(1)^{T}, $
where $\left(  {\mathbf S}_{M,L}(0)^{T} \right)^{\dagger}$ denotes the Moore-Penrose inverse of  ${\mathbf S}_{M,L}(0)^{T} $.
\item Compute ${\bgamma} = (\gamma_{j})_{j=1}^{M}$  as the least squares solution of the linear system 
$ {\mathbf V}_{2N,M} ({\mathbf z}) \, \bgamma = {\mathbf f} $ with ${\mathbf V}_{2N,M} ({\mathbf z})$ as defined in (\ref{V}).
\end{enumerate}

\noindent
\textbf{Output:} $M \in {\mathbb N}$,  $z_{j}, \, \gamma_{j} \in {\mathbb C}$, $j=1, \ldots , M$.}
\end{algorithm}

\begin{remark}\label{remMPM}
In \cite{PT2013}, the authors propose to apply the matrices 
\begin{eqnarray*}{\mathbf S}'_{M,L}(0) &=& {\mathbf D}_{M}^{-1} ({\mathbf R}_{2N-L,L+1} {\mathbf P}_{L+1}^{T})(1:M,1:L), \\
{\mathbf S}'_{M,L}(1) &=& {\mathbf D}_{M}^{-1} ({\mathbf R}_{2N-L,L+1} {\mathbf P}_{L+1}^{T})(1:M,2:L+1)
\end{eqnarray*}
 in the second step of the algorithm for further preconditioning, where ${\mathbf D}_{M}$ contains the first $M$ diagonal entries of ${\mathbf R}_{2N-L,L+1}$. \end{remark}

The arithmetical complexity of the QR decomposition of ${\mathbf H}_{2N-L,L+1}$ in step 1 of Algorithm \ref{algMPM} requires ${\mathcal O}\big((2N-L)L^{2}\big)$ operations.
Step 2 involves besides the matrix inversion and matrix multiplication the solution of the eigenvalue problem for an $M \times M$ matrix with ${\mathcal O}(M^{3})$ operations.
Thus, we have overall computational costs of ${\mathcal O}(NL^{2})$ and for $L \approx N$ we require ${\mathcal O}(N^{3})$ operations.

\medskip

\subsection{ESPRIT Algorithm}
The ESPRIT algorithm derived in \cite{PT2013} employs the  SVD  of ${\mathbf H}_{2N-L,L+1}$ in (\ref{1.5}),
\begin{equation}\label{SVD}
 {\mathbf H}_{2N-L,L+1} = {\mathbf U}_{2N-L} \, {\mathbf D}_{2N-L,L+1} {\mathbf W}_{L+1}, 
\end{equation} 
where  ${\mathbf U}_{2N-L}$ and ${\mathbf W}_{L+1}$  are unitary  square matrices  and, in the case of unperturbed data,
$$ {\mathbf D}_{2N-L,L+1} = \left( \begin{array}{cc} \mathrm{diag} \left(  (\sigma_{j})_{j=1}^{M} \right) & {\mathbf 0} \\
{\mathbf 0} & {\mathbf 0} \end{array} \right) \in {\mathbb R}^{2N-L,L+1}, $$
where $\sigma_{1} \ge  \ldots  \ge \sigma_{M}$ denote the positive singular values of ${\mathbf H}_{2N-L,L+1}$.
Thus (\ref{SVD}) can be simplified to 
\begin{equation}\label{1.7}
{\mathbf H}_{2N-L,L+1} = {\mathbf U}_{2N-L} \, {\mathbf D}_{2N-L,M} {\mathbf W}_{M,L+1}, 
\end{equation}
where the (last) $L+1-M$ zero-columns in ${\mathbf D}_{2N-L,L+1}$ and the last $L+1-M$ rows in ${\mathbf W}_{L+1}$ are removed.
Similarly as before, the SVDs of ${\mathbf H}_{2N-L,L}(0)$ and ${\mathbf H}_{2N-L,L}(1)$ can be directly  derived  from (\ref{1.7})  and the matrix pencil in (\ref{1.6}) now reads
\begin{eqnarray*}
 & & z \left(  {\mathbf U}_{2N-L} \, {\mathbf D}_{2N-L,M} {\mathbf W}_{M,L+1} \right)(1:2N-L,1:L)  \\
 & &  \hspace{5cm} -  \left(  {\mathbf U}_{2N-L} \, {\mathbf D}_{2N-L,M} {\mathbf W}_{M,L+1} \right)(1:2N-L,2:L+1),
\end{eqnarray*}
or equivalently
$$ z \left(   {\mathbf D}_{2N-L,M} {\mathbf W}_{M,L+1} \right)(1:2N-L,1:L) -  \left(  {\mathbf D}_{2N-L,M} {\mathbf W}_{M,L+1} \right)(1:2N-L,2:L+1). $$
Multiplication from the left with $\left( \mathrm{diag}( \sigma_{j}^{-1} )_{j=1}^{M} , {\mathbf 0} \right) \in {\mathbb R}^{M \times 2N-L}$  yields
the matrix pencil
$$ z \,  {\mathbf W}_{M,L+1}(1:M,1:L) - {\mathbf W}_{M,L+1}(1:M,2:L+1). $$
Again the wanted parameters $z_{j}$ are generalized eigenvalues of this matrix pencil problem.
The ``rotational invariance'' that inspires the name of the  ESPRIT method is best observed in (\ref{extra})--(\ref{extra1}), where the factorizations of ${\mathbf H}_{2N-L,L}(0)$ and ${\mathbf H}_{2N-L,L}(1)$ differ by the diagonal 
 matrix $\mathrm{diag}  \left(z_{j} \right)_{j=1}^{M}$, which are rotation factors if $|z_{j}|=1$, i.e., if $z_{j}= {\mathrm e}^{{\mathrm i} \phi_{j}}$ with frequency parameters $\phi_{j} \in {\mathbb R}$. The ESPRIT algorithm is summarized in Algorithm \ref{esprit_algo}.
 The numerical effort of the ESPRIT algorithm is comparable to that of MPM and is governed by the SVD of ${\mathbf H}_{2N-L,L+1}$ with $\mathcal{O}(N\, L^{2})$ operations.  Note that the computational cost of ESPRIT can be improved by using a partial SVD and fast
Hankel matrix-vector multiplications, see \cite{PT15,HMT2011}. 
However, in this case one needs some prior knowledge about the wanted rank $M$ of ${\mathbf H}_{2N-L,L+1}$, which is not given here.
 
\begin{algorithm}[t]\caption{ESPRIT algorithm}\label{esprit_algo}
\label{algESPRIT}
\small{
\textbf{Input:} ${\mathbf f} = (f_{k})_{k=0}^{2N-1} = \big(f(k)\big)_{k=0}^{2N-1}$ (equidistant sampling values of $f$ in (\ref{1.1}))\\
\phantom{\textbf{Input:}} $L$, upper bound  with $M \le L < N$\\
\phantom{\textbf{Input:}} accuracy $\epsilon >0$

\begin{enumerate}
\item Build the Hankel matrix ${\mathbf H}_{2N-L, L+1} \coloneqq \left( f_{\ell+m} \right)_{\ell,m=0}^{2N-L-1,L}$ and 
compute  the SVD  of ${\mathbf H}_{2N-L, L+1}$ in (\ref{SVD}).
Determine the numerical rank $M$ of ${\mathbf H}_{2N-L, L+1}$ by taking the smallest $M$ such that 
$ \sigma_{M+1} < \epsilon \, \sigma_{1}$, where $\sigma_{j}$ are the ordered diagonal entries of ${\mathbf D}_{2N-L,L+1}$ with $\sigma_{1} \ge \sigma_{2} \ge \ldots \geq \sigma_{L+1}$.
\item Form
 ${\mathbf W}_{M,L}(0) \coloneqq {\mathbf W}_{L+1}(1\!:\!M,1\!:\!L)$, 
  ${\mathbf W}_{M,L}(1) \coloneqq {\mathbf W}_{L+1}(1\!:\!M,2\!:\!L+1)$,
and determine the vector of eigenvalues ${\mathbf z} =(z_{1}, \ldots , z_{M})^{T}$ of 
$\left(  {\mathbf W}_{M,L}(0)^{T} \right)^{\dagger}  {\mathbf W}_{M,L}(1)^{T}, $
where $\left(  {\mathbf W}_{M,L}(0)^{T} \right)^{\dagger}$ denotes the Moore-Penrose inverse of  ${\mathbf W}_{M,L}(0)^{T} $.
\item Compute ${\bgamma} = (\gamma_{j})_{j=1}^{M}$  as the least squares solution of the linear system 
$ {\mathbf V}_{2N,M} ({\mathbf z}) \, \bgamma = {\mathbf f}. $
\end{enumerate}

\noindent
\textbf{Output:} $M \in {\mathbb N}$,  $z_{j}, \, \gamma_{j} \in {\mathbb C}$, $j=1, \ldots , M$.}
\end{algorithm}

\section{ESPIRA Algorithm}

\subsection{The Basic ESPIRA Algorithm}

We want  to propose a different algorithm  for parameter estimation  from the given samples $f_{\ell}$, $\ell=0, \ldots , 2N-1$, of $f$   in (\ref{1.1}), the \textbf{E}stimation of \textbf{S}ignal \textbf{P}arameters  by \textbf{I}terative \textbf{R}ational \textbf{A}pproximation (ESPIRA).
Our approach uses the close connection between the parameter estimation problem  and rational approximation.
For the vector ${\mathbf f} = (f_{\ell})_{\ell=0}^{2N-1}$ of given equidistant samples we compute the discrete Fourier transform (DFT)  $\hat{\mathbf f} = (\hat{f}_{k})_{k=0}^{2N-1}$, 
\begin{equation}\label{1.8}
\hat{f}_k\coloneqq\sum\limits_{\ell=0}^{2N-1} f_\ell  \, \omega_{2N}^{k\ell}, \qquad k=0,1,\ldots, 2N-1,
\end{equation}
where $ \omega_{2N} \coloneqq {\mathrm e}^{-2\pi {\mathrm i}/(2N)}$.
Then we  observe that
\begin{eqnarray} \nonumber
\hat{f}_k&=&  \sum_{\ell=0}^{2N-1} \left( \sum_{j=1}^{M} \gamma_{j} \, z_{j}^{\ell} \right) \, \omega_{2N}^{k\ell}= \sum_{j=1}^{M} \gamma_{j}  \sum_{\ell=0}^{2N-1} z_{j}^{\ell}  \, \omega_{2N}^{k\ell}\\
\label{1.9}
&=& \sum_{j=1}^{M} \gamma_{j} \left( \frac{1- (\omega_{2N}^{k}z_{j})^{2N}}{1-\omega_{2N}^{k}z_{j}} \right) = \omega_{2N}^{-k} \, \sum_{j=1}^{M} \gamma_{j} \left( \frac{1- z_{j}^{2N}}{ \omega_{2N}^{-k} - z_{j}} \right).
\end{eqnarray}
The representation of $\hat{f}_{k}$ in (\ref{1.9}) is  well-defined if all knots $z_{j}$, $j=1, \ldots , M$, satisfy  $z_{j}^{2N} \neq 1$. 
For $z_{j} = \omega_{2N}^{-k} $, the rule of L'Hospital yields 
\begin{equation}\label{lop} 
\lim_{ z_{j} \to \omega_{2N}^{-k} } \frac{\omega_{2N}^{-k} \, (1- z_{j}^{2N})}{\omega_{2N}^{-k} -z_{j}} = 2N. 
\end{equation}
Let us first assume that the wanted reconstruction parameters  $z_j$ in (\ref{1.1}) satisfy
\begin{equation}\label{1.91} 
z_{j}^{2N} \neq 1, \qquad j=1, \ldots , M.
\end{equation}
The case that some (or all)  $z_{j}$ satisfy $z_{j}^{2N} = 1$ is treated later separately.
With (\ref{1.91}) we obtain from (\ref{1.9}) 
\begin{equation} \label{1.11}
\omega_{2N}^{k} \, \hat{f}_{k} = \sum_{j=1}^{M} \gamma_{j} \frac{1-z_{j}^{2N}}{\omega_{2N}^{-k}-z_{j}}, \qquad k=0, \ldots , 2N-1. 
\end{equation}

Now, the problem of parameter estimation can be rephrased as  a rational interpolation  problem. Setting $a_{j} \coloneqq \gamma_{j} \, (1 -z_{j}^{2N})$, we consider the rational function
\begin{equation}\label{rm}
 r_{M}(z) \coloneqq \sum_{j=1}^{M} \frac{a_{j}}{z-z_{j}} 
\end{equation}
of type $(M-1,M)$ and use the interpolation conditions 
\begin{equation}\label{3.int}
 r_{M}(\omega_{2N}^{-k}) = \sum_{j=1}^{M} \gamma_{j} \frac{1-z_{j}^{2N}}{\omega_{2N}^{-k}-z_{j}} = \omega_{2N}^{k} \, \hat{f}_{k}, \qquad  k=0, \ldots , 2N-1, 
 \end{equation}
to determine $r_{M}(z)$ and thus the parameters $M$, $a_{j}$, $z_{j}$, $j=1, \ldots , M$, of its fractional decomposition.

Note that  the structure of $\omega_{2N}^{k} \, \hat{f}_{k}$ in (\ref{1.11}) implies that the data  cannot be interpolated by a rational function of smaller type than $(M-1,M)$. Therefore, $r_{M}(z)$ of type $(M-1,M)$ is uniquely determined by the given interpolation conditions, since $N > M$.

We obtain the pseudocode for the ESPIRA-I algorithm in Algorithm \ref{alg1}.
This algorithm is based  on the AAA algorithm in \cite{NST18}. All steps of Algorithm \ref{alg1} will be studied in more detail in the next subsection. Since computation of DFT requires ${\mathcal O}(N\, \log N)$ operations and the complexity of Algorithm \ref{alg2} (the AAA algorithm) is  ${\mathcal O}(N\, M^{3})$, as Algorithm \ref{alg2} involves SVDs of Loewner matrices of dimension $(N-j) \times j$ for $j=1, \ldots, M+1$, the overall computational costs of Algorithm \ref{alg1} is ${\mathcal O}(N\, (M^{3}+\log N))$.


\begin{algorithm}[ht]\caption{ESPIRA-I}
\label{alg1}
\small{
\textbf{Input:} ${\mathbf f} = (f_{k})_{k=0}^{2N-1} = \big(f(k)\big)_{k=0}^{2N-1}$ (equidistant sampling values of $f$ in (\ref{1.1}))\\
\phantom{\textbf{Input:}} ${tol}>0$ tolerance for the approximation error

\begin{enumerate}
\item Compute the DFT-vector $\hat{\mathbf f}= (\hat{f}_{k})_{k=0}^{2N-1}$ with $\hat{f}_{k} = \sum_{j=0}^{2N-1} f_{j} \, \omega_{2N}^{kj}$ of ${\mathbf f}$ (with $\omega_{2N} \coloneqq {\mathrm e}^{-2\pi {\mathrm i}/2N}$).
\item Use Algorithm \ref{alg2} to compute a rational function $r_{M}(z)$ of (smallest possible) type $(M-1,M)$ $(M \le N)$, such that 
 $$ |r_{M}(\omega_{2N}^{-k}) -\omega_{2N}^{k} \, \hat{f}_{k}| < tol, \qquad k=0, \ldots , 2N-1.$$ 
\item Use Algorithm \ref{alg3} to compute a fractional decomposition representation of $r_{M}(z)$, 
$$ r_{M}(z) = \sum_{j=1}^{M} \frac{a_{j}}{z - z_{j}}, $$
i.e., compute $a_{j}$, $z_{j}$, $j=1, \ldots , M$
and set  $\gamma_{j} \coloneqq \frac{a_{j}}{1-z_{j}^{2N}}$, $j=1, \ldots, M$.
\end{enumerate}

\noindent
\textbf{Output:} $M$, $\gamma_{j}$, $z_{j}$, for $j=1, \ldots , M$ (all parameters of $f$).}
\end{algorithm}

\begin{remark}
The idea of ESPIRA can be seen as a discrete analog of the approaches in \cite{DP21} and \cite{PPD21}, where a reconstruction of exponential sums via rational interpolation has been proposed, using a set of given Fourier coefficients of $f$ corresponding to a Fourier expansion in a given finite interval. ESPIRA has now the advantage that it uses only discrete Fourier coefficients which can be directly computed from a vector of given equidistant function values of $f$.
\end{remark}

\subsection{The AAA Algorithm for rational approximation}
\label{aaa-algorithm}

We will employ the recently proposed AAA algorithm \cite{NST18} for rational interpolation, which possesses  a stable numerical performance  due to a barycentric  representation  of the rational interpolant using an adaptively chosen subset of interpolation points.

The AAA algorithm with the implementation in \cite{NST18} can be directly  applied for our  purpose. It  computes  rational functions
 $r_{J}(z)$ of type $(J,J)$ for $J=1, 2,\ldots $ thereby including  $J+1$ interpolation conditions  directly, while the remaining  interpolation  data are approximated  by solving a linearized least squares problem. Since the type of the rational interpolating function $r_{M}$ is in our case $(M-1,M)$ by (\ref{1.11}), the AAA algorithm will provide the correct result.
In case of perturbed data we enforce the condition that  $r_{M}$ is of type $(M-1,M)$ in step 3 of Algorithm \ref{alg1} by fixing the structure of $r_{M}(z)$.

We will explain the AAA algorithm for our recovery problem at hand in detail, 
since the occurring structured matrices in this algorithm play a key role when we compare ESPIRA with  MPM and  ESPRIT.
We refer to \cite{NST18} for the original AAA algorithm.
\medskip

Let $I \coloneqq \{0, \ldots , 2N-1\}$ be the index set, $\Gamma =\{ \omega_{2N}^{-k}: \,  k \in I \}$  the set of \textit{support points}, and $\{ \omega_{2N}^{k} \hat{f}_{k}: \,  {k \in I} \}$ the corresponding set of known function values.
Our task is to find the rational function $r_{M}(z)$ of type $(M, M)$ such that $r_{M}(\omega_{2N}^{-k}) = \omega_{2N}^{k} \hat{f}_{k}$ for $k \in I$, where $N$ is an upper bound of the unknown degree $M$.
(As we will show in Section \ref{sec:conv}, the type of the resulting function $r_{M}(z)$ will indeed be $(M-1, M)$ for exact data.)

At the iteration step $J \ge 1$, we proceed as follows to compute a rational function $r_{J-1}$ of type $(J-1,J-1)$, see also \cite{NST18,DP21,PPD21}.
For $J>1$,  we have already a given partition $S_{J} \cup \Gamma_{J} = I$ of the index set from the previous iteration step, with $\# S_{J} = J$ and $\# \Gamma_{J} = 2N-J$,  which has been found by an adaptive (greedy) procedure.
For $J=1$, we start with the initialization $S_{1} \coloneqq \{ \argmax_{\ell\in I} |\hat{f}_{\ell}| \}$ and $\Gamma_{1} \coloneqq I \setminus S_{1}$.
 We require now that  $r_{J-1}$  satisfies the interpolation conditions 
$r_{J-1}(\omega_{2N}^{-k}) = \omega_{2N}^{k} \, \hat{f}_{k}$ for $k \in S_{J}$,  while $r_{J-1}(\omega_{2N}^{-k})$ approximates the data $\omega_{2N}^{k} \, \hat{f}_{k}$ for $k \in \Gamma_{J}$.
The rational function $r_{J-1}(z)$ will be constructed in a  barycentric form $r_{J-1}(z) \coloneqq \frac{\tilde{p}_{J-1}(z)}{\tilde{q}_{J-1}(z)}$ with 
\begin{equation}\label{bar-form}
\tilde{p}_{J-1}(z) \coloneqq \sum_{k \in S_{J}} \frac{w_k \, (\omega_{2N}^{k} \, \hat{f}_{k})}{z- \omega_{2N}^{-k}}, \qquad  \tilde{q}_{J-1}(z) \coloneqq \sum_{k \in S_{J}} \frac{w_k}{z- \omega_{2N}^{-k}}, 
\end{equation}
where $w_k \in {\mathbb C}$, $k \in S_{J}$, are weights. The representation (\ref{bar-form}) already implies that  the interpolation conditions $r_{J-1}(\omega_{2N}^{-k})  = \omega_{2N}^{k} \, \hat{f}_{k}$ are satisfied for $w_{k} \neq 0$, $k \in S_{J}$.
This can be simply  observed from (\ref{bar-form}) by multiplying $\tilde{p}_{J-1}(z)$ and $\tilde{q}_{J-1}(z)$ with $\prod\limits_{k \in S_{J}}(z-\omega_{2N}^{-k})$.

The weight vector ${\mathbf w} \coloneqq \left( w_k \right)_{k \in S_{J}}$ is now chosen such that $r_{J-1}(z)$ approximates the remaining data and $\| {\mathbf w} \|_2^2 = \sum_{k \in S_{J}} w_k^2 = 1$.
To compute ${\mathbf w}$, we consider the restricted least-squares problem obtained by linearizing the interpolation conditions for $\ell \in \Gamma_{J}$,  
\begin{equation}\label{mini}
\min_{\mathbf w} \sum_{\ell \in \Gamma_{J}} \left|(\omega_{2N}^{\ell} \, \hat{f}_{\ell}) \, \tilde{q}_{J-1}(\omega_{2N}^{-\ell})-\tilde{p}_{J-1}(\omega_{2N}^{-\ell})\right|^2, \quad \textrm{such~that} \quad  \| {\mathbf w} \|_2^2 = 1.
\end{equation}
We define the Loewner matrix
$$ 
{\mathbf L}_{2N-J,J} \coloneqq \left( \frac{\omega_{2N}^{\ell} \, \hat{f}_{\ell} - \omega_{2N}^{k} \, \hat{f}_{k}}{\omega_{2N}^{-\ell}- \omega_{2N}^{-k}} \right)_{\ell \in \Gamma_{J}, k \in S_{J}},
$$
and rewrite the term in (\ref{mini}) as
\begin{eqnarray*}
\sum_{\ell \in \Gamma_{J}} \left| (\omega_{2N}^{\ell} \, \hat{f}_{\ell}) \, \tilde{q}_{J-1}(\omega_{2N}^{-\ell})-\tilde{p}_{J-1}(\omega_{2N}^{-\ell})\right|^2 &=& \sum_{\ell \in \Gamma_{J}} \left|       {{\mathbf w}}^T \, \left( 
\frac{\omega_{2N}^{\ell} \, \hat{f}_{\ell}- \omega_{2N}^{k} \, \hat{f}_{k}}{\omega_{2N}^{-\ell}-\omega_{2N}^{-k}} \right)_{k \in S_{J}} \right|^2 \\
&=& \| {\mathbf L}_{2N-J,J} {{\mathbf w}} \|_2^2. 
\end{eqnarray*}
Thus,  the minimization problem in (\ref{mini}) takes the form
\begin{equation}\label{mini1}
\min_{\|{\mathbf w}\|_{2} = 1 }\| {\mathbf L}_{2N-J,J} {{\mathbf w}} \|_2^2 ,
\end{equation}
and the solution vector ${\mathbf w} \in {\mathbb C}^{J}$ is the right singular vector  corresponding to the smallest singular value of ${\mathbf L}_{2N-J,J}$.
Having determined the weight vector ${\mathbf w}$, the rational function $r_{J-1}$ is completely fixed by (\ref{bar-form}).
Finally, we  consider the errors $|r_{J-1}(\omega_{2N}^{-\ell}) - \omega_{2N}^{\ell} \, \hat{f}_{\ell}|$ for all $\ell \in \Gamma_{J}$, where we do not interpolate.
The algorithm terminates if $\max_{\ell \in \Gamma_{J}} | r_{J-1}(\omega_{2N}^{-\ell}) - \omega_{2N}^{\ell} \, \hat{f}_{\ell}| < \epsilon$ for a predetermined bound $\epsilon$ or if $J$ reaches a predetermined maximal degree. Otherwise, we find the updated index set 
$$ S_{J+1} \coloneqq S_{J} \cup  \argmax_{\ell \in \Gamma_{J}} | r_{J-1}(\omega_{2N}^{-\ell}) - \omega_{2N}^{\ell} \, \hat{f}_{\ell}| $$
and update $\Gamma_{J+1} = I \setminus S_{J+1}$.
\medskip

\begin{algorithm}[ht]\caption{Iterative rational approximation by  AAA algorithm \cite{NST18}} 
\label{alg2} 
\small{
\textbf{Input:} 
$\hat{\mathbf f} \in {\mathbb C}^{2N}$, DFT of ${\mathbf f}$ \\
\phantom{\textbf{Input:}} ${tol}>0$ tolerance for the approximation error \\ 
\phantom{\textbf{Input:}} $jmax \in \nn$ with $\textit{jmax} < N $ maximal order of polynomials in the rational function

\noindent
\textbf{Initialization:}\\
Set the ordered index set  ${\mathbf \Gamma} \coloneqq \left( k \right)_{k=0}^{2N-1}$.
Compute ${\mathbf g}_{{\mathbf \Gamma}} = (g_{k})_{k=0}^{2N-1}\coloneqq \left( \omega_{2N}^{k} \, \hat{f}_{k} \right)_{k=0}^{2N-1}$.\\

\noindent
\textbf{Main Loop:}

\noindent
for $j=1:$ \textit{jmax}
\begin{enumerate}
\item If $j=1$, choose ${\mathbf S}\coloneqq (k_{1})$, $\mathbf{g}_{\mathbf{S}} \coloneqq (g_{k_{1}})$, 
where $k_{1} \coloneqq \argmax_{\ell \in \mathbf{\Gamma}} |g_{\ell}|$; update $\mathbf{\Gamma}$ and $\mathbf{g}_{{\mathbf \Gamma}}$ by deleting $k_{1}$ in $\mathbf{\Gamma}$ and $g_{k_{1}}$ in $\mathbf{g}_{{\mathbf \Gamma}}$.\\
 If $ j>1$, compute $k \coloneqq \argmax_{\ell \in \mathbf{\Gamma}} | r_{\ell} - g_{\ell}|$; update $\mathbf{S}$, $\mathbf{g}_{\mathbf{S}}$, $\mathbf{\Gamma}$ and $\mathbf{g}_{{\mathbf \Gamma}}$ by adding $k$ to $\mathbf{S}$ and deleting  $k$ in $\mathbf{\Gamma}$, adding $g_{k}$ to $\mathbf{g}_{\mathbf{S}}$ and deleting it in $\mathbf{g}_{{\mathbf \Gamma}}$.
\item Build  $\mathbf{C}_{2N-j,j}\!\coloneqq\!\left( \frac{1}{\omega_{2N}^{-\ell}-\omega_{2N}^{-k}} \right)_{\ell \in \mathbf{\Gamma}, k \in \mathbf{S}}$, ${\mathbf L}_{2N-j,j}\!\coloneqq\! \left( \frac{g_{\ell} - g_{k}}{\omega_{2N}^{-\ell}-\omega_{2N}^{-k}} \right)_{\ell \in \mathbf{\Gamma}, k \in \mathbf{S}}$.
\item Compute the normalized singular vector ${\mathbf w}$
 corresponding to the smallest singular value of $\mathbf{L}_{2N-j,j}$.
\item Compute $\mathbf{p}\coloneqq\mathbf{C}_{2N-j,j} ({\mathbf w}.\ast \mathbf{g}_{\mathbf{S}})$, $\mathbf{q} \coloneqq \mathbf{C}_{2N-j,j} {\mathbf w}$ and $\mathbf{r}\coloneqq \mathbf{p}./\mathbf{q} \in \cc^{2N-j}$,  where $.*$ denotes componentwise multiplication and $./$ componentwise division.
\item If $\| \mathbf{r}-\mathbf{g}_{{\mathbf \Gamma}} \|_\infty<$ \textit{tol}  then set $M\coloneqq j-1$ and stop.
\end{enumerate}
end (for)\\
\textbf{Output: } \\ 
$M=j-1$, where $(M,M)$ is the type of the rational function $r_{M}$ \\
$\mathbf{S} \in {\mathbb Z}^{M+1}$ determining the index set $S_{M+1}$ where $r_{M}$ satisfies the interpolation conditions \\
  $\mathbf{g}_{\mathbf{S}}= (g_{k})_{k \in S_{M+1}} \in \cc^{M+1}$ is the vector of the corresponding interpolation values\\
   ${\mathbf w}=(w_k)_{k \in S_{M+1}} \in \cc^{M+1}$ is the weight vector. }

\end{algorithm}

Algorithm \ref{alg2} provides the rational function $r_M(z)$ in a barycentric form 
$r_{M}(z)  = \frac{\tilde{p}_{M}(z)}{\tilde{q}_{M}(z)}$ with 
\begin{equation}\label{bar-form-1}
\tilde{p}_{M}(z) \coloneqq \sum_{k \in S_{M+1}} \frac{w_k \, (\omega_{2N}^{k} \, \hat{f}_{k})}{z-\omega_{2N}^{-k}}, \qquad  \tilde{q}_{M}(z) \coloneqq \sum_{k \in S_{M+1}} \frac{w_k}{z-\omega_{2N}^{-k}},
\end{equation}
which are determined by the output parameters of this algorithm. Note that it is important to take the occurring index sets and data sets in Algorithm  \ref{alg2}  as ordered sets, therefore they are given as vectors ${\mathbf S}$, ${\mathbf \Gamma}$, ${\mathbf g}_{\mathbf S}$ and ${\mathbf g}_{\mathbf \Gamma}$, as in the original algorithm, \cite{NST18}.

\subsection{Partial fraction decomposition}

In order to rewrite $r_M(z)$ in (\ref{bar-form-1}) in the form  of a partial fraction decomposition,
\begin{equation}\label{rn1} 
r_{M}(z) = \sum_{j=1}^{M} \frac{a_{j}}{z-z_{j}},
\end{equation}
we need to determine $a_1, \ldots , a_M$ and $z_1, \ldots , z_M$ from the output of Algorithm \ref{alg2}.

The zeros of the denominator $\tilde{q}_{M}(z)$ are the poles $z_j$ of $r_M(z)$ and can be computed by solving an $(M+2)\times (M+2)$ generalized eigenvalue problem (see \cite{NST18} or \cite{PPD21}), that has for $S_{M+1}=\{k_{1}, \ldots,  k_{M+1} \}$ the form 
\begin{equation}\label{eig} \left( \begin{array}{ccccc}
0 & w_{k_{1}} & w_{k_{2}} & \ldots & w_{k_{M+1}} \\
1 & \omega_{2N}^{-k_{1}} &   &     & \\
1 & & \omega_{2N}^{-k_{2}} & & \\
\vdots & & & \ddots & \\
1 & & & & \omega_{2N}^{-k_{M+1}} \end{array} \right) \, {\mathbf v}_{z}= z \left( \begin{array}{ccccc}
0 & & & & \\
& 1 & & & \\
& & 1 & & \\
& & & \ddots & \\
& & & & 1 \end{array} \right) \, {\mathbf v}_{z}.
\end{equation}
Two eigenvalues of this generalized eigenvalue problem are infinite and the other $M$ eigenvalues are the wanted zeros $z_j$ of $\tilde{q}_{M}(z)$ (see \cite{
Klein,NST18,PPD21} for more detailed explanation). 
We apply Algorithm \ref{alg3} to the output of Algorithm \ref{alg2}.

\begin{algorithm}[ht]\caption{Reconstruction of parameters $a_j$ and $z_j$ of partial fraction representation}
 \label{alg3}
\small{
\textbf{Input:} $\hat{\mathbf f} \in {\mathbb C}^{2N}$ DFT of ${\mathbf f}$ \\
\phantom{\textbf{Input:}} $\mathbf{S}\in \cc^{M+1}$,
  ${\mathbf{g}}_{\mathbf{S}} \in \cc^{M+1}$,
   ${\mathbf w} \in \cc^{M+1}$ the output vectors of Algorithm \ref{alg2}

\begin{enumerate}
\item Build the matrices in (\ref{eig}) and solve this eigenvalue problem to find the  vector ${\mathbf z}^{T}=(z_1, \ldots, z_M)^{T}$ of the $M$ finite eigenvalues;

\item Build the matrix $\mathbf{C}_{2N,M}=\left(\frac{1}{\omega_{2N}^{-k}-z_j} \right)_{k =0,  j=1}^{2N-1,M}\in \cc^{2N\times M} $ and compute the least squares solution of  the linear system
$$
\mathbf{C}_{2N,M} \mathbf{a}= \left( \omega_{2N}^{k} \hat{f}_{k} \right)_{k=0}^{2N-1} .
$$
\end{enumerate}
\textbf{Output: } Parameter vectors ${\mathbf z}= (z_j)_{j=1}^{M}$, ${\mathbf a} = (a_j)_{j=1}^{M}$ determining $r_M(z)$ in (\ref{rn1}). }
\end{algorithm}

\begin{remark} \label{rem3}
1.  Note that the rational function $r_{M}= \frac{p_{M}}{q_{M}}$ obtained by Algorithm \ref{alg2} is of type $(M,M)$, where by construction, the denominator polynomial $q_{M}$ has exact degree $M$, and the numerator polynomial  $p_{M}$ has at most degree $M$. 
By fixing the Cauchy matrix ${\mathbf C}_{2N,M}$ as given in step 2 of Algorithm \ref{alg3}, we incorporate the wanted type $(M-1,M)$ of the rational interpolant $r_{M}$.  This step provides the parameters $a_{j}$ that are related to   $\gamma_{j}$ in (\ref{1.1}) by $a_{j} = \gamma_{j}(1-z_{j}^{2N}) $ for $z_{j}^{2N} \neq 1$.

Instead of using Algorithm 4, we  could also apply a slightly modified AAA algorithm, where a side condition is incorporated in each iteration step that forces a rational function $r_{J-1}$ of type $(J-2,J-1)$, see \cite{NST18}, Section 9.
Such a modified AAA algorithm has also been  used in \cite{DP21} and \cite{PPD21}. In our setting, the modification reads as follows: At each iteration step $J>1$ of Algorithm \ref{alg2}, 
we  compute the vector of weights $\mathbf{w}$ by solving the least-squares problem (\ref{mini}) or (\ref{mini1}) with one more side condition,
\begin{equation}\label{side}
{\mathbf w}^T  \,  \left(\omega_{2N}^{k} \hat{f}_k\right)_{k\in S_{J}}= \sum_{k \in S_J} w_k \, (\omega_{2N}^{k} \hat{f}_k) = 0,
\end{equation}
ensuring that $r_{J-1}$ is of type $(J-2, J-1)$. 
Practically, to compute the solution vector ${\mathbf w} = {\mathbf w}_{J}$  of  (\ref{mini1}) approximately, we compute the right (normalized) singular vectors ${\mathbf v}_J$ and  ${\mathbf v}_{J-1}$ of the matrix ${\mathbf L}_{2N-J,J}$ corresponding to the two smallest singular values $\sigma_J \leq \sigma_{J-1}$ of ${\mathbf L}_{2N-J,J}$ and take 
$$
{\mathbf w}_{J} \coloneqq \frac{\left( \left({\mathbf v}_{J-1}^{T}  \left(\omega_{2N}^{k} \hat{f}_k\right)_{k\in S_{J}} \right){\mathbf v}_J - \left({\mathbf v}_J^{T} \left(\omega_{2N}^{k} \hat{f}_k\right)_{k\in S_{J}}  \right){\mathbf v}_{J-1} \right)}{\sqrt{\left({\mathbf v}_J^{T}  \left(\omega_{2N}^{k} \hat{f}_k\right)_{k\in S_{J}}    \right)^{2}+\left({\mathbf v}_{J-1}^{T}   \left(\omega_{2N}^{k} \hat{f}_k\right)_{k\in S_{J}} \right)^{2}}} .
$$
Then ${\mathbf w}_{J}$ satisfies (\ref{side}), and we have $\|{\mathbf w}_{J}\|_{2} = 1$ as well as $\| {\mathbf L}_{2N-J,J} \, {\mathbf w}_{J} \|_{2} \le \sigma_{J-1}$.

We will show in Theorem \ref{AAAconv} that 
  in case of exact data a rational function $r_M$ constructed in Algorithm \ref{alg1} has already type $(M-1,M)$, therefore,  this modification of the AAA algorithm is not needed. 
 For noisy data and for function approximation, the modified AAA algorithm may lead to a different greedy choice of interpolation indices compared to Algorithm \ref{alg2}. However, our numerical experiments with the modified AAA algorithm did not yield improved recovery results compared  to the original AAA algorithm (Algorithm \ref{alg2}).

3. Instead of solving ${\mathbf C}_{2N,M} \, {\mathbf a} = \left( \omega_{2N}^{k} \hat{f}_{k} \right)_{k=0}^{2N-1}$  in the second step of Algorithm \ref{alg3}
we could also directly solve
${\mathbf V}_{2N,M} \, {\bgamma} = {\mathbf f}$ as in step 3 of Algorithm \ref{algMPM} and \ref{algESPRIT}.
Indeed, with the Fourier matrix ${\mathbf F}_{2N} = (\omega_{2N}^{jk})_{j,k=0}^{2N-1}$ it follows that the system
 ${\mathbf V}_{2N,M} \, {\bgamma} = {\mathbf f}$ is equivalent to 
 $$\textrm{diag}   \Big((\omega_{2N}^{k})_{k=0}^{2N-1}\Big) \, {\mathbf F}_{2N} \, {\mathbf V}_{2N,M} \, {\bgamma} = \textrm{diag}  \Big((\omega_{2N}^{k})_{k=0}^{2N-1}\Big) \, {\mathbf F}_{2N} \, {\mathbf f} = \textrm{diag}   \Big((\omega_{2N}^{k})_{k=0}^{2N-1}\Big)  \, \hat{\mathbf f}, $$ 
 and it can be simply verified that 
 $$\textrm{diag}   \Big((\omega_{2N}^{k})_{k=0}^{2N-1}\Big)  \, {\mathbf F}_{2N} \, {\mathbf V}_{2N,M} = {\mathbf C}_{2N,M} \, \textrm{diag}  \Big(( 1-z_{j}^{2N})_{j=1}^{M} \Big).$$
\end{remark}

\subsection{Recovery of parameters $z_{j} = \omega_{2N}^{k}$}
\label{secper}

Assume now that the function $f$ in (\ref{1.1}) also contains parameters $z_{j}$ with $z_{j}^{2N} = 1$.  Then we can write $f(t)$ as a sum
$f(t) = f_{1}(t)
+ f_{2}(t)$, where 
\begin{equation}\label{f1} f_{1}(t) = \sum_{j=1}^{M_{1}} \gamma_{j} \, z_{j}^{t}, \qquad \{z_{j}: \, j=1, \ldots , M_{1}\} \cap \{\omega_{2N}^{-k}: \, k=0, \ldots , 2N-1\} = \emptyset, 
\end{equation}
and 
\begin{equation}\label{f2}  f_{2}(t) = \sum_{j=M_{1}+1}^{M} \gamma_{j} \, z_{j}^{t}, \qquad \{z_{j}: \, j=M_{1}+1, \ldots , M\} \subset \{\omega_{2N}^{-k}: \, k=0, \ldots , 2N-1\}. \end{equation}
Let  $\hat{\mathbf f}_{1} = \big((\hat{f}_{1})_{k}\big)_{k=0}^{2N-1}$  and $\hat{\mathbf f}_{2} = \big((\hat{f}_{2})_{k}\big)_{k=0}^{2N-1}$ 
denote the DFTs of $\big(f_{1}(j)\big)_{j=0}^{2N-1}$ and $\big(f_{2}(j)\big)_{j=0}^{2N-1}$, respectively. Then it follows as in (\ref{1.11}) that 
$$ \omega_{2N}^{k} (\hat{f}_{1})_{k} = \sum_{j=1}^{M_{1}} \gamma_{j} \, \frac{1-z_{j}^{2N}}{\omega_{2N}^{-k} - z_{j}}, \qquad  k=0, \ldots , 2N-1,$$
while
\begin{equation}\label{f1c}
\omega_{2N}^{k} (\hat{f}_{2})_{k} = \sum_{j=M_{1}+1}^{M} \gamma_{j}  \sum_{\ell=0}^{2N-1} (z_{j}w_{2N}^{k})^{\ell}= \left\{ \begin{array}{ll}
 2N \, \gamma_{j} 
 & z_{j} = \omega_{2N}^{-k}, \\
 0 & \omega_{2N}^{-k} \not\in \{ z_{M_{1}+1}, \ldots , z_{M}\}.
 \end{array} \right.
\end{equation}
Hence, we have $ \omega_{2N}^{k} \, \hat{f}_{k} = \omega_{2N}^{k} (\hat{f}_{1})_{k}$ for all $k \in I$ with $\omega_{2N}^{-k} \not\in \{ z_{M_{1}+1}, \ldots , z_{M}\}$, while only $M-M_{1}$ DFT coefficients $\hat{f}_{k}$ do not possess the fractional structure.

We can therefore apply a modification of Algorithm \ref{alg2} as follows. 
We apply  the iteration process in Algorithm \ref{alg2} as before and iteratively find sets of indices $S_{J}$, $J=1,2,\ldots$.
As we will show in Theorem \ref{corconv},  we obtain after $M+1$ iteration steps a kernel vector ${\mathbf w} \in {\mathbb C}^{M+1}$ of the Loewner matrix in the main loop of Algorithm \ref{alg2}.  This can be simply recognized by checking the occurring smallest singular value.
The remaining part of Algorithm \ref{alg2} is now  modified as follows. We store the set  $\Sigma^{(1)}$ of indices $j$ corresponding to zero components of ${\mathbf w}$ (i.e., with $|w_{j}| \le \epsilon$).  Further, we consider the vector ${\mathbf r}-{\mathbf g}_{{\mathbf \Gamma} 
}$ and store the the set $\Sigma^{(2)}$ of indices of components with amplitude larger than $tol$.
Then $\Sigma= \Sigma^{(1)} \cup \Sigma^{(2)}$ is the wanted index set that provides the knots $\{z_{j}: \, j=M_{1}+1, \ldots , M \}$ of the form $z_{j}= \omega_{2N}^{-k}$, $k \in \Sigma$, for $f_{2}(t)$.

Furthermore, ${\mathbf w}$, ${\mathbf S}$ and ${\mathbf g}_{{\mathbf S}}$ provide as before a rational function $r_{M_{1}}(z)$ of type $(M_{1}-1, M_{1})$ after removing zero components of ${\mathbf w}$ and possibly occurring Froissart doublets. 
This rational function satisfies the interpolation conditions
$r_{M_{1}}(\omega_{2N}^{-k}) = \omega_{2N}^{k} (\hat{f}_{1})_{k}$, and we can recover all parameters of $f_{1}(t)$.
Finally, the coefficients $\gamma_{j}$ of $f_{2}$ are obtained via $\omega_{2N}^{k} (\hat{f}_{2})_{k} = \omega_{2N}^{k} \hat{f}_{k} - r_{M_{1}}(\omega_{2N}^{-k})$ for $k \in \Sigma$ with (\ref{f1c}).

\subsection{Convergence of the iteration procedure of Algorithm \ref{alg2}}
\label{sec:conv}

To examine Algorithm \ref{alg2} theoretically in our setting, we suppose that the given data $f_{k}$, $k=0, \ldots , 2N-1$, for the recovery of  $f$ in (\ref{1.1}) are exact.
Let us first assume  that condition (\ref{1.91}) is satisfied, i.e.,  $z_{j}^{2N} \neq 1$ for $j=1, \ldots , M$. We show that in this case, the iteration in Algorithm \ref{alg2} stops after $M+1$ steps and provides the wanted rational function $r_{M}(z)$ of type $(M-1, M)$.

\begin{theorem}\label{AAAconv}
Let $f(t)$ be of the form $(\ref{1.1})$ with $M \in {\mathbb N}$, $\gamma_{j} \in {\mathbb C}\setminus \{ 0 \}$, and $z_{j} \in {\mathbb C} \setminus \{0 \}$, where we assume that the $z_{j}$ satisfy $z_{j}^{2N} \neq 1$ and are pairwise distinct.
 Let ${\mathbf f} = (f_{k})_{k=0}^{2N-1}= (f(k))_{k=0}^{2N-1} \in {\mathbb C}^{2N}$ be the given function values, where $M < N$, and let
 $\hat{\mathbf f} = (\hat{f}_{\ell})_{\ell=0}^{2N-1} = {\mathbf F}_{2N} \, {\mathbf f}$ $($with ${\mathbf F}_{2N} \coloneqq (\omega_{2N}^{\ell k})_{\ell,k=0}^{2N-1})$   be the DFT of ${\mathbf f}$.  Then Algorithm \ref{alg2}  terminates after  $M+1$ steps and determines 
 a partition $S_{M+1} \cup \Gamma_{M+1}$ of $I\coloneqq \{0, \ldots , 2N-1\}$ and a rational function $r_{M}(z)$ of type ($M-1, M)$ satisfying  $r_{M} (\omega_{2N}^{-\ell}) = \omega_{2N}^{\ell} \hat{f}_{\ell}$, $\ell=0, \ldots , 2N-1$.
 In particular, the Loewner matrix 
\begin{equation}\label{LL} {\mathbf L}_{2N-M-1,M+1}  = \left( \frac{\omega_{2N}^{\ell} \, \hat{f}_{\ell} - \omega_{2N}^{k} \, \hat{f}_{k} }{\omega_{2N}^{-\ell} - \omega_{2N}^{-k}} \right)_{\ell \in \Gamma_{M+1}, k \in S_{M+1}}, 
\end{equation}
 has rank $M$ and the kernel vector ${\mathbf w} \in {\mathbb C}^{M+1}$ has only nonzero  components.
 \end{theorem}

\begin{proof}  Assume that $z_{j} \in {\mathbb C} \setminus \{0 \} $ and $z_{j}^{2N } \neq 1$ for $j=1, \ldots , M$.
First observe that a rational function $r_{M}(z) = \frac{p_{M-1}(z)}{q_{M}(z)} $ with numerator polynomial  $p_{M-1}(z)$ of degree at most $M-1$ and denominator polynomial $q(z)$ of degree $M$ is already completely  determined by $2M$ (independent) interpolation conditions 
$r_{M} (\omega_{2N}^{-\ell}) = \omega_{2N}^{\ell} \hat{f}_{\ell}$, if the rational interpolation problem is solvable at all.  But this is clear because of (\ref{1.11}).  In particular, the given data $\omega_{2N}^{\ell} \hat{f}_{\ell}$, $\ell \in I$, cannot be interpolated by a rational function of smaller type than $(M-1,M)$.
Assume now that at the $(M+1)$-st iteration step in Algorithm \ref{alg2}, the index set $S_{M+1} \subset I$ with $M+1$ interpolation indices has been chosen, and let $\Gamma_{M+1} = I \setminus S_{M+1}$.
Then ${\mathbf L}_{2N-M-1,M+1}$ in (\ref{LL}) can by (\ref{1.11}) be factorized as
\begin{align}
\nonumber
& {\mathbf L}_{2N-M-1,M+1} = \left( \frac{ \sum\limits_{j=1}^{M} \gamma_{j} (1-z_{j}^{2N}) \left( \frac{1}{\omega_{2N}^{-\ell} - z_{j}} - \frac{1}{\omega_{2N}^{-k} - z_{j}} \right) }{ \omega_{2N}^{-\ell} - \omega_{2N}^{-k} } \right)_{\ell \in \Gamma_{M+1}, k\in S_{M+1}} \\
\nonumber
&\quad {}= \left( \sum_{j=1}^{M} \gamma_{j} (z_{j}^{2N}-1)\left( \frac{1}{ (\omega_{2N}^{-\ell} - z_{j})  (\omega_{2N}^{-k} - z_{j})} \right) \right)_{\ell \in \Gamma_{M+1}, k \in S_{M+1}}\\
\label{LLfac}
&\quad {}= \left( \frac{1}{\omega_{2N}^{-\ell} - z_{j}} \right)_{\ell\in \Gamma_{M+1}, j=1, \ldots , M}  \mathrm{diag}  \left( \big(\gamma_{j} (z_{j}^{2N}-1)\big)_{j=1}^{M} \right)  \left( \frac{1}{\omega_{2N}^{-k} - z_{j}} \right)_{j=1, \ldots , M, k\in S_{M+1} }. 
\end{align}
Therefore, ${\mathbf L}_{2N-M-1,M+1}$ has exactly rank $M$, since all three matrix factors  have full rank $M$.  Thus the  normalized kernel vector ${\mathbf w}$ with ${\mathbf L}_{2N-M-1,M+1} \, {\mathbf w}={\mathbf 0}$ is uniquely determined and, by (\ref{LLfac}), also satisfies
$$ \left( \frac{1}{\omega_{2N}^{-k} - z_{j}} \right)_{j=1, \ldots , M, k\in S_{M+1} } \, {\mathbf w} = {\mathbf 0}. $$
Since any $M$ columns of the Cauchy matrix $\left( \frac{1}{\omega_{2N}^{-k} - z_{j}} \right)_{j=1, \ldots , M, k\in S_{M+1} }$ are linearly independent, we conclude that all components of ${\mathbf w}$ need to be nonzero.
We now observe that the rational function $r(z) = \tilde{p}_{M}(z) / \tilde{q}_{M}(z)$ as in (\ref{bar-form-1}) is completely determined by $S_{M+1}$,  ${\mathbf w}= (w_{k})_{k\in S_{M+1}}$ and $(\hat{f}_{k})_{k\in S_{M+1}}$ and  satisfies the interpolation conditions
$r (\omega_{2N}^{-k}) = \omega_{2N}^{k} \, \hat{f}_{k}$ for $k \in S_{M+1}$ by construction.
Further, ${\mathbf L}_{2N-M-1,M+1} \, {\mathbf w}={\mathbf 0}$ yields, by (\ref{mini}) and (\ref{mini1}), that 
$\tilde{p}_{M} (\omega_{2N}^{-\ell}) = \omega_{2N}^{\ell}\,  \hat{f}_{\ell} \, \tilde{q}_{M}(\omega_{2N}^{-\ell})$, $\ell \in \Gamma_{M+1}$.
Therefore, $z_{j} \neq \omega_{2N}^{-\ell}$ for all $j=1, \ldots, M$ and $\ell \in \Gamma_{M+1}$ implies 
$r (\omega_{2N}^{-\ell}) = \omega_{2N}^{\ell} \hat{f}_{\ell}$, $\ell \in \Gamma_{M+1}$, and the iteration in Algorithm \ref{alg2} terminates. Since a rational function $r(z)$ satisfying all interpolation conditions is uniquely determined and needs to have type $(M-1,M)$ by (\ref{1.11}), it follows that Algorithm \ref{alg2} provides the wanted rational function $r_{M}(z)$. 
\end{proof}

\begin{remark} 
Note that the idea of the proof of Theorem \ref{AAAconv} works for any partition $S_{M+1} \cup \Gamma_{M+1} = I$, and the special iterative greedy determination of this partition described in Algorithm \ref{alg2} is not taken into account. 
 In other words, Theorem \ref{AAAconv} only shows convergence of the AAA algorithm for exact input data and disregards possible rounding errors, i.e. it does not say anything about the numerical stability of the AAA algorithm.
The essential improvement of numerical stability achieved by greedy selection of the support points in the AAA algorithm has been shown only numerically  so far, see \cite{NST18}.
\end{remark}


Let us now consider the more general case that there are also knots $z_{j}$ with $z_{j}^{2N} = 1$.

\begin{theorem}\label{corconv}
Let $f(t) = f_{1}(t)+ f_{2}(t)$ be of the form $(\ref{f1})$ and $(\ref{f2})$  with $0 \le M_{1} <M \in {\mathbb N}$, $\gamma_{j} \in {\mathbb C}\setminus \{ 0 \}$, and $z_{j} \in {\mathbb C}\setminus \{ 0 \}$  pairwise distinct.
Let ${\mathbf f} = (f_{k})_{k=0}^{2N-1}= (f(k))_{k=0}^{2N-1} \in {\mathbb C}^{2N}$ be the given function values, where $M < N$, and let
 $\hat{\mathbf f} = (\hat{f}_{\ell})_{\ell=0}^{2N-1} = {\mathbf F}_{2N} \, {\mathbf f}$  be the DFT of ${\mathbf f}$.  Then the modification of Algorithm $\ref{alg2}$  described in Section $\ref{secper}$ terminates after  $M+1$ steps and determines 
 a partition $S_{M+1} \cup \Gamma_{M+1}$ of $I \coloneqq \{0, \ldots , 2N-1\}$ and a rational function $r_{M_{1}}(z)$ of type ($M_{1}-1, M_{1})$ satisfying  $r_{M_{1}} (\omega_{2N}^{-\ell}) = \omega_{2N}^{\ell} (\hat{f}_{1})_{\ell}$, $\ell=0, \ldots , 2N-1$.
\end{theorem}

\begin{proof} 1. Let 
$\Sigma \coloneqq \{ k_{j}, \, j=M_{1}+1, \ldots , M: \, z_{j} = \omega_{2N}^{-k_{j}} \}$ be the index set  corresponding to $z_{j}$, $j=M_{1}+1, \ldots , M$, of $f_{2}$.
 Then $\Sigma$ has cardinality $M-M_{1}$. 
From our observations in Section \ref{secper} it follows that 
\begin{equation}\label{peri} \omega_{2N}^{k} \hat{f}_{k} = \left\{ \begin{array}{ll}
\sum\limits_{j=1}^{M_{1}} \gamma_{j} \, \left( \frac{1-z_{j}^{2N}}{\omega_{2N}^{-k}- z_{j}} \right) & \textrm{if} \quad k \in I \setminus \Sigma,\\
\sum\limits_{j=1}^{M_{1}} \gamma_{j} \, \left( \frac{1-z_{j}^{2N}}{\omega_{2N}^{-k}- z_{j}} \right) + 2N\, \gamma_{\ell} & \textrm{if} \quad k \in \Sigma, \, z_{\ell} = \omega_{2N}^{-k}. \end{array} \right. 
\end{equation}
Thus $M-M_{1}$ interpolation values $\omega_{2N}^{k} \hat{f}_{k}$ do not  possess the rational structure, while  all other $2N-(M-M_{1})$ have the rational form.
Employing Algorithm \ref{alg2}, we  obtain at the $(M+1)$-st iteration step  a Loewner matrix ${\mathbf L}_{2N-M-1,M+1}$. This matrix has rank $M$, as can be seen (independently from the position of the knots $z_{j}$, $j=1, \ldots, M$) from (\ref{5.2}), where the Loewner matrix is factorized using the Hankel matrix ${\mathbf H}_{2N-M,M+1}$, which has rank $M$ by (\ref{1.4}).
In particular, it can be shown along the lines of the proof of Theorem \ref{fac} that any matrix ${\mathbf L}_{2N-J,J}$, found at the $J$-th iteration step, with $J < M+1$ has full column rank $J$, see also \cite{AA86}. Therefore the iteration of Algorithm \ref{alg2} will not stop earlier than after $M+1$ steps. 

2. Let $S_{M+1}$ be the index set for interpolation found at the $(M+1)$-st iteration step of Algorithm \ref{alg2}, 
i.e.,  we have $I= S_{M+1} \cup \Gamma_{M+1} =  \Sigma \cup (I \setminus \Sigma)$.
Then $S_{M+1}$ contains at least $M+1- (M-M_{1})= M_{1}+1$ indices, where the DFT coefficients possess the rational structure.
Let $\Sigma^{(1)}  \cup \Sigma^{(2)} = \Sigma$ be a partition of $\Sigma$ such that $\Sigma^{(1)}  \subset S_{M+1}$  and $\Sigma^{(2)} \subset \Gamma_{M+1}= I \setminus S_{M+1}$. Further, we denote with ${\mathbf e}^{(2N-M-1)}_{\ell} \in {\mathbb C}^{2N-M-1}$  the $\ell$-th unit vector of length $2N-M-1$, and with ${\mathbf e}^{(M+1)}_{\ell} \in {\mathbb C}^{M+1}$  the $\ell$-th unit vector of length $M+1$. 
Then the  Loewner matrix  ${\mathbf L}_{2N-M-1,M+1}$ obtained at the $(M+1)$-st iteration step has by (\ref{LL}) and (\ref{peri}) the structure
\begin{eqnarray}\label{ext}
{\mathbf L}_{2N-M-1,M+1} &=&  \left( \frac{ \sum\limits_{j=1}^{M_{1}} \gamma_{j} (1-z_{j}^{2N}) \left( \frac{1}{\omega_{2N}^{-\ell} - z_{j}} - \frac{1}{\omega_{2N}^{-k} - z_{j}} \right) }{ \omega_{2N}^{-\ell} - \omega_{2N}^{-k} } \right)_{\ell \in \Gamma_{M+1}, k\in S_{M+1}} \\
 \nonumber
& & 
- \sum_{k_{j} \in \Sigma^{(1)}}  2N \, {\gamma}_{j} \,\left( \frac{1}{\omega_{2N}^{-\ell} - z_{j}}\right)_{\ell \in \Gamma_{M+1}} \, \left({\mathbf e}_{\mathrm{ind}(k_j)}^{(M+1)}\right)^{T} \\
\nonumber
& & +   \sum_{k_{j} \in \Sigma^{(2)}} 2N \, {\gamma}_{j} \, {\mathbf e}_{\mathrm{ind}(k_j)}^{(2N-M-1)} \, \left( \left( \frac{1}{z_{j} - \omega_{2N}^{k}} \right)_{k \in S_{M+1}} \right)^{T},
\end{eqnarray}
where $\gamma_{j}$ corresponds to  $z_{j} = \omega_{2N}^{-k_{j}}$, $k_{j} \in \Sigma$ and $\mathrm{ind}(k_j)$ denotes the position, or index, of $k_j$ in the ordered set $S_{M+1}$, respectively $\Gamma_{M+1}$.
In other words, each of the ``non-rational'' DFT coefficients of $f_{2}$ corrupts the rational structure of the Loewner matrix and changes either a  column (if the corresponding index is in $ \Sigma^{(1)} \subset \,  S_{M+1}$) or a row (if the corresponding index is in $ \Sigma^{(2)} \subset \,  \Gamma_{M+1}$). As in  the proof of Theorem \ref{AAAconv}, we observe now that the first matrix in (\ref{ext}) factorizes as 
\begin{eqnarray*} \label{first}
& &  \tilde{\mathbf L}_{2N-M-1, M+1} \coloneqq \left( \frac{ \sum\limits_{j=1}^{M_{1}} \gamma_{j} (1-z_{j}^{2N}) \left( \frac{1}{\omega_{2N}^{-\ell} - z_{j}} - \frac{1}{\omega_{2N}^{-k} - z_{j}} \right) }{ \omega_{2N}^{-\ell} - \omega_{2N}^{-k} } \right)_{\ell \in \Gamma_{M+1}, k\in S_{M+1}}\\
&=& \left( \frac{1}{\omega_{2N}^{-\ell} - z_{j}} \right)_{\ell\in \Gamma_{M+1}, j=1, \ldots , M_{1}}  \mathrm{diag}  \left( \big(\gamma_{j} (1-z_{j}^{2N})\big)_{j=1}^{M_{1}} \right)  \left( \frac{1}{\omega_{2N}^{-k} - z_{j}} \right)_{j=1, \ldots , M_{1}, k\in S_{M+1} }
\end{eqnarray*}
and therefore has rank $M_{1}$. 

3.  
Since ${\mathbf L}_{2N-M-1,M+1}$  has rank $M$ by (\ref{5.2}) in Theorem \ref{fac}, and $\Sigma = \Sigma^{(1)} \cup \Sigma^{(2)}$ has cardinality $M-M_{1}$, we conclude that each of the $M-M_{1}$ rank-1 matrices  in (\ref{ext}) enlarges the rank by $1$. Therefore, the normalized kernel vector ${\mathbf w}$ of ${\mathbf L}_{2N-M-1,M+1}$ is uniquely  determined by the conditions
$\tilde{\mathbf L}_{2N-M-1, M+1} {\mathbf w} = {\mathbf 0}$,  as well as $({\mathbf e}_{j}^{(M+1)})^{T} {\mathbf w} =0$ for $k_{j} \in \Sigma^{(1)}$ and 
$\left( \Big( \frac{1}{z_{j} - \omega_{2N}^{k}} \Big)_{k \in S_{M+1}} \right)^{T} \, {\mathbf w} =0$ for $k_{j} \in \Sigma^{(2)}$. 
In particular, $\tilde{\mathbf L}_{2N-M-1, M+1} {\mathbf w} = {\mathbf 0}$ already implies that $S_{M+1}$ and ${\mathbf w}$ provide a rational function $r$ that satisfies all interpolation conditions $r(\omega_{2N}^{-\ell}) = \omega_{2N}^{\ell} \hat{f}_{\ell}$ for $\ell \in I \setminus \Sigma$. 
However, since ${\mathbf w}$ is a kernel vector  to all rank-1 matrices in (\ref{ext}) caused by the DFT-coefficients of $f_{2}$, it follows that 
$r(\omega_{2N}^{-\ell}) = \omega_{2N}^{\ell} (\hat{f}_{1})_{\ell}$ for all $\ell \in I$. 
We can simplify $r(z)$ by removing all zero components of ${\mathbf w}$ in the definition of $\tilde{p}(z)$ and $\tilde{q}(z)$ in (\ref{bar-form-1}).
The indices of $\Sigma^{(2)}$ may produce Froissart doublets that can be removed, such that we finally get the rational function $r=r_{M_{1}}$ of type $(M_{1}-1, M_{1})$.
\end{proof}

\begin{remark}\label{remrank} 
Let the degree of a rational function $r(z)= \frac{p(z)}{q(z)}$ be defined  as $\mathrm{deg} \, r \coloneqq \max \{ \mathrm{deg} \, p, \, \mathrm{deg} \, q\}$, where $\mathrm{deg} \, p$ and $\mathrm{deg} \, q$ denote the degrees of the numerator and the denominator polynomials $p$ and $q$. If the rational interpolation problem is solvable, then the degree  of the rational interpolant $r(z)$  is always equal  to the rank of  the Loewner matrix obtained from  any partition  $S \cup \Gamma = I$ of interpolation indices, as long as  $S$ and $\Gamma$ are both large enough. This has already been shown in \cite{Be70} for pairwise distinct interpolation points, and in \cite{AA86} for the more general case of rational Hermite-type interpolation. In the proof of Theorem \ref{AAAconv}, the rank  condition for  ${\mathbf L}_{2N-M-1,M+1}$ in (\ref{LL})  follows directly from  the factorization (\ref{LLfac}). In Theorem \ref{corconv}, it is much more difficult to see that ${\mathbf L}_{2N-M-1,M+1}$ in (\ref{ext}) has exactly rank $M$. Here, we referred to the factorization (\ref{5.2}) that connects the Loewner matrix with a Hankel matrix. To our knowledge, the more general setting of Theorem \ref{corconv} has not been studied before in the context of rational interpolation.
\end{remark}

\section{Connection  between ESPIRA, MPM and ESPRIT}
\subsection{Relations between Hankel and Loewner Matrices}
\label{sec6.1}

As shown in Section \ref{sec:MPM}, MPM and ESPRIT are essentially based  on solving the generalized eigenvalue problem 
$$ {\mathbf H}_{2N-L, L}(1) {\mathbf v} = z \, {\mathbf H}_{2N-L,L}(0) \, {\mathbf v}, $$
where  ${\mathbf H}_{2N-L,L}(0)$ and ${\mathbf H}_{2N-L,L}(1)$ are submatrices of ${\mathbf H}_{2N-L,L+1} = (h_{\ell+k})_{\ell,k=0}^{2N-L-1,L}$.
By contrast, in the ESPIRA algorithm, Algorithm \ref{alg2} is employed, which  iteratively  solves the minimization problem 
$$ \min_{{\mathbf w}} \| {\mathbf L}_{2N-J,J} \, {\mathbf w} \|_{2}, \qquad J=1,2, \ldots , M+1, $$
and we obtain  at the $(M+1)$-st iteration step the Loewner matrix
$ {\mathbf L}_{2N-M-1,M+1}$ in (\ref{LL}).
Using the results in \cite{Fiedler} and \cite{AA86} on the connection between Hankel and Loewner matrices, we will now show, that ${\mathbf L}_{2N-M-1,M+1}$ can be represented as a matrix product
$$ {\mathbf L}_{2N-M-1,M+1} = {\mathbf Q}_{2N-M-1, 2N-M}({\mathbf u}^{(1)}) \, {\mathbf H}_{2N-M,M+1} \, {\mathbf Q}_{M+1,M+1} ({\mathbf u}^{(2)})^{T},
$$ 
where ${\mathbf H}_{2N-M,M+1}= (f_{\ell+k})_{\ell,k=0}^{2N-1-M,M}$ is the Hankel matrix of given function values, and where
the matrices ${\mathbf Q}_{2N-M-1, 2N-M}({\mathbf u}^{(1)})$ and ${\mathbf Q}_{M+1, M+1}({\mathbf u}^{(2)})$ have a special structure.
This observation will lead us to a new interpretation of ESPIRA as a matrix pencil method for Loewner matrices. 

We start with some notations.
For a given vector ${\mathbf u}= (u_{n})_{n=1}^{K} \in {\mathbb C}^{K}$, $K \in {\mathbb N}$,  let 
$$ q_{\mathbf u} (z) \coloneqq \prod_{n=1}^{K} (z-u_{n}) , $$
and for $\ell=1, \ldots , K$, 
$$ q_{{\mathbf u}, \ell} (z) \coloneqq \frac{q_{\mathbf u}(z)}{z- u_{\ell}} = \sum_{r=0}^{K-1} q_{\ell,r} \, z^{r}$$
with coefficient vectors ${\mathbf q}_{{\mathbf u},\ell} \coloneqq (q_{\ell,r})_{r=0}^{K-1} \in {\mathbb C}^{K}$.
We  define for $K \le L$ the matrices ${\mathbf Q}_{K,K}({\mathbf u})$ and ${\mathbf Q}_{K,L}({\mathbf u})$ corresponding to ${\mathbf u} \in {\mathbb C}^{K}$ as
\begin{equation}\label{W} {\mathbf Q}_{K,K}({\mathbf u} ) \coloneqq \left( \begin{array}{c}
{\mathbf q}_{{\mathbf u},1}^{T} \\
{\mathbf q}_{{\mathbf u},2}^{T} \\
\vdots  \\
{\mathbf q}_{{\mathbf u},K}^{T} \end{array} \right) \in {\mathbb C}^{K \times K},
\qquad 
{\mathbf Q}_{K,L}({\mathbf u} ) \coloneqq \left( \begin{array}{c| c c c }
{\mathbf q}_{{\mathbf u},1}^{T} & 0 & \ldots &0 \\
{\mathbf q}_{{\mathbf u},2}^{T} & 0 & \ldots &0 \\
\vdots & \vdots & & \vdots \\
{\mathbf q}_{{\mathbf u},K}^{T} & 0 & \ldots & 0\end{array} \right) \in {\mathbb C}^{K \times L},
\end{equation}
where in case of $L>K$, we have added $L-K$ zero columns to the $K \times K$ matrix obtained from the $K$ rows ${\mathbf q}_{{\mathbf u},\ell}^{T}$. 

Recall that  ${\mathbf V}_{L,M} ({\mathbf z})= \left( z_{j}^{\ell} \right)_{\ell=0,j=1}^{L-1,M} $  denotes the Vandermonde matrix generated by ${\mathbf z}=(z_{1}, \ldots , z_{M})^{T} \in {\mathbb C}^{M}$. Then, we find for $L \ge M$, $L \ge K$ and arbitrary vectors ${\mathbf u} \in {\mathbb C}^{K}$ and ${\mathbf z} \in {\mathbb C}^{M}$,
\begin{equation}\label{5.1} {\mathbf Q}_{K,L} ({\mathbf u} ) \, {\mathbf V}_{L,M}({\mathbf z})  = \left( \frac{q_{{\mathbf u}}(z_{j}) }{z_{j} - u_{\ell}} \right)_{\ell,j=1}^{K,M}= \Big(\prod\limits_{\substack{ n=1 \\ n \neq \ell }}^{K} (z_j-u_n)  \Big)_{\ell,j=1}^{K,M}.
\end{equation}
In particular, for ${\mathbf z}= {\mathbf u} \in {\mathbb C}^{K}$, the matrix
 ${\mathbf Q}_{K,L}({\mathbf u})$  diagonalizes the Vandermonde matrix ${\mathbf V}_{L,K} ({\mathbf u})$, since
$$ {\mathbf Q}_{K,L} ({\mathbf u}) \, {\mathbf V}_{L,K}({\mathbf u})  = \left( \frac{q_{{\mathbf u}}(u_{j})}{u_{j} - u_{\ell}} \right)_{\ell,j=1}^{K,K} = \left\{ \begin{array}{ll}
0 & j\neq \ell,\\
q_{{\mathbf u},\ell}(u_{\ell}) & j = \ell, \end{array} \right.
$$
and ${\mathbf Q}_{K,L} ({\mathbf u})$ has full rank $K\le L$ for pairwise distinct  $u_{n}$, $n=1, \ldots , K$, see also \cite{Fiedler}, Theorem 3.

Then we have

\begin{theorem}\label{fac}
Let ${\mathbf f}=(f_{k})_{k=0}^{2N-1}$ be the vector of given function values of $f$ in $(\ref{1.1})$ with $N > M$.
Let $S_{M+1} \cup \Gamma_{M+1}  = \{0, \ldots , 2N-1\} \eqqcolon I$ be a partition  of the index set $I$, where $S_{M+1}$ contains $M+1$ indices,  and 
\begin{equation}\label{uvec}
{\mathbf u}^{(1)} \coloneqq (\omega_{2N}^{-\ell})_{\ell \in \Gamma_{M+1}} \in {\mathbb C}^{2N-M-1}, \qquad 
{\mathbf u}^{(2)} \coloneqq (\omega_{2N}^{-k})_{k \in S_{M+1}} \in {\mathbb C}^{M+1}.
\end{equation} 
Further, let 
${\mathbf L}_{2N-M-1, M+1} \in {\mathbb C}^{(2N-M-1)\times (M+1)} $ be the Loewner matrix as in $(\ref{LL})$  determined 
by $\hat{\mathbf f} = (\hat{f}_{\ell})_{\ell=0}^{2N-1}$ and $S_{M+1} \cup \Gamma_{M+1}$ and ${\mathbf H}_{2N-M,M+1}$  the Hankel matrix as in $(\ref{HH})$. Then we have
\begin{equation}\label{5.2}
{\mathbf L}_{2N-M-1,M+1} = {\mathbf Q}_{2N-M-1, 2N-M}({\mathbf u}^{(1)}) \, {\mathbf H}_{2N-M,M+1} \, {\mathbf Q}_{M+1,M+1}({\mathbf u}^{(2)})^{T} ,
\end{equation} 
where ${\mathbf Q}_{2N-M-1, 2N-M}({\mathbf u}^{(1)})$ and ${\mathbf Q}_{M+1,M+1} ({\mathbf u}^{(2)})$ are defined as in $(\ref{W})$.
In particular, $\mathrm{rank}\, {\mathbf L}_{2N-M-1,M+1} = \mathrm{rank}\, {\mathbf H}_{2N-M,M+1} = M$.
\end{theorem}

\begin{proof}
Recall from (\ref{1.4}) that 
$$ {\mathbf H}_{2N-M,M+1} = {\mathbf V}_{2N-M, M}  ({\mathbf z}) \, \textrm{diag}  \left( (\gamma_{j})_{j=1}^{M} \right) \, {\mathbf V}_{M+1,M}({\mathbf z})^{T},$$
where ${\mathbf z} = (z_{j})_{j=1}^{M}$ and $(\gamma_{j})_{j=1}^{M}$ are the parameters determining $f$ in $(\ref{1.1})$ and rank ${\mathbf H}_{2N-M,M+1} = M$.
By (\ref{5.1}), (\ref{uvec}),  and $(\ref{1.11})$ it follows for $z_j^{2N}\neq 1$ that
\begin{eqnarray*}
&& \hspace*{-15mm} {\mathbf Q}_{2N-M-1, 2N-M}({\mathbf u}^{(1)}) \, {\mathbf H}_{2N-M,M+1} \, {\mathbf Q}_{M+1,M+1} ({\mathbf u}^{(2)})^{T} \\
&=& {\mathbf Q}_{2N-M-1, 2N-M}({\mathbf u}^{(1)}) \, {\mathbf V}_{2N-M, M}  ({\mathbf z}) \, \textrm{diag}  \Big( (\gamma_{j})_{j=1}^{M} \Big) \, {\mathbf V}_{M+1,M}({\mathbf z})^{T} \, {\mathbf Q}_{M+1,M+1}({\mathbf u}^{(2)})^{T} \\
&=& \left( \frac{q_{{\mathbf u}^{(1)}}(z_{j})}{z_{j}-w_{2N}^{-\ell}} \right)_{\ell \in \Gamma_{M+1}, j=1, \ldots , M} 
\textrm{diag}  \left((\gamma_{j})_{j=1}^{M} \right) \,  \left( \frac{q_{{\mathbf u}^{(2)}}(z_{j})}{z_{j}-w_{2N}^{-k}} \right)^{T}_{k \in S_{M+1}, j=1, \ldots , M}  \\
&=& \left( \sum_{j=1}^{M} \gamma_{j} \frac{ \prod_{r=0}^{2N-1} (z_{j}- \omega_{2N}^{-r})}{(z_{j} - \omega_{2N}^{-\ell}) (z_{j} - \omega_{2N}^{-k})}  \right)_{\ell \in \Gamma_{M+1}, k \in S_{M+1}} \\
&=& \left( \sum_{j=1}^{M} \frac{\gamma_{j}(z_{j}^{2N} -1)}{(\omega_{2N}^{-\ell} - \omega_{2N}^{-k})} \left(  \frac{ 1}{z_{j} - \omega_{2N}^{-\ell}} - \frac{1}{z_{j} - \omega_{2N}^{-k}} \right)\right)_{\ell \in \Gamma_{M+1}, k \in S_{M+1}} \\
&=& \left( \frac{\sum\limits_{j=1}^{M} \gamma_{j}\left( \frac{1-z_{j}^{2N}}{\omega_{2N}^{-\ell} - z_{j}}\right)  -  \sum\limits_{j=1}^{M} \gamma_{j}\left( \frac{1-z_{j}^{2N}}{\omega_{2N}^{-k} - z_{j}}\right)  }{ \omega_{2N}^{-\ell} - \omega_{2N}^{-k}}\right)_{\ell \in \Gamma_{M+1}, k \in S_{M+1}} \\
&=&\left( \frac{\omega_{2N}^{\ell} \, \hat{f}_{\ell} - \omega_{2N}^{k} \, \hat{f}_{k} }{\omega_{2N}^{-\ell} - \omega_{2N}^{-k}} \right)_{\ell \in \Gamma_{M+1}, k \in S_{M+1}} = {\mathbf L}_{2N-M-1,M+1}.
\end{eqnarray*}
In the case $z_j^{2N}=1$ we apply (\ref{lop}) and similar considerations.
The rank condition follows directly by observing that rank ${\mathbf Q}_{2N-M-1, 2N-M}({\mathbf u}^{(1)})= 2N-M-1$ and rank ${\mathbf Q}_{M+1,M+1} ({\mathbf u}^{(2)})^{T} = M+1$.
\end{proof}

Similarly, we can  transform the two submatrices ${\mathbf H}_{2N-M, M}(0)$ and ${\mathbf H}_{2N-M, M}(1)$ applied in  MPM and  ESPRIT into Loewner matrices with index sets $S_{M}$ and $\Gamma_{M} = I \setminus S_{M}$, where $S_{M}$ contains $M$ indices of $I$. In practice, such  a partition will be obtained by applying the iterative greedy search of interpolation indices in Algorithm \ref{alg2}. Then we have

\begin{theorem}\label{01}
Let 
$${\mathbf H}_{2N-M, M}(0) = (f_{k+\ell})_{\ell,k=0}^{2N-M-1,M-1}, \qquad {\mathbf H}_{2N-M, M}(1) = (f_{k+\ell+1})_{\ell,k=0}^{2N-M-1,M-1}$$
be the Hankel matrices determined by ${\mathbf f} = (f_{k})_{k=0}^{2N-1}$. Let $S_{M} \cup \Gamma_{M} = \{0, \ldots , 2N-1\} \coloneqq I$ be a given partition of $I$, where $S_{M}$ contains $M$ indices and $\Gamma_{M}$ the remaining $2N-M$ indices.
Let ${\mathbf u}^{(1)} \coloneqq (\omega_{2N}^{-\ell})_{\ell \in \Gamma_{M}}$ and ${\mathbf u}^{(2)} \coloneqq (\omega_{2N}^{-k})_{k \in S_{M}}$.
Then
\begin{eqnarray}\nonumber
{\mathbf L}_{2N-M,M}(0) &\coloneqq&  \!\!\left( \frac{\omega_{2N}^{\ell} \, \hat{f}_{\ell} - \omega_{2N}^{k} \, \hat{f}_{k} }{\omega_{2N}^{-\ell} - \omega_{2N}^{-k}} \right)_{\ell \in \Gamma_{M}, k \in S_{M}}  \\
&=& {\mathbf Q}_{2N-M,2N-M}({\mathbf u}^{(1)}) \, {\mathbf H}_{2N-M,M}(0) \, {\mathbf Q}_{M,M}({\mathbf u}^{(2)})^{T}, 
\label{L1}
\end{eqnarray}
and 
\begin{eqnarray}\nonumber
{\mathbf L}_{2N-M,M}(1) &\coloneqq &  \left( \frac{ \hat{f}_{\ell} -  \hat{f}_{k} }{\omega_{2N}^{-\ell} - \omega_{2N}^{-k}} \right)_{\ell \in \Gamma_{M}, k \in S_{M}}  \\
\label{L2}
&=& {\mathbf Q}_{2N-M,2N-M}({\mathbf u}^{(1)}) \, {\mathbf H}_{2N-M,M}(1) \, {\mathbf Q}_{M,M}({\mathbf u}^{(2)})^{T},
\end{eqnarray}
where ${\mathbf Q}_{2N-M,2N-M}({\mathbf u}^{(1)})$ and ${\mathbf Q}_{M,M}({\mathbf u}^{(2)})$ are defined as in $(\ref{W})$. 
\end{theorem}

\begin{proof}
First consider ${\mathbf H}_{2N-M, M}(0) = (f_{k+\ell})_{\ell,k=0}^{2N-M-1,M-1}$, which is defined by the components $f_{k}$, $k=0, \ldots , 2N-2$.
In particular, this matrix does not contain $f_{2N-1}$. Therefore, we consider instead of ${\mathbf f}= (f_{k})_{k=0}^{2N-1}$  and $\hat{\mathbf f} = (\hat{f}_{k})_{k=0}^{2N-1} = {\mathbf F}_{2N} {\mathbf f}$ the vector ${\mathbf f}^{(0)} \coloneqq  (f_{0}, \ldots, f_{2N-2}, 0)^{T}$. Let $\hat{\mathbf  f}^{(0)} = (\hat{f}_{k}^{(0)})_{k=0}^{2N-1}= {\mathbf F}_{2N} \, {\mathbf f}^{(0)}$ be the discrete Fourier  transform of ${\mathbf f}^{(0)}$. 
Then we obtain the factorization 
$$
\left( \frac{\omega_{2N}^{\ell} \, \hat{f}^{(0)}_{\ell} - \omega_{2N}^{k} \, \hat{f}^{(0)}_{k} }{\omega_{2N}^{-\ell} - \omega_{2N}^{-k}} \right)_{\ell \in \Gamma_{M}, k \in S_{M}}  \\
= {\mathbf Q}_{2N-M,2N-M}({\mathbf u}^{(1)}) \, {\mathbf H}_{2N-M,M}(0) \, {\mathbf Q}_{M,M}({\mathbf u}^{(2)})^{T}, 
$$
analogously as in the proof of Theorem \ref{fac}. Further, we observe that 
\begin{eqnarray*}
\omega_{2N}^{\ell} \, \hat{f}^{(0)}_{\ell} - \omega_{2N}^{k} \, \hat{f}^{(0)}_{k} &=&  
\omega_{2N}^{\ell} \, \Big(\hat{f}_{\ell} - \omega_{2N}^{\ell(2N-1)} \, f_{2N-1}\Big) - \omega_{2N}^{k} \, \Big(\hat{f}_{k} - \omega_{2N}^{k(2N-1)}f_{2N-1} \Big)\\
&=& \omega_{2N}^{\ell} \, \hat{f}_{\ell} - \omega_{2N}^{k} \, \hat{f}_{k}
\end{eqnarray*}
for all $\ell, \, k \in I$, 
and the factorization of ${\mathbf L}_{2N-M,M}(0)$ follows.

For ${\mathbf H}_{2N-M, M}(1) = (f_{k+\ell+1})_{\ell,k=0}^{2N-M-1,M-1}$, only the components $f_{k}$, $k=1, \ldots , 2N-1$, are needed.
Therefore, we consider $\hat{\mathbf  f}^{(1)} = {\mathbf F}_{2N} \, {\mathbf f}^{(1)}$ with  ${\mathbf f}^{(1)} \coloneqq (f_{1}, f_{2}, \ldots, f_{2N-1}, 0)^{T}$ of length $2N$ and can conclude analogously as in the proof of Theorem \ref{fac} the factorization 
$$   \left( \frac{ \omega_{2N}^{\ell} \, \hat{f}_{\ell}^{(1)} -  \omega_{2N}^{k} \, \hat{f}_{k}^{(1)} }{\omega_{2N}^{-\ell} - \omega_{2N}^{-k}} \right)_{\ell \in \Gamma_{M}, k \in S_{M}}  
= {\mathbf Q}_{2N-M,2N-M}({\mathbf u}^{(1)}) \, {\mathbf H}_{2N-M,M}(1) \, {\mathbf Q}_{M,M}({\mathbf u}^{(2)})^{T}.
$$
Finally, we show that $ \omega_{2N}^{\ell} \, \hat{f}_{\ell}^{(1)} -\omega_{2N}^{k} \, \hat{f}_{k}^{(1)} =   \hat{f}_{\ell} -   \hat{f}_{k}$, for all $k,\ell \in I$, which in turn leads to (\ref{L2}).
Indeed,
\begin{eqnarray*}
\omega_{2N}^{\ell} \, \hat{f}_{\ell}^{(1)} -\omega_{2N}^{k} \, \hat{f}_{k}^{(1)} &=&
 \sum_{r=0}^{2N-2}  \left( f_{r+1} \, \omega_{2N}^{\ell (r+1)} -  f_{r+1} \, \omega_{2N}^{k(r+1)} \right)
= \sum_{r=1}^{2N-1}  f_{r} \, (\omega_{2N}^{\ell r} -   \omega_{2N}^{kr}) \\
&=& (\hat{f}_{\ell} - f_{0} ) - (\hat{f}_{k} - f_{0}) = \hat{f}_{\ell} - \hat{f}_{k}.
\end{eqnarray*}
\end{proof} 

\begin{remark} 
1. For ${\mathbf u} \in {\mathbb C}^{K}$, ${\mathbf z} \in {\mathbb C}^{M}$ and $L \ge \max\{K,M\}$,  we define
$$
{\mathbf A}_{K,M}({\mathbf u},  {\mathbf z}) \coloneqq
{\mathbf Q}_{K,L} ({\mathbf u} ) \, {\mathbf V}_{L,M}({\mathbf z})  = \Big(\prod\limits_{\substack{ n=1 \\ n \neq \ell }}^{K} (z_j-u_n)  \Big)_{\ell,j=1}^{K,M}.
$$
Then Theorems \ref{fac} and \ref{01} imply the following factorizations for Loewner matrices,
\begin{align*}
{\mathbf L}_{2N-M-1,M+1} &={\mathbf A}_{2N-M-1,M}({\mathbf u}^{(1)},  {\mathbf z}) \,  \textrm{diag}   \left( (\gamma_{j})_{j=1}^{M} \right) \, {\mathbf A}_{M+1,M}({\mathbf u}^{(2)},  {\mathbf z})^{T}, \\
{\mathbf L}_{2N-M,M}(0)&={\mathbf A}_{2N-M,M}({\mathbf u}^{(1)},  {\mathbf z}) \,  \textrm{diag}  \left( (\gamma_{j})_{j=1}^{M} \right) \, {\mathbf A}_{M,M}({\mathbf u}^{(2)},  {\mathbf z})^{T}, \\
{\mathbf L}_{2N-M,M}(1)&={\mathbf A}_{2N-M,M}({\mathbf u}^{(1)},  {\mathbf z}) \,  \textrm{diag}   \left( (\gamma_{j} z_j)_{j=1}^{M} \right) \, {\mathbf A}_{M,M}({\mathbf u}^{(2)},  {\mathbf z})^{T},
\end{align*}
where  ${\mathbf z} = (z_{j})_{j=1}^{M}$ and $(\gamma_{j})_{j=1}^{M}$ are the parameters that determine $f$ in $(\ref{1.1})$ and  ${\mathbf u}^{(1)}$ and ${\mathbf u}^{(2)}$ are defined in (\ref{uvec}).

2. Theorems \ref{fac} and \ref{01}  can be seen as special cases of Theorem 12 in \cite{Fiedler}. For the proof of Theorem 12, Fiedler referred to an unpublished manuscript that we have not been able to find.  A similar factorization as in (\ref{5.2}) has been also used in \cite{AA86} to show the rank condition for Loewner matrices, see Remark \ref{remrank}.
\end{remark} 

\subsection{ESPIRA as a Matrix Pencil Method}

As seen in Section \ref{sec6.1}, the two new matrices ${\mathbf L}_{2N-M,M}(0)$ and ${\mathbf L}_{2N-M,M}(1)$ in (\ref{L1}) and (\ref{L2}) are obtained as matrix  products involving ${\mathbf H}_{2N-M, M}(0)$ and ${\mathbf H}_{2N-M, M}(1)$, respectively, where the two square matrices ${\mathbf Q}_{2N-M,2N-M}({\mathbf u}^{(1)})$ and ${\mathbf Q}_{M,M}({\mathbf u}^{(2)})$ in Theorem \ref{01} are invertible. In particular, we conclude that the matrix pencil $z{\mathbf L}_{2N-M,M}(0)-{\mathbf L}_{2N-M,M}(1)$ possesses the same eigenvalues $z_{j}$ as the matrix pencil $z{\mathbf H}_{2N-M, M}(0)- {\mathbf H}_{2N-M, M}(1)$, which is considered in ESPRIT and MPM.

We can therefore derive a second variant of the ESPIRA algorithm, which  solves the matrix pencil problem for the new Loewner matrices instead of the  Hankel matrices. 
In order to find a suitable partition $S_{M} \cup \Gamma_{M}$ of $I$, we will apply the greedy search of interpolation indices of the iterative AAA algorithm, which improves the condition numbers of the Loewner matrices, see Algorithm \ref{algESP2}.

\begin{algorithm}[h!]\caption{ESPIRA II}
\label{algESP2}
\small{
\textbf{Input:} ${\mathbf f} = (f_{k})_{k=0}^{2N-1} = \big(f(k)\big)_{k=0}^{2N-1}$ (equidistant samples of $f$ in (\ref{1.1}) with $M  < N$)\\
\phantom{\textbf{Input:}} $tol >0$ tolerance for the (relative) approximation error

\begin{enumerate}
\item Initialization step: 
Compute $\hat{\mathbf f} \coloneqq (\hat{f}_{k})_{k=0}^{2N-1}$ by FFT. \\
Set  $\mathbf \Gamma \coloneqq \left( k \right)_{k=0}^{2N-1}$, ${\mathbf g}_{\mathbf \Gamma} \coloneqq (g_{k})_{k=0}^{2N-1} \coloneqq (\omega_{2N}^{k} \hat{f}_{k})_{k=0}^{2N-1}$, ${\mathbf S}\coloneqq []$; ${\mathbf g}_{\mathbf S}\coloneqq []$, ${\mathbf r} \coloneqq (r_{k})_{k=0}^{2N-1} \coloneqq {\mathbf 0}$; $jmax \coloneqq N$. 
\item Preconditioning step:  \\
for $j=1:$ \textit{jmax}
\begin{itemize}
\item Compute $k \coloneqq \argmax_{\ell \in \mathbf \Gamma} |r_{\ell} - g_{\ell}|$
and update ${\mathbf S} \coloneqq ({\mathbf S}^{T}, k)^{T}$, ${\mathbf g}_{\mathbf S} \coloneqq ({\mathbf g}_{\mathbf S}^{T}, g_{k})^{T}$, and delete $k$ in $\mathbf{\Gamma}$ and $g_{k}$ in ${\mathbf g}_{\mathbf \Gamma}$.
\item Build  $\mathbf{C}_{2N-j,j}\!\coloneqq\!\left( \frac{1}{\omega_{2N}^{-\ell}-\omega_{2N}^{-k}} \right)_{\ell \in \mathbf{\Gamma}, k \in \mathbf{S}}$, ${\mathbf L}_{2N-j,j}\!\coloneqq\! \left( \frac{g_{\ell} - g_{k}}{\omega_{2N}^{-\ell}-\omega_{2N}^{-k}} \right)_{\ell \in \mathbf{\Gamma}, k \in \mathbf{S}}$.
\item Compute the singular vector $
\mathbf{w}$  corresponding to the smallest singular value $\sigma_{j}$ of $\mathbf{L}_{2N-j,j}$. If $\sigma_{j} < tol \, \sigma_{1}$ (where $\sigma_{1}$ is the largest singular value of ${\mathbf L}_{2N-j,j}$), then 
set $\tilde M \coloneqq j-1$; 
delete the last entry of ${\mathbf S}$ and add it to ${\mathbf \Gamma}$;
stop.
 \item Compute $\mathbf{p} \coloneqq \mathbf{C}_{2N-j,j} ({\mathbf  w}.\ast \mathbf{g}_{\mathbf{S}})$, $\mathbf{q} \coloneqq \mathbf{C}_{2N-j,j} {\mathbf w}$ and $\mathbf{r} \coloneqq \mathbf{p}./\mathbf{q}$, where $.*$ denotes componentwise multiplication and $./$ componentwise division.
\end{itemize}
end (for)\\

\item Build the Loewner matrices  
$${\mathbf L}_{2N-\tilde{M},\tilde{M}}(0) = \left( \frac{\omega_{2N}^{\ell} \, \hat{f}_{\ell} - \omega_{2N}^{k} \, \hat{f}_{k} }{\omega_{2N}^{-\ell} - \omega_{2N}^{-k}} \right)_{\ell \in {\mathbf \Gamma}, k \in {\mathbf S}} , \quad
{\mathbf L}_{2N-\tilde{M},\tilde{M}}(1) = \left( \frac{ \hat{f}_{\ell} -  \hat{f}_{k} }{\omega_{2N}^{-\ell} - \omega_{2N}^{-k}} \right)_{\ell \in {\mathbf \Gamma}, k \in {\mathbf S}},$$
and the joint matrix ${\mathbf L}_{2N-\tilde{M},2\tilde{M}} \coloneqq \left(  {\mathbf L}_{2N-\tilde{M},\tilde{M}}(0) , {\mathbf L}_{2N-\tilde{M},\tilde{M}}(1) \right) \in {\mathbb C}^{2N-\tilde{M},2\tilde{M}}$.

\item
Compute the SVD  ${\mathbf L}_{2N-\tilde{M},2\tilde{M}} = {\mathbf U}_{2N-\tilde{M}} \, {\mathbf D}_{2N-\tilde{M},2\tilde{M}} \, {\mathbf W}_{2\tilde{M}}$ (where ${\mathbf U}_{2N-\tilde{M}}$ and ${\mathbf W}_{2\tilde{M}}$ are unitary square matrices).\\
Determine the numerical rank $M$ of ${\mathbf D}_{2N-\tilde{M},2\tilde{M}}$ by taking the smallest $M$ such that $\sigma_{M+1} < tol \, \sigma_{1}$, where $\sigma_{1} \ge \sigma_{2} \ge \ldots \ge \sigma_{2\tilde{M}}$ are the ordered diagonal entries of ${\mathbf D}_{2N-\tilde{M},2\tilde{M}}$.\\
 Determine the vector of eigenvalues ${\mathbf z} =(z_{1}, \ldots , z_{M})^{T}$ of 
$$\left(  {\mathbf W}_{2\tilde{M}}(1:2\tilde{M}, 1:M) \right)^{\dagger}  {\mathbf W}_{2\tilde{M}}(1:2\tilde{M}, M+1:2M), $$
where $\Big({\mathbf W}_{2\tilde{M}}(1:2 \tilde{M}, 1:M)\Big)^{\dagger}$ denotes the Moore-Penrose inverse of  ${\mathbf W}_{2\tilde{M}}(1:2\tilde{M}, 1:M) $.
\item Compute ${\bgamma} = (\gamma_{j})_{j=1}^{M}$  as the least squares solution of the linear system 
$ {\mathbf V}_{2N,M} ({\mathbf z}) \, \bgamma = {\mathbf f}. $
\end{enumerate}

\noindent
\textbf{Output:} $M \in {\mathbb N}$,  $z_{j}, \, \gamma_{j} \in {\mathbb C}$, $j=1, \ldots , M$.}
\end{algorithm}

Note that the preconditioning step 2 in Algorithm \ref{algESP2} is still the core AAA algorithm for rational approximation.
But we do not need to compute the rational function $r$ and its fractional decomposition, since we only use the obtained index sets $S_{M}$ and $\Gamma_{M} = I \setminus S_{M}$ to build the Loewner matrices. For exact data the needed number of iterations $\tilde{M}+1$ in the preconditioning step corresponds to the number of terms in (\ref{1.1}), i.e., we have  $\tilde{M} =M$ according to Theorem \ref{corconv}.
One advantage of ESPIRA-II  compared to ESPIRA-I is that we completely avoid any case study. All wanted knots $z_{j}$ are now solutions of the eigenvalue  problem in step 4, 
and we don't need to consider parameters of the form $z_{j} = \omega_{2N}^{k}$ with a post-processing procedure. 
Compared to MPM and ESPRIT, the essential advantage is that  we only need to compute the FFT of a vector of size $2N$ and the SVD of a matrix of size $(2N-{M}, 2 M)$.  Analogously as for ESPIRA-I,  we therefore have an overall computational cost of ${\mathcal O}(N (M^{3}+\log N))$ while MPM and ESPRIT in Section 2 require ${\mathcal O}(N L^{2})$ flops, and  good recovery results are only achieved with $L\approx N$. 

\begin{remark}\label{cauchyrem}
1. Similarly as in Algorithm \ref{algMPM}, we can apply a QR decomposition of ${\mathbf L}_{2N-M,2M}$ instead of the SVD in step 4 of Algorithm \ref{algESP2}.\\
If all $z_{j}$ satisfy $z_{j}^{2N} \neq 1$, then we can replace  step 5 in Algorithm \ref{algESP2} by \\[1ex]
{\small 5'. Solve the linear system 
$${\mathbf C}_{2N,M} {\mathbf a} = \Big( \omega_{2N}^{k} \hat{f}_{k} \Big)_{k =0}^{2N-1}$$
 with  the Cauchy matrix ${\mathbf C}_{2N, M} \coloneqq \Big( \frac{1}{\omega_{2N}^{-k} - z_{j}} \Big)_{k =0, j=1}^{2N-1,M}$  and set $\gamma_{j} \coloneqq \frac{a_{j}}{1-z_{j}^{2N}}$, $j=1, \ldots , M$,}\\[1ex]
  see also Remark \ref{rem3}(2). This Cauchy matrix usually has a better condition than ${\mathbf V}_{2N,M}({\mathbf z})$ for $|z_{j}| \neq 1$, see e.g. Example \ref{exdiri}.
  
 2. In \cite{IA9}, an overview of system identification methods based on matrix pencils is given.  In that setting, the Loewner matrix pencils and Hankel matrix pencils correspond to frequency domain and time domain, respectively. These observations are in line with our results, where the Loewner matrices in Algorithm \ref{algESP2} are computed from the DFT vector $\hat{\mathbf f}$ while the Hankel matrices in Algorithms \ref{algMPM} and \ref{algESPRIT} are determined by the vector of function values ${\mathbf f}$. 
\end{remark}

\section{Numerical experiments}

In this section we compare the performance of the known MPM and ESPRIT algorithms with our new ESPIRA algorithms for exact and noisy input data. All algorithms are implemented in \textsc{Matlab} and use IEEE standard floating point arithmetic with double precision. 
The \textsc{Matlab} implementations can be found on \texttt{http://na.math.uni-goettingen.de} under \texttt{software}.  Let the  signal $f$ be given as in (\ref{1.1}),  
$$ f(t)=\sum_{j=1}^M  \gamma_{j} \, {\mathrm e}^{ \phi_j t} = \sum_{j=1}^{M} \gamma_{j} \, z_{j}^{t},  
$$
with $\gamma_{j} \in \cc \setminus \{0\}$ and pairwise distinct $\phi_{j} \in {\mathbb C}$ with $\mathrm{Im}\, \phi_j \in [-\pi,\pi]$. As above let $z_j \coloneqq {\mathrm e}^{ \phi_j } $.
We use the notations $\tilde{f}$, $\tilde{z}_j$, $\tilde{\phi}_j$ and $\tilde{\gamma}_j$ for the exponential sum, the knots, frequencies and coefficients, respectively, reconstructed  by the algorithms.
We define the relative errors by the following formulas, as in \cite{PT2013}. Let 
$$
e(f) \coloneqq \frac{\max|f(t)-\tilde{f}(t)|}{\max|f(t)|}
$$
be  the relative error of the exponential sum, where the maximum is taken over the equidistant points in $[0, 2N-1]$ with step size $0.001$. Further, we denote with 
$$
e(\boldsymbol{z}) \coloneqq \frac{\max\limits_{j=1,\ldots,M}|z_j-\tilde{z}_j|}{\max\limits_{j=1,\ldots,M}|z_j|} ,
\quad  e(\bphi) \coloneqq  \frac{\max\limits_{j=1,\ldots,M}|\phi_j-\tilde{\phi}_j|}{\max\limits_{j=1,\ldots,M}|\phi_j|} ,
\quad\text{and}\quad e(\gamra) \coloneqq \frac{\max\limits_{j=1,\ldots,M}|\gamma_j-\tilde{\gamma}_j|}{\max\limits_{j=1,\ldots,M}|\gamma_j|}
$$
 the relative errors for the  knots $z_j$, the  frequencies $\phi_j$, and the coefficients $\gamma_j$.

\subsection{Experiments for exact data}

First, we consider three examples from \cite{PT2013}.  In all examples for MPM and ESPRIT we take the upper bound $L$ for the order of exponential sums $M$ to be $L=N$, where $2N$ is the number of given function samples, in order to achieve optimal results for these methods. For  MPM we use preconditioning, see Remark \ref{remMPM}. In all examples, the number of terms $M$ is also recovered by each method.

\begin{example}\label{ex1}

Let $M=6$, $\gamma_j=j$, $j=1,\ldots,6$, and
$$
{\mathbf z} = (z_j)_{j=1}^{6}=\left(
\begin{matrix}
0.9856-0.1628 \mathrm{i}\\
0.9856+0.1628 \mathrm{i} \\
0.8976-0.4305 \mathrm{i} \\
0.8976+0.4305 \mathrm{i}\\
0.8127-0.5690\mathrm{i} \\
0.8127+0.5690\mathrm{i}
\end{matrix}
\right),
$$
where $z_j=\mathrm{e}^{\phi_j}$ in (\ref{1.1}).
We reconstruct the signal from $2N$ samples $f(k)$, $k=0,\ldots,2N-1$, by using the algorithms MPM, ESPRIT, ESPIRA-I, and ESPIRA-II. We take $L=N$ and $\varepsilon=10^{-10}$ for  MPM and ESPRIT and $tol=10^{-13}$ for ESPIRA-I and ESPIRA-II. The results are presented in Table \ref{tb1} for $N=30$ and $N=50$. We see that  for this simple example, which is used very often in testing system identification methods \cite{Baz2000, PT2013}, all algorithms work very well.

\begin{table}[h!]
\centering
\caption{\small { Results of Example \ref{ex1}.} }
\label{tb1} 
\begin{tabular}{ |p{0.3cm}||p{1.5cm}| p{1.5cm}|  p{1.5cm}|  }
 \hline
$N$ &\ \ \ \  $e(f)$ &\ \ \ \  $e(\mathbf{z})$    &\ \ \ \   $e(\gamra)$ \\[2mm]
 \hline
  \multicolumn{4}{|c|}{MPM} \\
 \hline
30 &   $1.19$e--$14$  & $2.25$e--$15$  &  $4.23$e--$14$  \\[2mm]
50 &   $2.27$e--$14$ &  $3.44$e--$15$  &   $5.58$e--$14$  \\[2mm]
 \hline
 \end{tabular}
 \qquad 
\begin{tabular}{ |p{0.3cm}||p{1.5cm}|p{1.5cm}| p{1.5cm}|   } 
 \hline
$N$ &\ \ \ \  $e(f)$  & \ \ \ \ $e({\mathbf z})$  &\ \ \ \  $e(\gamra)$ \\[2mm]
 \hline
  \multicolumn{4}{|c|}{ESPRIT} \\
 \hline
30 &   $3.48$e--$14$  & $1.48$e--$15$  &  $1.21$e--$14$  \\[2mm]
50 &   $1.27$e--$14$  &  $1.88$e--$15$  &   $4.14$e--$14$  \\[2mm]
  \hline
  \end{tabular}
\begin{tabular}{ |p{0.3cm}||p{1.5cm}|p{1.5cm}| p{1.5cm}|  } 
\hline
  \multicolumn{4}{|c|}{ESPIRA-I} \\
 \hline
30 &   $7.51$e--$15$  &  $2.02$e--$15$  &   $3.73$e--$14$  \\[2mm]
50 &   $1.95$e--$14$  &  $8.16$e--$16$  &   $7.73$e--$14$  \\[2mm]
 \hline
  \end{tabular}
 \qquad 
\begin{tabular}{ |p{0.3cm}||p{1.5cm}|p{1.5cm}| p{1.5cm}|  } 
 \hline
  \multicolumn{4}{|c|}{ESPIRA-II} \\
 \hline
30 &   $3.88$e--$15$  & $1.88$e--$15$  &  $1.40$e--$14$  \\[2mm]
50 &   $9.43$e--$15$  &  $1.10$e--$15$  &   $3.21$e--$14$  \\[2mm]
 \hline
\end{tabular}
\end{table}

\end{example}

\begin{example}\label{ex2}

We consider an exponential sum (\ref{1.1}) of order $M=6$ given by the vectors
\begin{equation}\label{par}
\bphi = (\phi_{j})_{j=1}^{6}=\frac{\mathrm{i}}{1000}(7, \, 21, \, 200, \, 201, \, 53, \, 1000)^{T}, \quad  \gamra= (\gamma_{j})_{j=1}^{6}=(6, \, 5, \, 4, \, 3, \, 2, \, 1)^{T},
\end{equation}
that has been studied also in \cite{PT2013} with a different sample step size.
We again take $\varepsilon=10^{-10}$ and $L=N$ for MPM and ESPRIT and $tol=10^{-13}$ for 
ESPIRA-I and ESPIRA-II.
We use $2N$ values $f(k)$, $k=0,\ldots,2N-1$, for the reconstruction. The results  show that our new ESPIRA algorithms recover the frequencies very well even in case of clustered knots, see Table \ref{tb2} for $N=10, \, 20, \, 30$. The reconstruction results of ESPIRA-I and ESPIRA-II are comparable to those of MPM and ESPRIT.

\begin{table}[h!]
\centering
\caption{\small {Results of Example \ref{ex2}.}}
\label{tb2}
\begin{tabular}{ |p{0.3cm}||p{1.6cm}|p{1.6cm}| p{1.6cm}|  }
 \hline
$N$ &\ \ \ \  $e(f)$  & \ \ \ \ $e({\mathbf z})$  &\ \ \ \  $e(\gamra)$ \\[2mm]
 \hline
  \multicolumn{4}{|c|}{MPM} \\
 \hline
10  & $3.40$e--$15$    &$2.44$e--$05$ & $2.55$e--$02$\\[2mm]
20  & $2.52$e--$15$    &$3.11$e--$09$ & $3.16$e--$06$\\[2mm]
30  & $7.69$e--$15$    &$5.80$e--$11$ & $5.87$e--$08$\\[2mm]
 \hline
 \end{tabular}
 \qquad 
\begin{tabular}{ |p{0.3cm}||p{1.6cm}|p{1.6cm}| p{1.6cm}|  } 
 \hline
$N$ &\ \ \ \  $e(f)$  & \ \ \ \ $e({\mathbf z})$  &\ \ \ \  $e(\gamra)$ \\[2mm]
 \hline
  \multicolumn{4}{|c|}{ESPRIT} \\
 \hline
10  & $1.22$e--$14$    &$2.75$e--$05$ & $2.88$e--$02$\\[2mm]
20  & $1.58$e--$14$    &$1.28$e--$09$ & $1.29$e--$06$\\[2mm]
30  & $6.87$e--$14$    &$1.16$e--$10$ & $1.16$e--$07$\\[2mm]
 \hline
  \end{tabular}
\begin{tabular}{ |p{0.3cm}||p{1.6cm}|p{1.6cm}| p{1.6cm}|  } 
  \hline
  \multicolumn{4}{|c|}{ESPIRA-I} \\
 \hline
10  & $4.19$e--$15$    &$1.59$e--$05$ & $2.51$e--$02$\\[2mm]
20  & $8.10$e--$15$    &$5.98$e--$10$ & $1.04$e--$06$\\[2mm]
30  & $4.04$e--$15$    &$1.10$e--$11$ & $6.25$e--$08$\\[2mm]
  \hline
   \end{tabular} \qquad 
\begin{tabular}{ |p{0.3cm}||p{1.6cm}|p{1.6cm}| p{1.6cm}|  } 
  \hline
\multicolumn{4}{|c|}{ESPIRA-II} \\
 \hline
10  & $1.33$e--$14$    &$7.71$e--$06$ & $1.00$e--$02$\\[2mm]
20  & $1.17$e--$14$    &$5.35$e--$09$ & $6.98$e--$06$\\[2mm]
30  & $1.45$e--$14$    &$3.54$e--$11$ & $4.19$e--$08$\\[2mm]
\hline
\end{tabular}
\end{table}
\end{example}

\begin{example}\label{ex3}
Next we consider an exponential sum (\ref{1.1}) of order $M=6$ given by the vectors
$$
\bphi= (\phi_{j})_{j=1}^{6} =\frac{\mathrm{i}}{1000}(200, \, 201, \, 202, \, 203, \, 204, \, 205)^{T}, \quad \gamra= (\gamma_{j})_{j=1}^{6} =(6, \, 5, \, 4, \, 3, \, 2, \, 1)^{T}.
$$
In this case all frequencies are clustered. 
We take $2N$ values $f(k)$, $k=0,\ldots,2N-1$, and choose $\varepsilon=10^{-10}$ for  MPM and ESPRIT and $tol=10^{-13}$ for ESPIRA-I and ESPIRA-II, see Table 3 for $N=400$, $500$, $600$.

\begin{table}[h!]
\centering
  \caption{\small  Results of Example \ref{ex3}. }
\label{tb3} 
\begin{tabular}{ |p{0.5cm}||p{1.6cm}|p{1.6cm}| p{1.6cm}|  } 
  \hline
$N$ &\ \ \ \  $e(f)$  & \ \ \ \ $e(\bphi)$  &\ \ \ \  $e(\gamra)$ \\[2mm]
 \hline
  \multicolumn{4}{|c|}{MPM} \\
 \hline
 400 &   $1.66$e--$13$  & $7.92$e--$06$& $6.32$e--$04$\\[2mm]
500 &   $8.13$e--$14$  & $8.85$e--$07$ & $7.71$e--$05$ \\[2mm]
600 &   $2.12$e--$13$  & $3.70$e--$08$ &$3.50$e--$06$ \\[2mm]
 \hline
 \end{tabular}
 \quad 
\begin{tabular}{ |p{0.5cm}||p{1.6cm}|p{1.6cm}| p{1.6cm}|  } 
 \hline
$N$ &\ \ \ \  $e(f)$  & \ \ \ \ $e(\bphi)$  &\ \ \ \  $e(\gamra)$ \\[2mm]
 \hline
  \multicolumn{4}{|c|}{ESPRIT} \\
 \hline
400 &   $4.09$e--$13$  & $7.26$e--$04$ & $7.28$e--$02$  \\[2mm]
500 &   $2.01$e--$13$ & $2.35$e--$05$ & $2.04$e--$03$  \\[2mm]
600 & $6.51$e--$13$  & $4.33$e--$06$ & $4.01$e--$04$  \\[2mm]
 \hline
  \end{tabular}
  \begin{tabular}{ |p{0.5cm}||p{1.6cm}|p{1.6cm}| p{1.6cm}|  }
  \hline
  \multicolumn{4}{|c|}{ESPIRA-I} \\
 \hline
400 &   $6.95$e--$14$  & $2.05$e--$05$ & $1.68$e--$03$  \\[2mm]
500 &   $6.75$e--$14$  & $6.70$e--$07$ & $6.09$e--$05$  \\[2mm]
600 &   $8.67$e--$14$  & $7.38$e--$07$ & $6.72$e--$05$  \\[2mm]
 \hline
 \end{tabular}
 \quad 
\begin{tabular}{ |p{0.5cm}||p{1.6cm}|p{1.6cm}| p{1.6cm}|  } 
 \hline
  \multicolumn{4}{|c|}{ESPIRA-II} \\
 \hline
400 &   $2.87$e--$13$  & $7.40$e--$06$ & $8.39$e--$04$  \\[2mm]
500 &   $8.85$e--$13$ & $5.95$e--$06$ & $5.35$e--$04$  \\[2mm]
600 & $3.22$e--$13$  & $1.00$e--$06$ & $9.05$e--$05$  \\[2mm]
 \hline
  \end{tabular}
\end{table}

Comparing the results in Table \ref{tb3} we see that the two ESPIRA algorithms work equally well as MPM and slightly better than ESPRIT.
However, we want to mention at this point that MPM and EPRIT are much more expensive than the two ESPIRA methods. For both, MPM and ESPRIT it is essential to choose $L=N$, which means in this case, that the SVD with computational costs of ${\mathcal O}(N^{3})$ needs to be performed, while 
ESPIRA-I and ESPIRA-II have computational costs of ${\mathcal O}(N \, (M^{3}+\log N))$ where here $M=6$. If we take $L\ll N$ for MPM to reduce the computational costs, the recovery results get worse. For example, MPM achieves with $N=600$ and $L=100$ only the errors $e(f)= 1.07$e--$09$, $e(\bphi)=3.69$e--$04$ and $e(\bgamma) = 3.54$e--$02$ instead of those in Table \ref{tb3}.
This observation  impressively shows the importance of preconditioning  achieved by the greedy choice of the index set $S_{M+1}$ (or $S_{M}$) in ESPIRA-I and ESPIRA-II, which leads  to essentially better condition numbers  of the obtained tall Loewner matrix ${\mathbf L}_{2N-M,M}$ compared to  the tall Hankel matrix ${\mathbf H}_{2N-M,M}(0)$. For $N=600$, the condition  number of ${\mathbf L}_{2N-M,M}(0)$ is lowered by a factor $10.000$ compared to the condition number of ${\mathbf H}_{2N-M,M}$.
\end{example}

\subsection{Experiments for noisy data}

In this subsection we consider noisy data where we assume that the given data $f(k)$ from (\ref{1.1}) is corrupted with additive noise, i.e., the measured data is of the form 
$$ y_{k} = f(k) + n_{k}, \qquad k=0, \ldots , 2N-1,$$
where  $n_{k}$ are assumed to be either i.i.d.\ random  variables drawn from a standard normal distribution with mean value $0$ or from a uniform distribution with mean value $0$.  We will compare the performance of MPM and ESPIRA-II.

\begin{example}\label{exuni}
Let $f$ be as in (\ref{1.1}) with $M=8$, $\gamra=(4,5,4,3,2,1,2,3)$, and 
$$\bphi =\frac{1}{1000} ( 11\mathrm{i}, 21 \mathrm{i}, 23 \mathrm{i}, 203\mathrm{i},205 \mathrm{i}, 279  \mathrm{i}, 553 \mathrm{i}, 1000\mathrm{i})^{T}.$$
The considered knots $z_{j}={\mathrm e}^{\phi_{j}}$ are all on the unit circle  and we have two clusters with $\min|\mathrm{Re}(z_{j} - z_{k}) | = \cos(\frac{42\pi}{1000}) - \cos(\frac{46\pi}{1000}) = 1.73$e--$03$ and $\min|\mathrm{Im}(z_{j} - z_{k}) | = \sin(\frac{406\pi}{1000}) - \sin(\frac{410 \pi}{1000}) = 3.58$e--$03$. 
We employ $2N$ noisy data values $y_{k}=f(k) + n_{k}$, $k=0, \ldots, 2N-1$, where the random noise
$(n_{k})_{k=0}^{2N-1}$ is generated in \textsc{Matlab} by 
\texttt{20*(rand(2N,1)-0.5)}, i.e., we take uniform noise in the interval $[-10,10]$.
We compute $10$ iterations and give the reconstruction errors for MPM and for ESPIRA-II in Table \ref{tab_noise1}.
In this case, the average SNR value of the given noisy data is 3.66, the minimal SNR is 3.44 and the maximal SNR is 3.96. The average PSNR is $12.38$.
For MPM we use $N=L=600$ as well as $N=L=800$, for ESPIRA-II we take $N=600$ and $N=800$. In both cases we assume that $M=8$ is known beforehand and we modify the algorithms slightly. For MPM we compute the QR decomposition in step 1 of Algorithm \ref{algMPM} but do not compute the numerical rank of ${\mathbf R}_{2N-L,L+1}$ and instead take the known $M$ in step 2. For ESPIRA-II we set $jmax=M+1$ in Algorithm \ref{algESP2}.

 ESPIRA-II strongly outperforms MPM.
As it can be seen in Table \ref{tab_noise1}, the ESPIRA-II algorithm provides very accurate estimates for the 8 knots, while MPM can no longer be used to obtain good parameter estimates.
In Figure \ref{plotnoise1}, we present the reconstruction result of the original function from the noisy data for  MPM  and for ESPIRA-II.
We show here only the interval $[0,50]$ with noisy input data points (red dots), the original function $f$ (black) and the reconstructed function (red and blue) for MPM and ESPIRA-II, respectively. 

\begin{table}[h!]
\centering
 \caption{\small Reconstruction error for noisy data with additive i.i.d.\ noise drawn from uniform distribution in $[-10,10]$ with average SNR of 3.66.}
  \label{tab_noise1}
\begin{tabular}{ |p{1.5cm}||p{1.6cm}|p{1.6cm}| p{1.6cm}||p{1.6cm}|p{1.6cm}| p{1.6cm}| }
 \hline
   & \multicolumn{3}{|c||}{MPM} & \multicolumn{3}{|c|}{ESPIRA-II}\\
   \hline
$N =600$ & min & max & average & min & max & average\\[2mm]
 \hline
$e(\mathrm{Re}({\mathbf z}))$ & $4.64$e--$02$    &$1.81$e--$01$ & $9.34$e--$02$ &   $2.30$e--$04$    &$9.13$e--$04$ & $4.26$e--$04$ \\[2mm]
$e(\mathrm{Im}({\mathbf z}))$  & $3.55$e--$01$    &$8.12$e--$01$ & $6.12$e--$01$ & $2.89$e--$04$    &$9.22$e--$04$ & $4.67$e--$04$ \\[2mm]
$e(\gamra)$ & $1.12$e--$00$    & $4.36$e--$00$ & $2.76$e--$00$ & $8.70$e--$02$ & $2.19$e--$01$  & $1.36$e--$01$\\[2mm]
$e(f)$ & $8.29$e--$01$    & $1.10$e--$00$ & $9.75$e--$01$ & $1.01$e--$01$ & $6.48$e--$01$  & $5.78$e--$01$\\[2mm]
 \hline
 \end{tabular}
 \begin{tabular}{ |p{1.5cm}||p{1.6cm}|p{1.6cm}| p{1.6cm}||p{1.6cm}|p{1.6cm}| p{1.6cm}| }
 \hline
   & \multicolumn{3}{|c||}{MPM} & \multicolumn{3}{|c|}{ESPIRA-II}\\
   \hline
$N =800$ & min & max & average & min & max & average\\[2mm]
 \hline
$e(\mathrm{Re}({\mathbf z}))$ & $4.10$e--$02$    &$2.23$e--$01$ & $1.13$e--$01$ &   $6.88$e--$05$    &$4.85$e--$04$ & $2.48$e--$04$ \\[2mm]
$e(\mathrm{Im}({\mathbf z}))$  & $3.62$e--$01$    &$7.67$e--$01$ & $5.32$e--$01$ & $9.56$e--$05$    &$4.33$e--$04$ & $2.15$e--$04$ \\[2mm]
$e(\gamra)$ & $1.19$e--$00$    & $3.25$e--$00$ & $2.65$e--$00$ & $6.70$e--$02$ & $1.36$e--$01$  & $9.30$e--$02$\\[2mm]
$e(f)$ & $7.67$e--$01$    & $1.12$e--$00$ & $1.02$e--$00$ & $9.70$e--$02$ & $6.70$e--$01$  & $5.88$e--$01$\\[2mm]
 \hline
 \end{tabular}
\end{table}

\begin{figure}[h!]
\includegraphics[width=7.0cm]{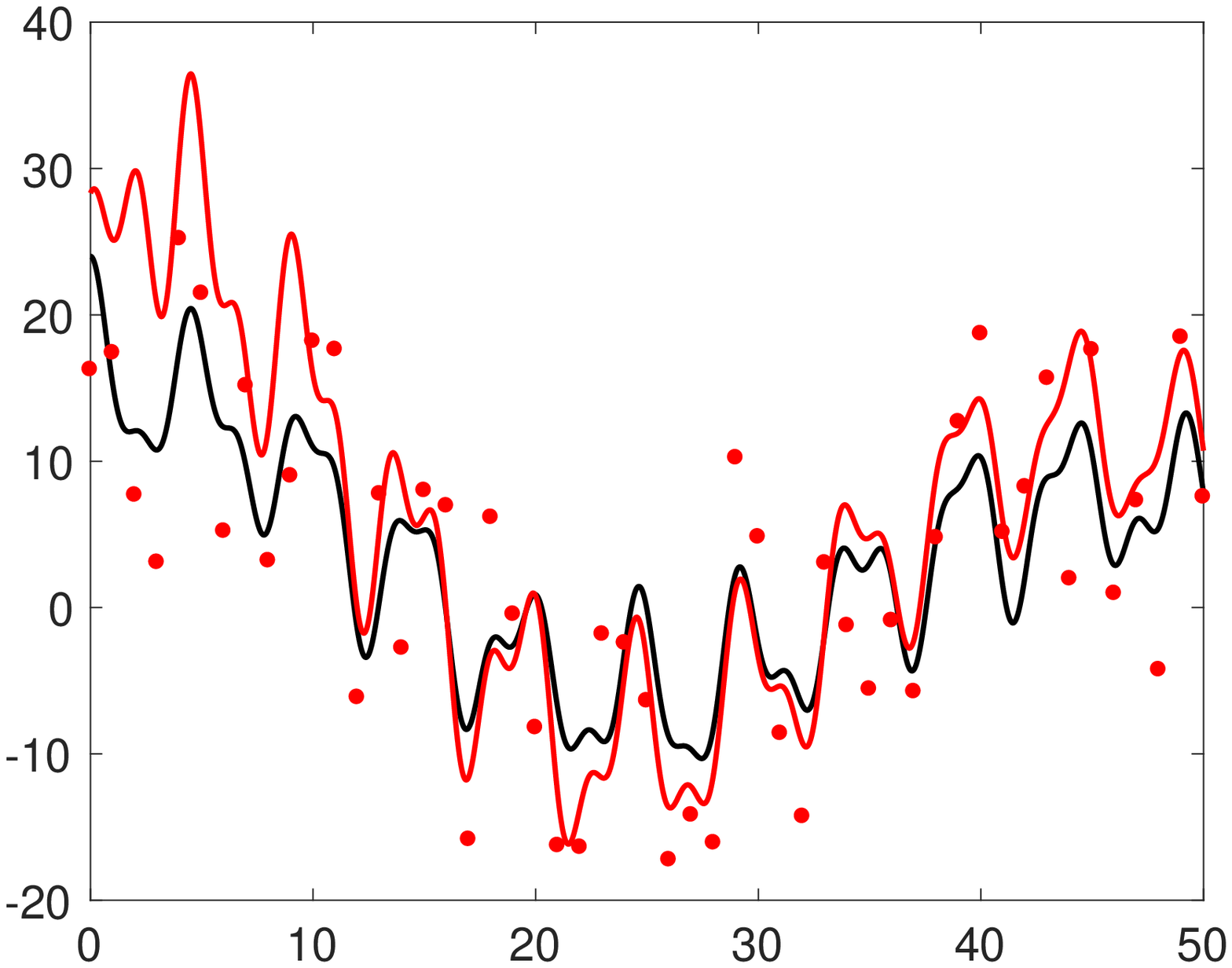}
\includegraphics[width=7.0cm]{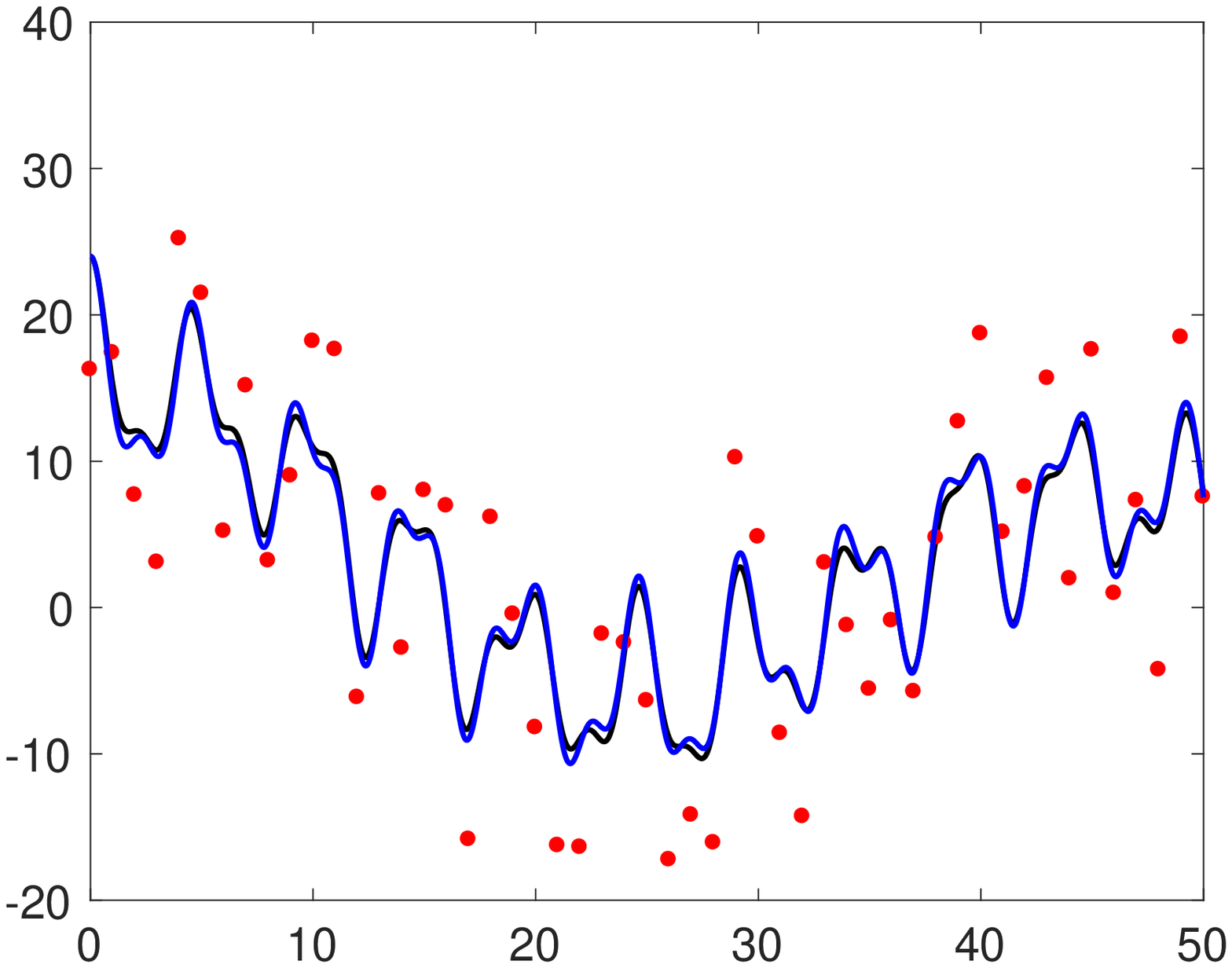}
\centering
\caption{\small Plot of the real part of the original function $f(t)$ (black), the noisy data (red dots), and the achieved reconstructions (of the real part) restricted to the interval $[0, 50]$. Left: Reconstruction by MPM (red). Right: Reconstruction by ESPIRA-II (blue). For this Example \ref{exuni}, we have $N=800$ and uniform noise in $[-10,10]$ with SNR 3.64.}
\label{plotnoise1}
\end{figure}

\end{example}

\begin{example}\label{exgauss}
We employ the same signal  $f$  as in Example \ref{exuni} with $M=8$, but this time the noise vector
$(n_{k})_{k=0}^{2N-1}$ is generated from a normal distribution with mean value zero using the \textsc{Matlab} function
\texttt{randn}. We take \texttt{sigma*randn(2N)}, where \texttt{sigma = 0.5*std(f(:))} and \texttt{std} denotes the standard deviation of $(f(k))_{k=0}^{2N-1}$.
We compute $10$ iterations with an average SNR of 6.07 (minimum SNR 5.80, maximum SNR 6.35), this corresponds to an average PSNR of 14.76.
The reconstruction errors for MPM and for ESPIRA-II are presented in Table \ref{tab_noise2}.
For MPM we use $N=L=600$ as well as $N=L=800$, for ESPIRA-II we take $N=600$ and $N=800$. In both cases we also assume that $M=8$ is known beforehand.
As before for uniform noise, it can be seen in in Table \ref{tab_noise2} that  ESPIRA-II outperforms MPM and provides very accurate estimates for the 8 knots, while MPM does not.
In Figure \ref{plotnoise2}, we present the reconstruction of the original function from the noisy data for  MPM and for ESPIRA-II
on the interval $[0,50]$, where the given noisy input data is plotted as red dots, the real part of the original function $f$ is plotted in black, and the reconstructed (real part) of the function  for MPM and ESPIRA-II are presented in red and blue, respectively.
As before ESPIRA-II strongly outperforms the results of MPM.

\begin{table}[h!]
\centering
\caption{Reconstruction error for noisy data with additive i.i.d.\ noise drawn from standard normal distribution with average SNR of 6.07.}
  \label{tab_noise2}
\begin{tabular}{ |p{1.5cm}||p{1.6cm}|p{1.6cm}| p{1.6cm}||p{1.6cm}|p{1.6cm}| p{1.6cm}| }
 \hline
   & \multicolumn{3}{|c||}{MPM} & \multicolumn{3}{|c|}{ESPIRA-II}\\
   \hline
$N =600$ & min & max & average & min & max & average\\[2mm]
 \hline
$e(\mathrm{Re}({\mathbf z}))$ & $1.03$e--$02$    &$1.63$e--$01$ & $7.18$e--$02$ &   $1.04$e--$04$    &$3.61$e--$04$ & $2.31$e--$04$ \\[2mm]
$e(\mathrm{Im}({\mathbf z}))$  & $4.44$e--$02$    &$7.73$e--$01$ & $3.92$e--$01$ & $1.52$e--$04$    &$5.98$e--$04$ & $2.87$e--$04$ \\[2mm]
$e(\gamra)$ & $6.48$e--$01$    & $3.79$e--$00$ & $2.22$e--$00$ & $5.43$e--$02$ & $1.18$e--$01$  & $8.94$e--$02$\\[2mm]
$e(f)$ & $5.50$e--$01$    & $9.28$e--$01$ & $8.76$e--$01$ & $6.19$e--$02$ & $6.63$e--$01$  & $5.68$e--$01$\\[2mm]
 \hline
 \end{tabular}
 \begin{tabular}{ |p{1.5cm}||p{1.6cm}|p{1.6cm}| p{1.6cm}||p{1.6cm}|p{1.6cm}| p{1.6cm}| }
 \hline
   & \multicolumn{3}{|c||}{MPM} & \multicolumn{3}{|c|}{ESPIRA-II}\\
   \hline
$N =800$ & min & max & average & min & max & average\\[2mm]
 \hline
$e(\mathrm{Re}({\mathbf z}))$ & $1.36$e--$02$    &$1.54$e--$01$ & $5.71$e--$02$ &   $7.83$e--$05$    &$4.10$e--$04$ & $2.16$e--$04$ \\[2mm]
$e(\mathrm{Im}({\mathbf z}))$  & $9.44$e--$02$    &$4.74$e--$01$ & $2.19$e--$01$ & $6.02$e--$05$    &$3.34$e--$04$ & $1.79$e--$04$ \\[2mm]
$e(\gamra)$ & $5.75$e--$01$    & $3.75$e--$00$ & $1.31$e--$00$ & $5.10$e--$02$ & $8.56$e--$02$  & $7.15$e--$02$\\[2mm]
$e(f)$ & $5.18$e--$01$    & $9.80$e--$01$ & $8.91$e--$01$ & $4.47$e--$02$ & $6.51$e--$01$  & $5.81$e--$01$\\[2mm]
 \hline
 \end{tabular}
\end{table}

\begin{figure}[h!]
\includegraphics[width=7.0cm]{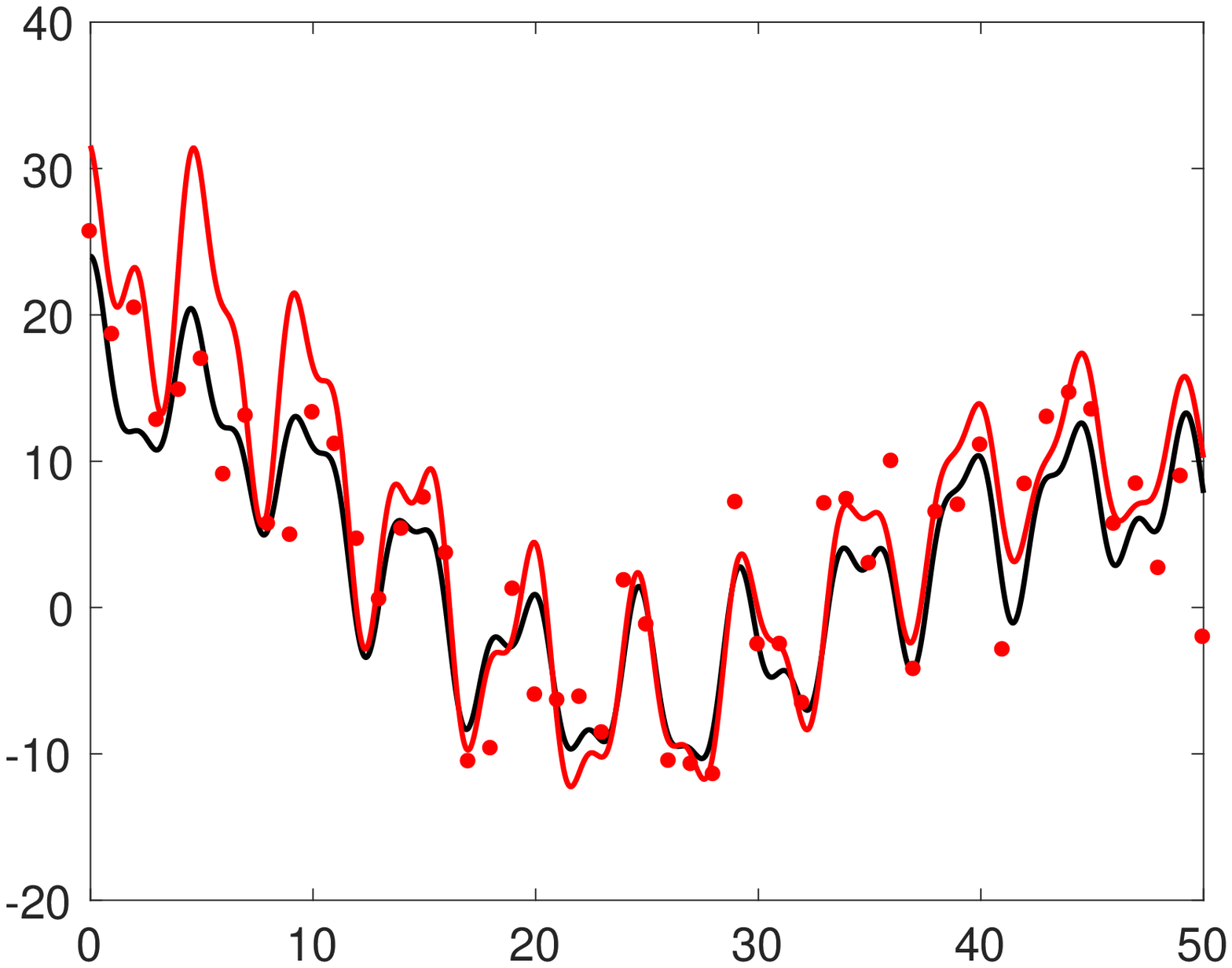}
\includegraphics[width=7.0cm]{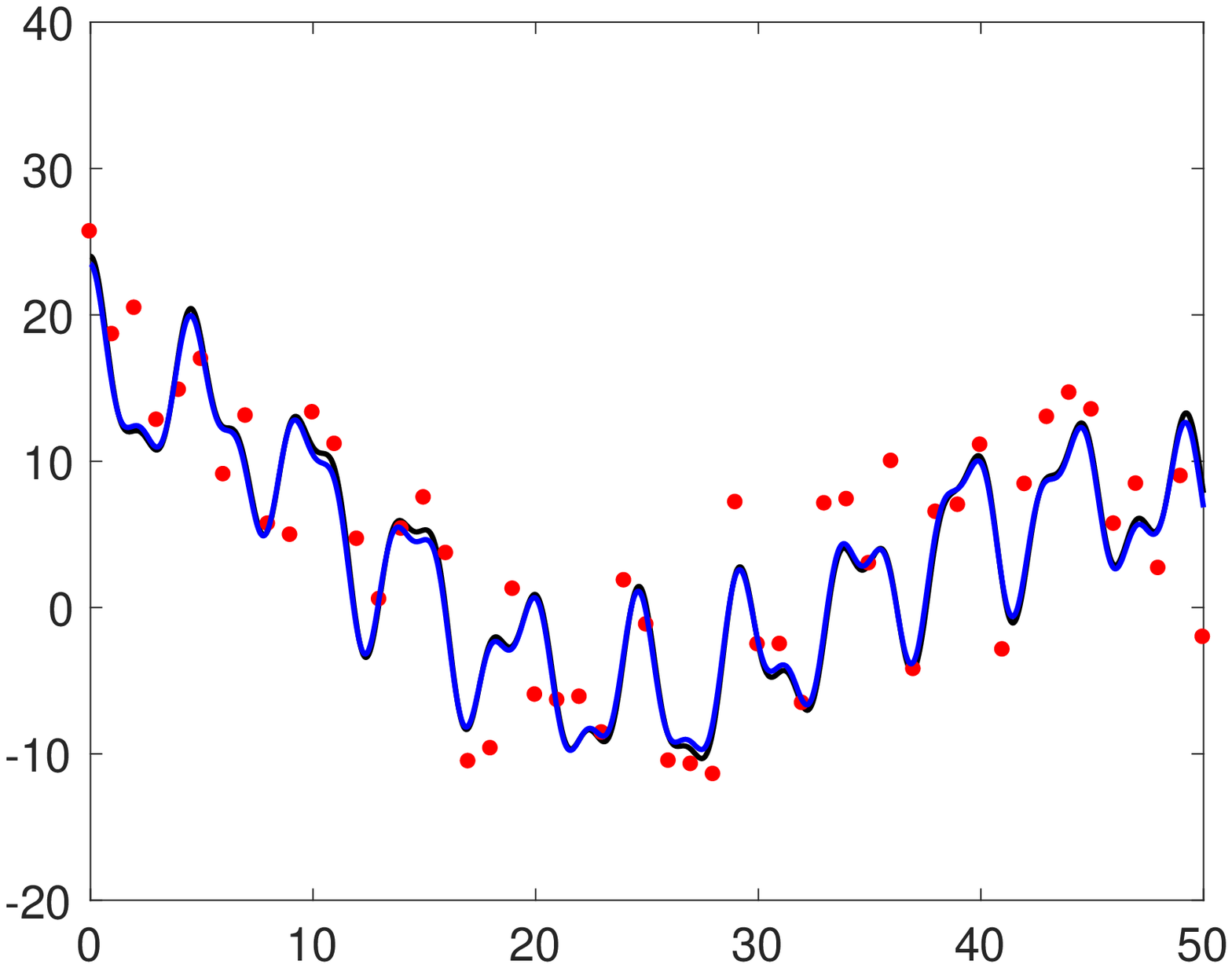}
\centering
\caption{\small Plot of the real part of the original function $f(t)$ (black), the noisy data (red dots) and the achieved reconstructions (of the real part) restricted to the interval $[0, 50]$. Left: Reconstruction by MPM (red). Right: Reconstruction by ESPIRA-II (blue) for Example \ref{exgauss} with  $N=800$ and normally distributed noise with mean $0$ and SNR 5.98.}
\label{plotnoise2}
\end{figure}
\end{example}

Since our new method ESPIRA-II strongly improves upon the estimation of parameters from noisy data 
compared to MPM, we conjecture that the new method is indeed statistically consistent while MPM and ESPRIT are not, see also \cite{BM86, Osborne95, ZP18}.

\section{Application of ESPIRA for function approximation}

We consider three examples for approximation of smooth functions by short exponential sums and show that ESPIRA outperforms ESPRIT and MPM also for approximation. 
The algorithms are implemented in \textsc{Matlab} and use IEEE standard floating point arithmetic with double precision.

\begin{example}\label{ex1/x}

We consider the approximation of $g(t)=\frac{1}{1+t} $ for $t \in [0,1]$.
We employ the equidistant samples $g(\frac{k}{2N}) = (1+\frac{k}{2N})^{-1}$, $k=0, \ldots , 2N-1$, for the approximation of $g$ by a short exponential sum ${f}$ of the form (\ref{1.1}) with 
$M=5$ terms. We use $N=60$, i.e., we take $120$ equidistant function values.
The  maximum error in $L_\infty[0,1]$-norm, 
\begin{equation} \label{errcm}
\max_{t \in [0, 1]}\left|\frac{1}{1+t}-\sum\limits_{j=1}^{M} \gamma_j \mathrm{e}^{\phi_j t} \right|,
\end{equation}
obtained by the  4 algorithms (ESPIRA I and II, MPM and ESPRIT),  
 is shown in Figure \ref{errorcom} using a logarithmic scale.

\begin{figure}[h]
\includegraphics[width=7cm]{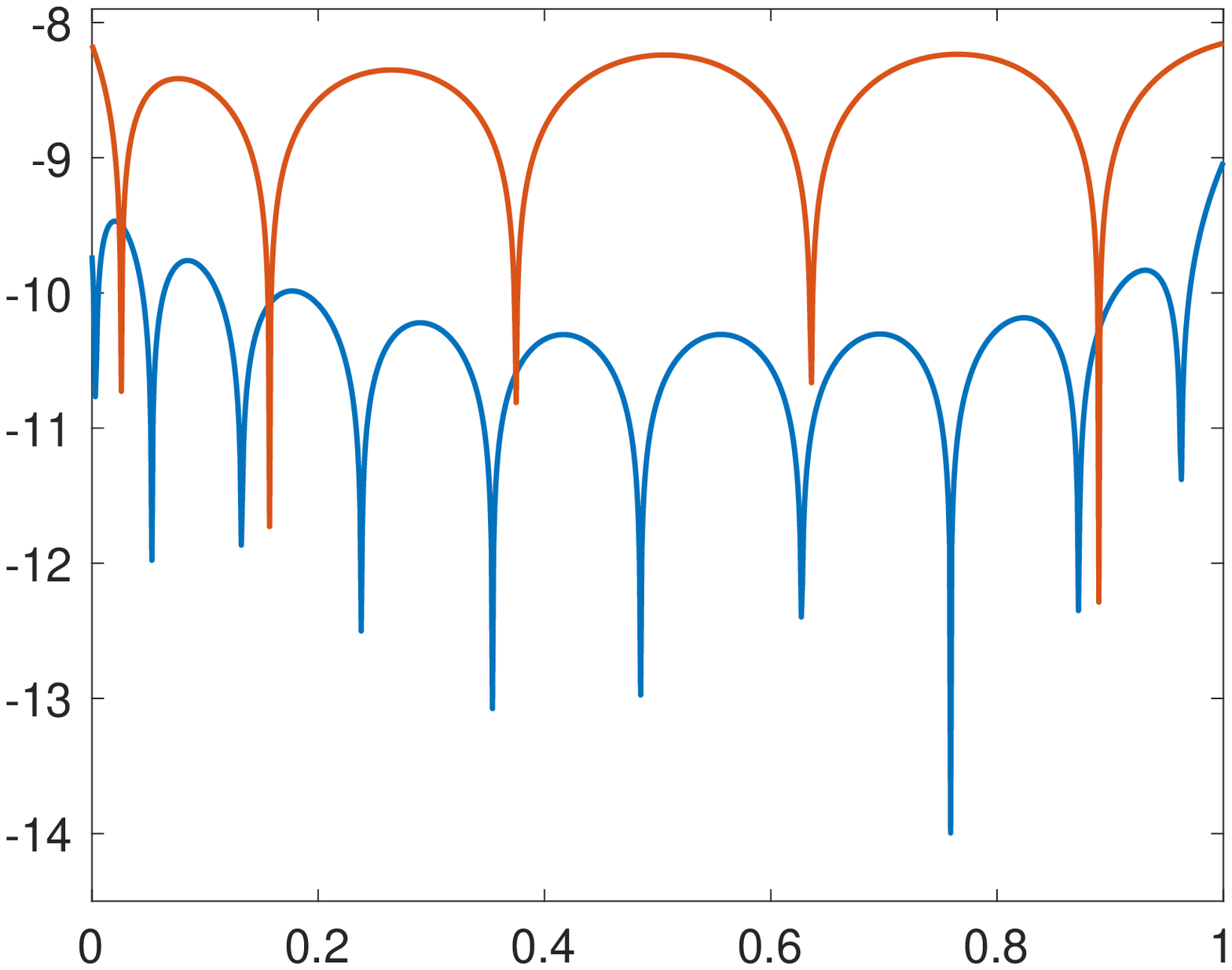} \qquad 
\includegraphics[width=7cm]{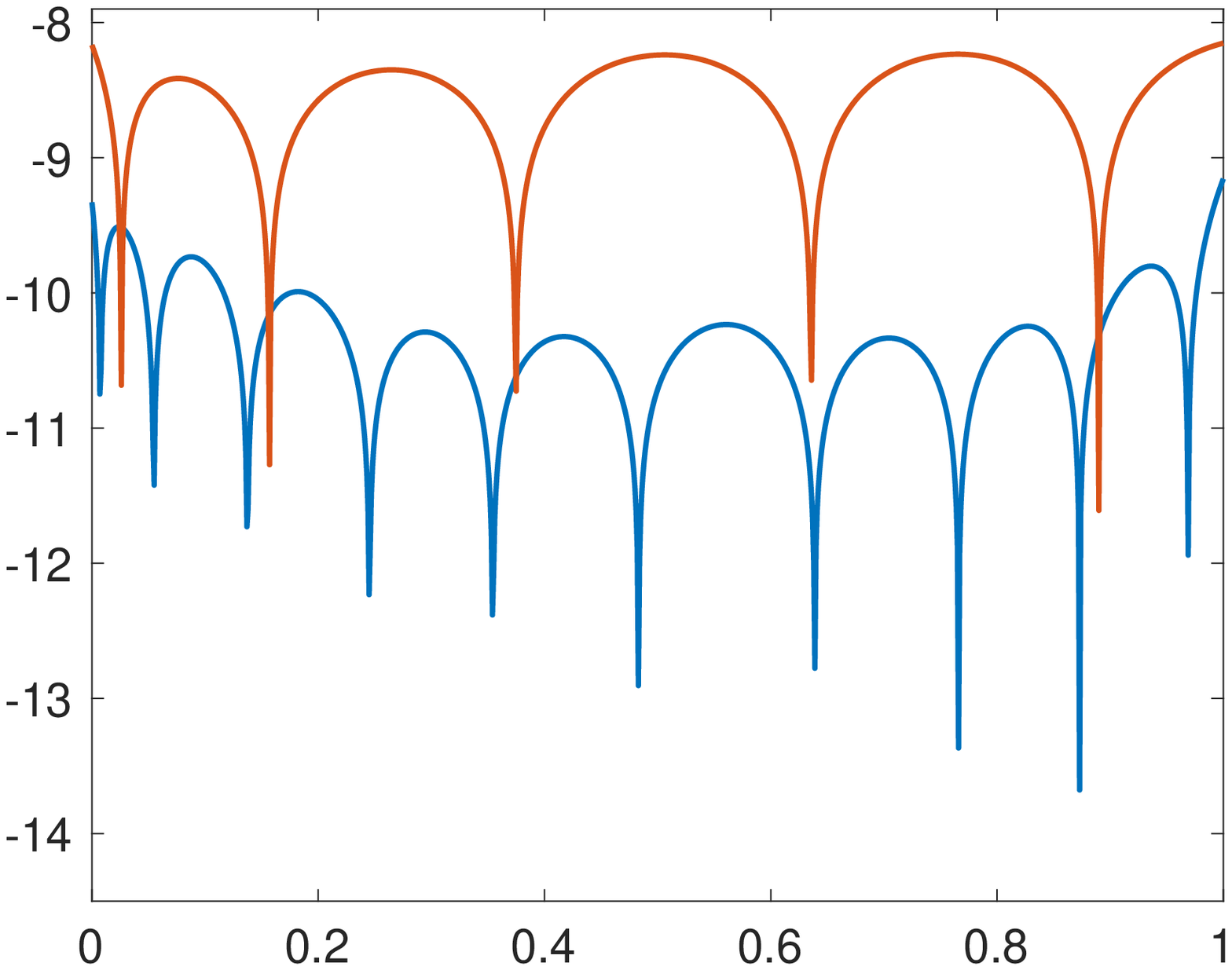}
\centering
\caption{\small The error in (\ref{errcm}) in logarithmic scale for $M=5$ for Left: ESPIRA-I (blue) and MPM (red), Right: ESPIRA-II (blue) and ESPRIT (red). }
\label{errorcom}
\end{figure}

For MPM and ESPRIT we employ the same input information as for ESPIRA with $120$ equidistant samples and choose $L=10$ as a bound for the order $M=5$ of the exponential sums and $\epsilon=10^{-11}$. For this setting ESPIRA gives  better approximation results.

For other comparable approximations of the function $f(t)= \frac{1}{1+t}$ on an interval $[0,R]$ with $R \ge 1$ or $[0, \infty)$  we refer to \cite{PT2013} and to \cite{Bbook,BH05,H2019}, where the Remez algorithm  is employed to obtain very good approximations of $f(t)$. The Remez algorithm provides slightly better results but requires a higher computational effort and (because of numerical instabilities) a computation with extended precision. The Remez algorithm directly aims at minimizing the error with respect to the maximum norm, while the nature of the ESPIRA algorithms is closer to achieving a small $2$-norm of the error of the given samples. The Remez algorithm is based on the requirement that the function $f(t)$ is available at any point on the desired interval, while ESPIRA works similarly as MPM and ESPRIT with a fixed number of given equidistant function values. 
However, for $M=5$ and $L=22$ we achieve with ESPIRA-II an absolute error of $1.96$e--$10$, while the Remez algorithm achieves with high precision computation $1.26$e--$10$, see \cite{H2019}, Table 1.

\end{example}

\begin{example}\label{exbessel}
Let us now consider the approximation of the Bessel function $J_0(100\pi t)$ on $t\in [0,1]$, see \cite{BM05, PT2013, Cu2020}.
Beylkin and Monzon \cite{BM05}, presented an algorithm for approximation of functions by short exponential sums which is related to Prony's method and to a finite version of AAK theory. The results in \cite{PT2013} apply to MPM and ESPRIT. In \cite{Cu2020}, a Prony-type matrix pencil method for cosine sums has been employed, where the numerical computations need a very high precision (up to 1200 digits) because of numerical instabilities.  
We would like to show that our ESPIRA algorithms provide approximation  results which are comparable to the results in  \cite{BM05}, while we  use IEEE standard floating point arithmetic with double precision and a numerical effort of ${\mathcal O}(N (M^{3}+\log N))$, where $N$ is the number of equidistant function samples and $M$ the number of terms in the exponential sum. 

The function $J_{0}$  is defined  by 
$$
J_0(t) \coloneqq \sum\limits_{k=0}^{\infty}\frac{(-1)^{k}}{(k!)^{2}} \, \left(\frac{t}{2} \right)^{2k}, \qquad  t \in {\mathbb R}.
$$
As in  \cite{BM05},  we approximate $J_0(100\pi t)$ on $t\in [0,1]$.
 It was shown in  \cite{BM05} that an error 
  \begin{equation}\label{j0apr}
\max_{t \in [0,1]}\left|J_0(100\pi t) - \sum\limits_{j=1}^{M} \gamma_j \mathrm{e}^{\phi_j t} \right| \lesssim 10^{-11}
\end{equation}
can be achieved with $M=28$  terms. For ESPIRA-I and ESPIRA-II we employ the $1030$  function values $J_0\left(100\pi \frac{\ell}{1030}\right)$, $\ell=0,\ldots,1029$, and construct the exponential sum of order 28 in (\ref{j0apr}) with the error $\varepsilon=8.52\cdot 10^{-12}$.
For comparison, we employ MPM and ESPRIT choosing the same number of $2N=1030$ equidistant function values,  $L=250$ as an upper bound for the order of the exponential sums, and $M=28$,  see Figure \ref{fig4}. 

\begin{figure}[h]
\centering
\includegraphics[width=6.5cm]{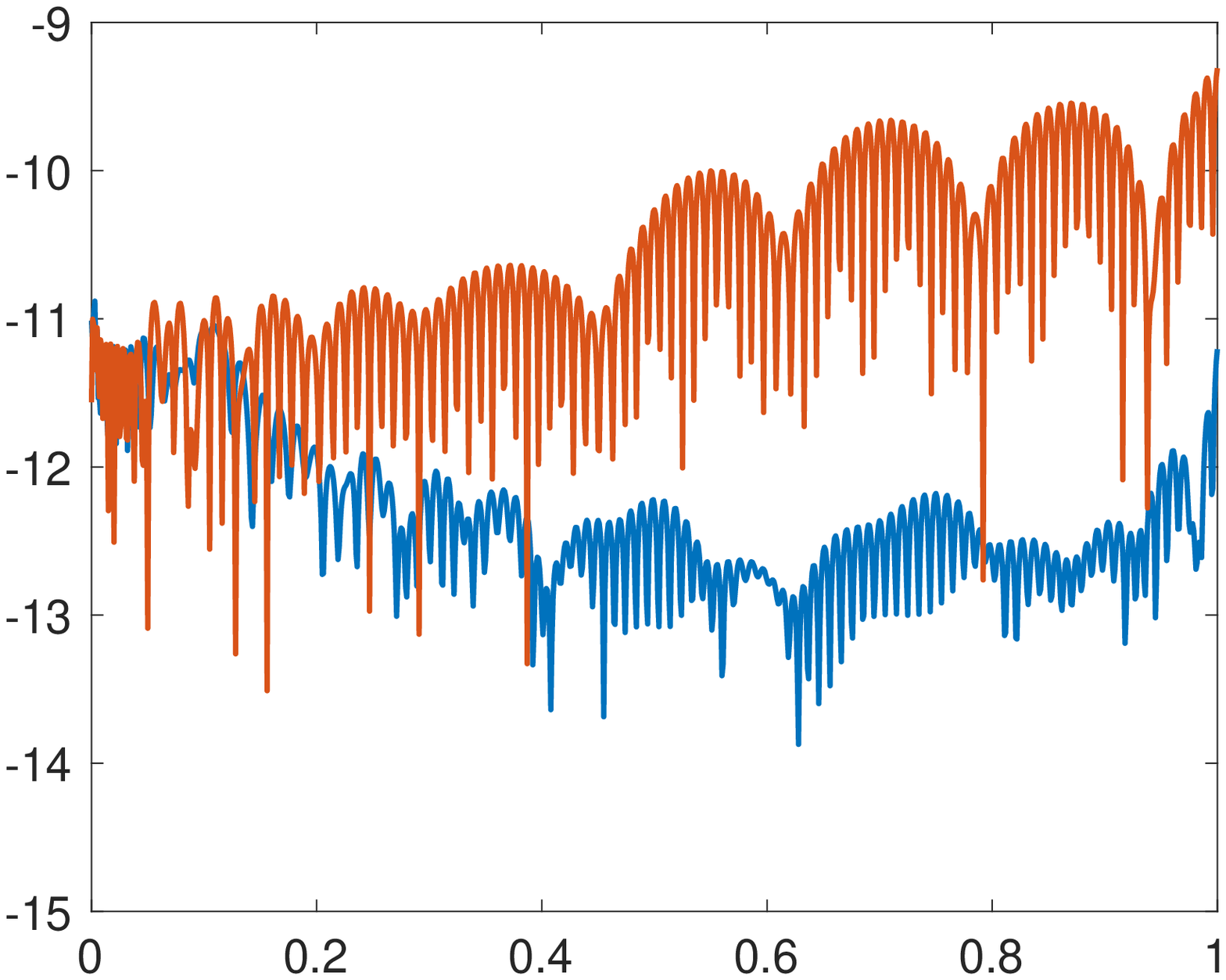} \qquad \qquad 
\includegraphics[width=6.5cm]{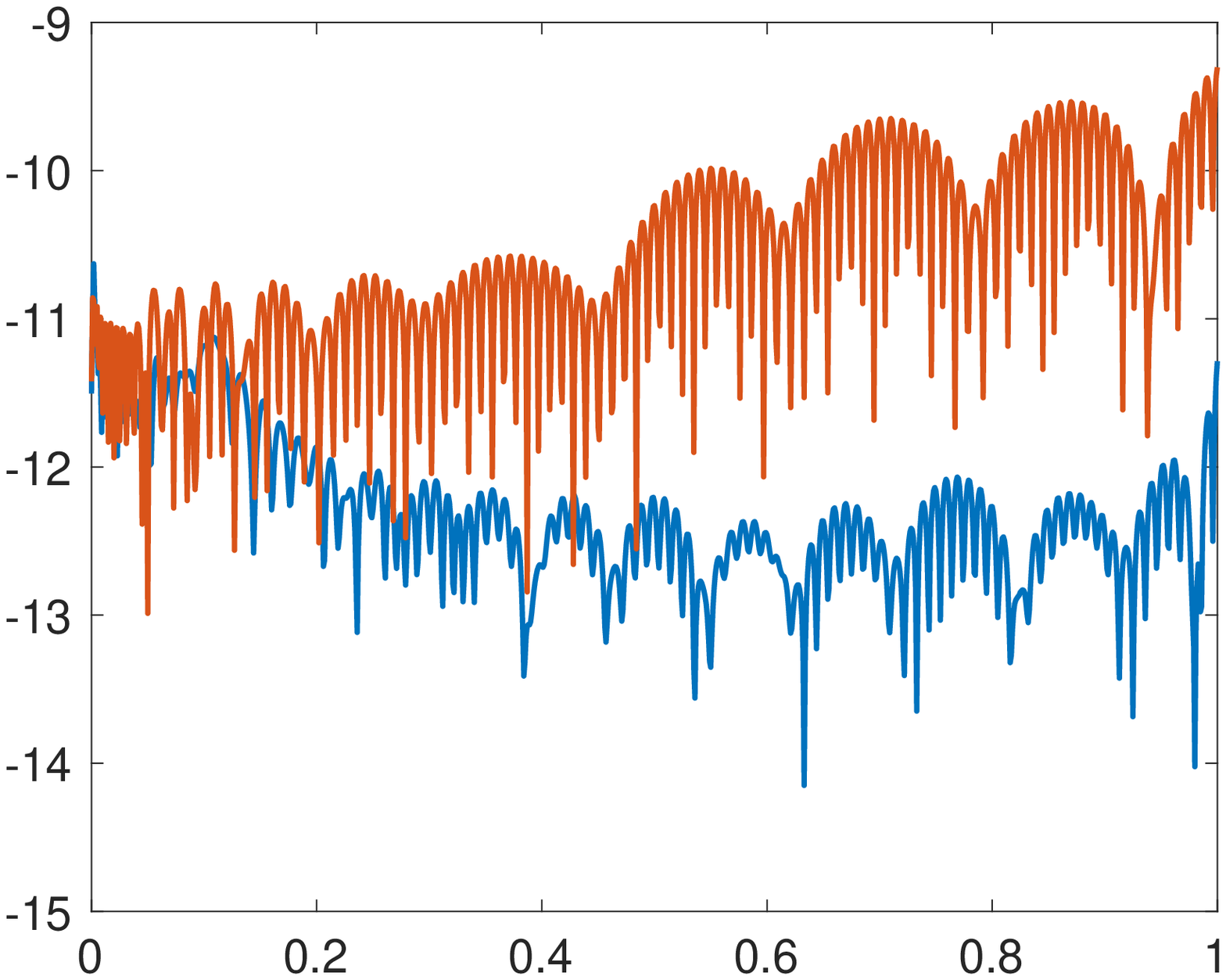}
\centering
\caption{\small The approximation error in (\ref{j0apr}) in logarithmic scale obtained using an exponential sum with $M=28$ terms. Left: with ESPIRA-I (blue) and MPM (red),
Right: with  ESPIRA-II (blue) and ESPRIT (red). }
\label{fig4}
\end{figure}

Similarly as observed with the algorithm in \cite{BM05} and with MPM in \cite{PT2013}, the computed locations of  complex normalized knots  $z_{j} = \mathrm{e}^{\phi_j /(100 \pi)}$ as well as of the corresponding coefficients $\gamma_{j}$, $j=1,\ldots,28$, computed by ESPIRA-I and ESPIRA-II appear on special contours, see Figure \ref{fig3}.

\begin{figure}[h]
\centering
\begin{subfigure}{0.35\textwidth}
\includegraphics[width=\textwidth]{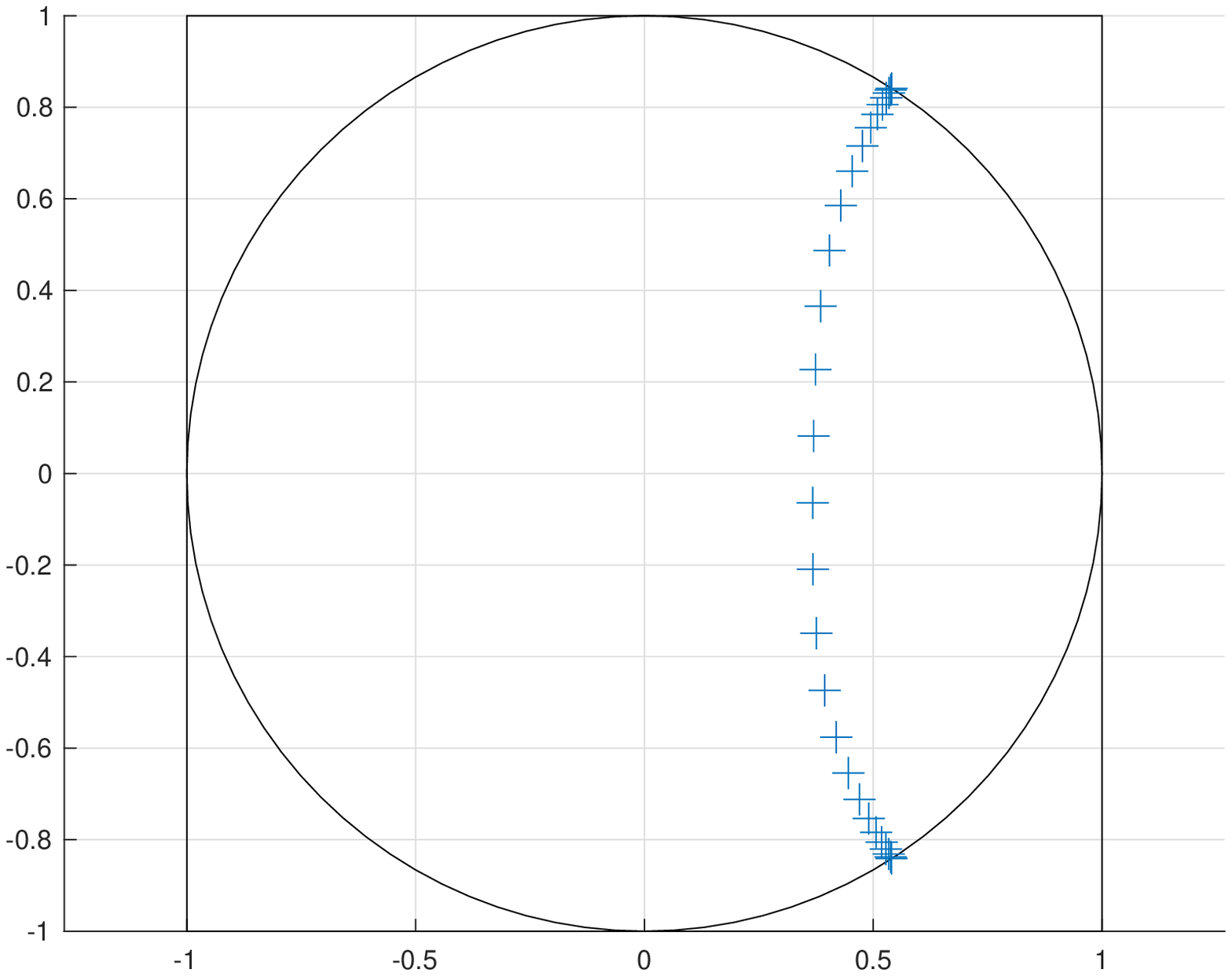}   
\end{subfigure}
\hspace*{-10mm}
\begin{subfigure}{0.35\textwidth}
\includegraphics[width=\textwidth]{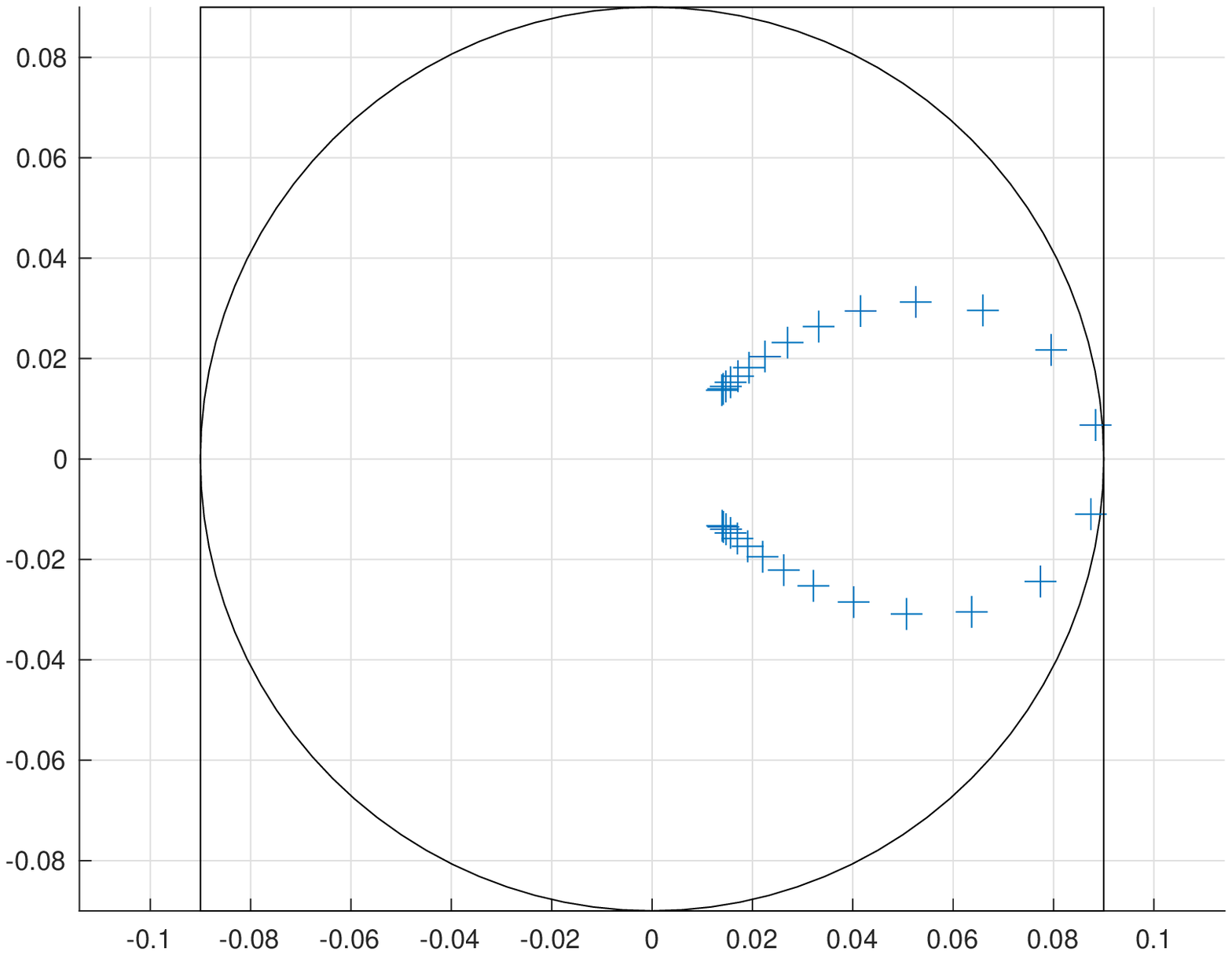}   
\end{subfigure}
\hspace*{-10mm}
\begin{subfigure}{0.35\textwidth}
\includegraphics[width=\textwidth]{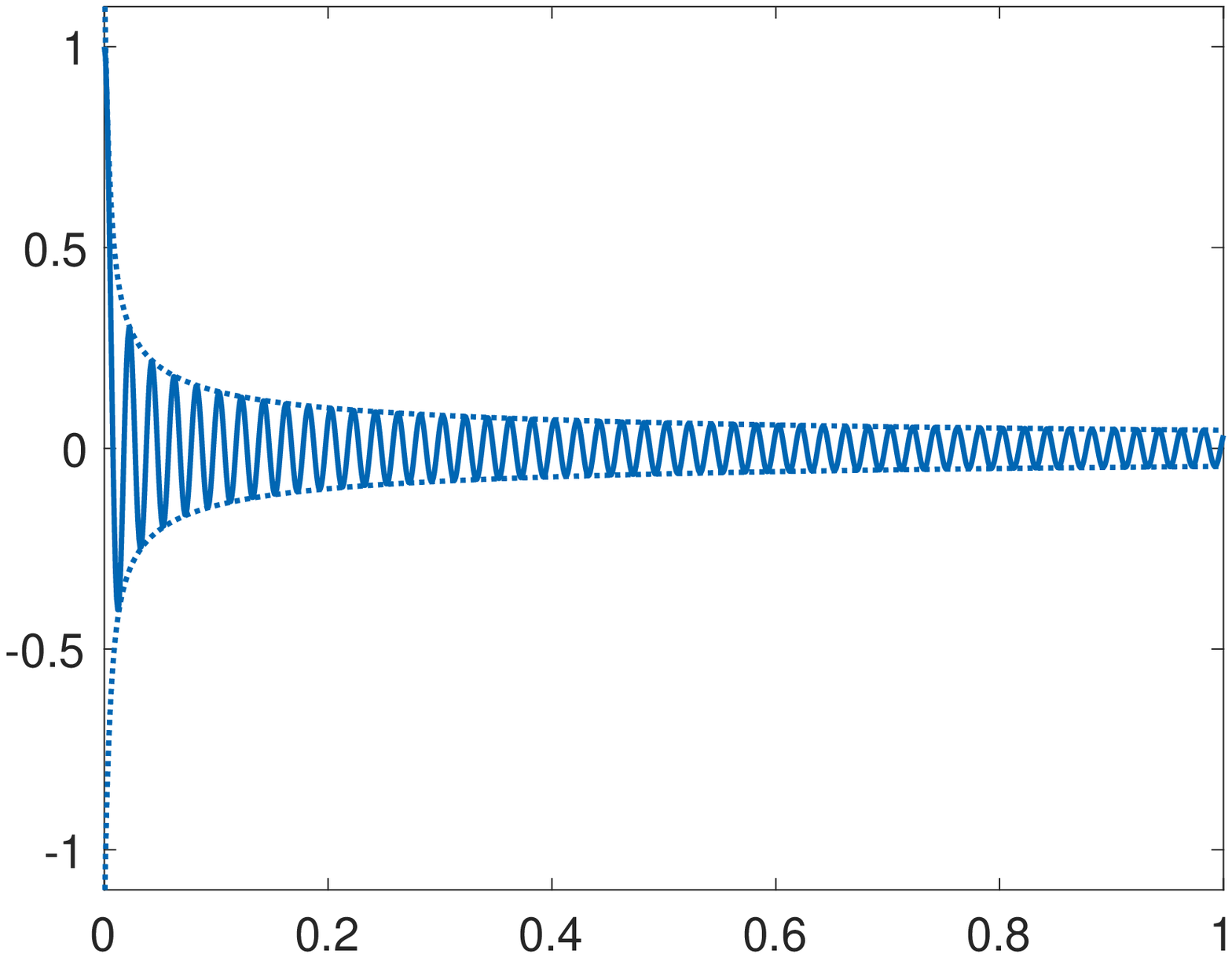}   
\end{subfigure}
\caption{\small Left: complex normalized knots $z_{j}=\mathrm{e}^{ \phi_j/(100\pi)}$ and Middle: complex weights $\gamma_{j}$ for the approximation of $J_0(100\pi t)$ by exponential sum of order $M=28$ on the segment $[0,1]$ obtained with ESPIRA-I. Right:
Graph of the Bessel function $J_0(100\pi t)$ together with its envelope functions of order 28 constructed with ESPIRA-I on the segment $[0,1]$.}
\label{fig3}
\end{figure}

As in \cite{BM05}, our algorithms can also be used for the construction of the decreasing envelope function of $J_0(100\pi t)$ that touches all local maxima of the Bessel function and of the  increasing envelope function that touches all local minima. Using the approximation in (\ref{j0apr}) we define 
$$
\mathrm{env}(t)=\sum\limits_{j=1}^{n} |\gamma_j| \mathrm{e}^{\mathrm{Re} \phi_j t}
$$
as the positive envelope of the Bessel function and $-\mathrm{env}(t)$ as its negative envelope. The result of ESPIRA-I is presented in Figure \ref{fig3} (right).

\end{example}

\begin{example}\label{exdiri}
Finally we consider one more example from \cite{BM05}, the periodic Dirichlet kernel
$$
D_n(t)=\frac{1}{2n+1}\sum\limits_{k=-n}^{n} \mathrm{e}^{2\pi i k t}=\frac{\sin ((2n+1)\pi t)}{(2n+1)\sin(\pi t)},
$$
of order 50 on the segment $[0,1]$,  see Figure \ref{dir}.

\begin{figure}[h]
\includegraphics[width=7.0cm]{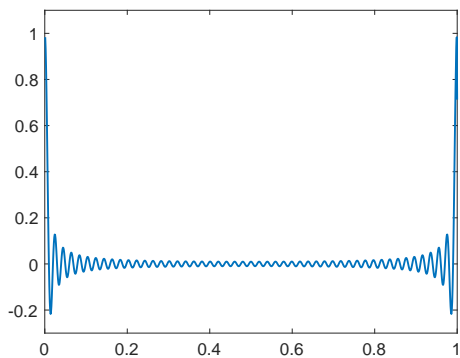} 
\centering
\caption{\small Dirichlet kernel $D_{50}(t)$ on the segment $[0,1]$ }
\label{dir}
\end{figure}

In \cite{BM05} it was shown that in order to achieve the approximation error
  \begin{equation}\label{dappr}
\max_{t \in [0,1]}\left|D_{50}(t) - \sum\limits_{j=1}^{M} \gamma_j \mathrm{e}^{\phi_j t} \right| \lesssim 10^{-8},
\end{equation}
it is enough to take $M=44$.
To catch the behavior of this periodic function by approximation with an exponential sum, the knots $z_{j}$ are located both inside and outside the unit circle. We construct an exponential sum of order 44 using 2000 samples $D_{50}(\ell/2000)$, $\ell=0,...,1999$, i.e., $N=1000$ and $L=N=1000$ using MPM and ESPRIT.  But for the found 44 knots $z_{j}$  the Vandermonde matrix is extremely ill-conditioned.

Similarly as for MPM and ESPRIT the approach in \cite{BM05} is based on the computation of the coefficient vector  $\bgamma=(\gamma_j)_{j=1}^{M}$ by solving the system $ {\mathbf V}_{2N,M} ({\mathbf z}) \, \bgamma = {\mathbf f} $ with the Vandermonde matrix ${\mathbf V}_{2N,M}({\mathbf z})$  defined in (\ref{V}). 
Therefore, the authors of \cite{BM05} employ the auxiliary function $G_{50}(t)$ that satisfies $D_{50}=G_{50}(t)+ G_{50}(1-t)$ and approximate $G_{50}(t)$ instead of $D_{50}(t)$ to achieve an approximation error of about $10^{-8}$, see \cite{BM05}, Figure 8.
Further, in our computations, the  extremely large condition  numbers of the Vandermonde matrices ($10^{119}$ for MPM and $10^{75}$ for ESPRIT)  imply that MPM and ESPRIT fail to construct an exponential sum of good approximation for the Dirichlet kernel, see Figure \ref{dirmpmes}.

\begin{figure}[h]
\includegraphics[width=6.5cm]{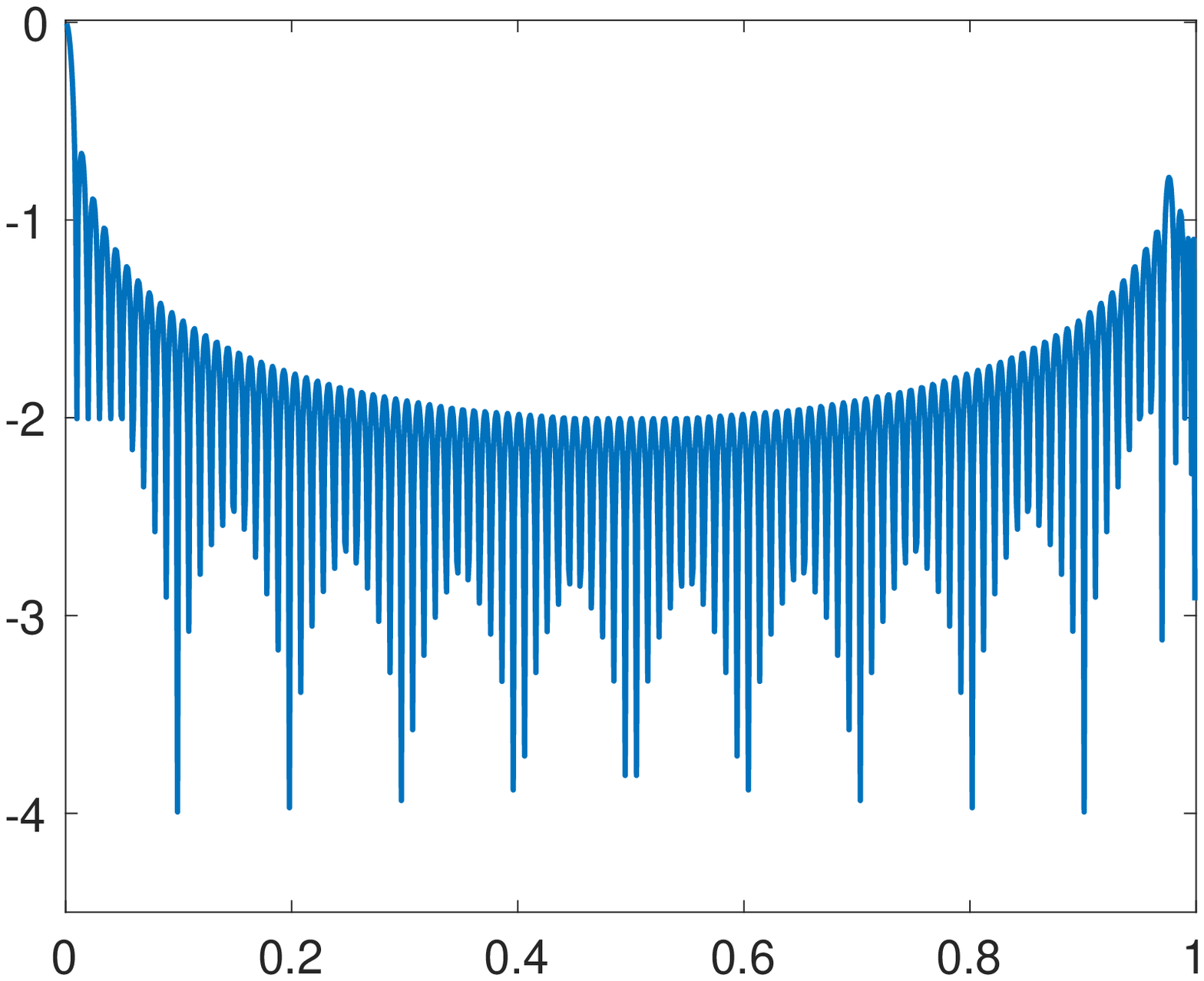} \qquad 
\includegraphics[width=6.5cm]{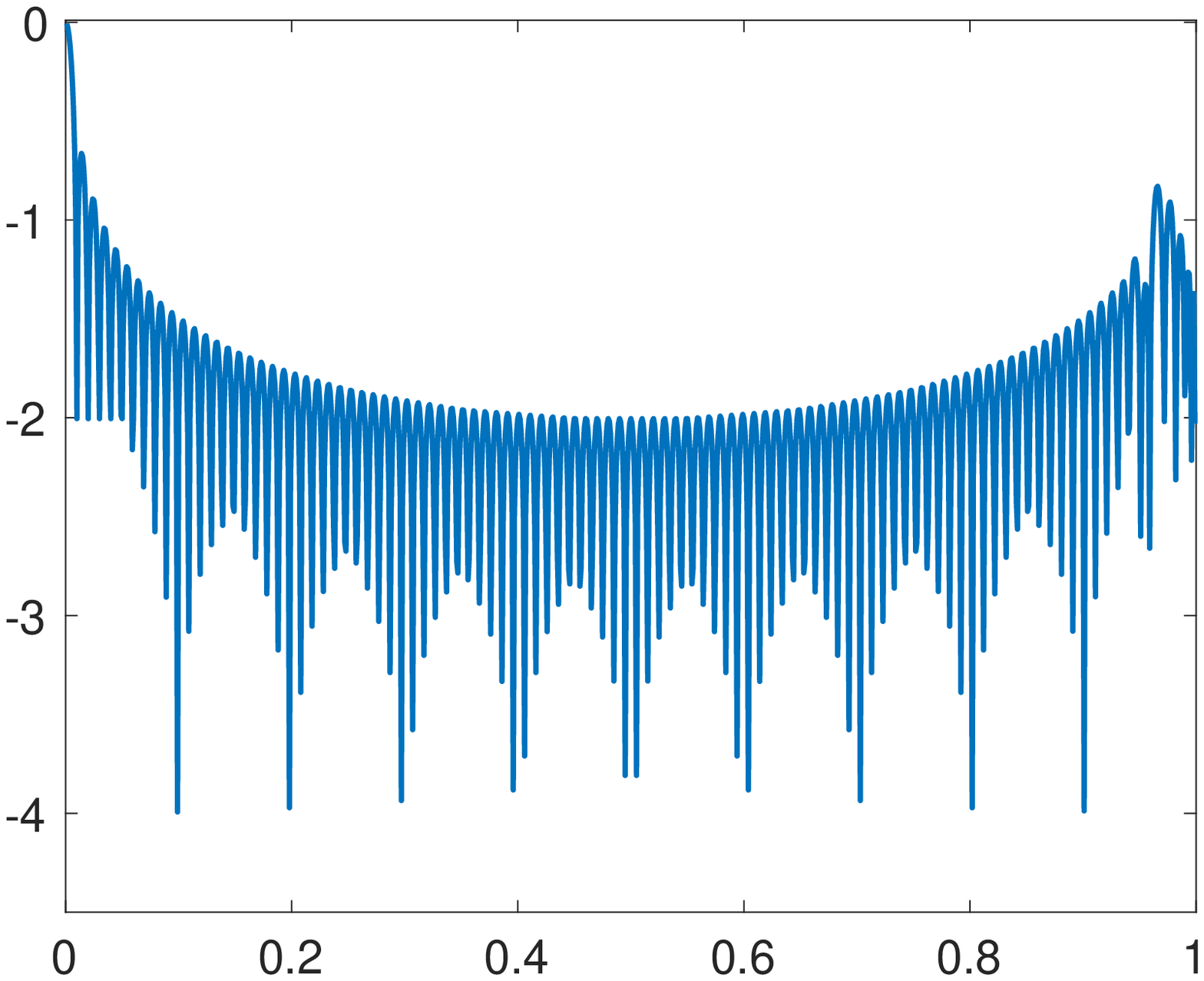}
\centering
\caption{\small The error in (\ref{dappr}) in logarithmic scale for $M=44$ obtained with MPM (left) and ESPRIT (right). }
\label{dirmpmes}
\end{figure}

 In ESPIRA-I, the coefficient vector $\bgamma=(\gamma_j)_{j=1}^{M}$ is computed by Algorithm \ref{alg3} using the Cauchy matrix $\mathbf{C}_{2N,M}$, which has a significantly smaller condition number. Therefore,  ESPIRA-I constructs the exponential sum of order $44$  with the same good approximation error (\ref{dappr}) as in \cite{BM05}. Taking this into account, we also modified ESPIRA-II  for this example according to Remark \ref{cauchyrem}, see Figure \ref{diresp}. Similarly as for approximation of the Bessel function in Example \ref{exbessel}, we obtain special location patterns for the  knots $\mathrm{e}^{\phi_j }$ and the weights $\gamma_j$ in (\ref{dappr}), see Figure \ref{noddir}.

\begin{figure}[h]
\includegraphics[width=6.5cm]{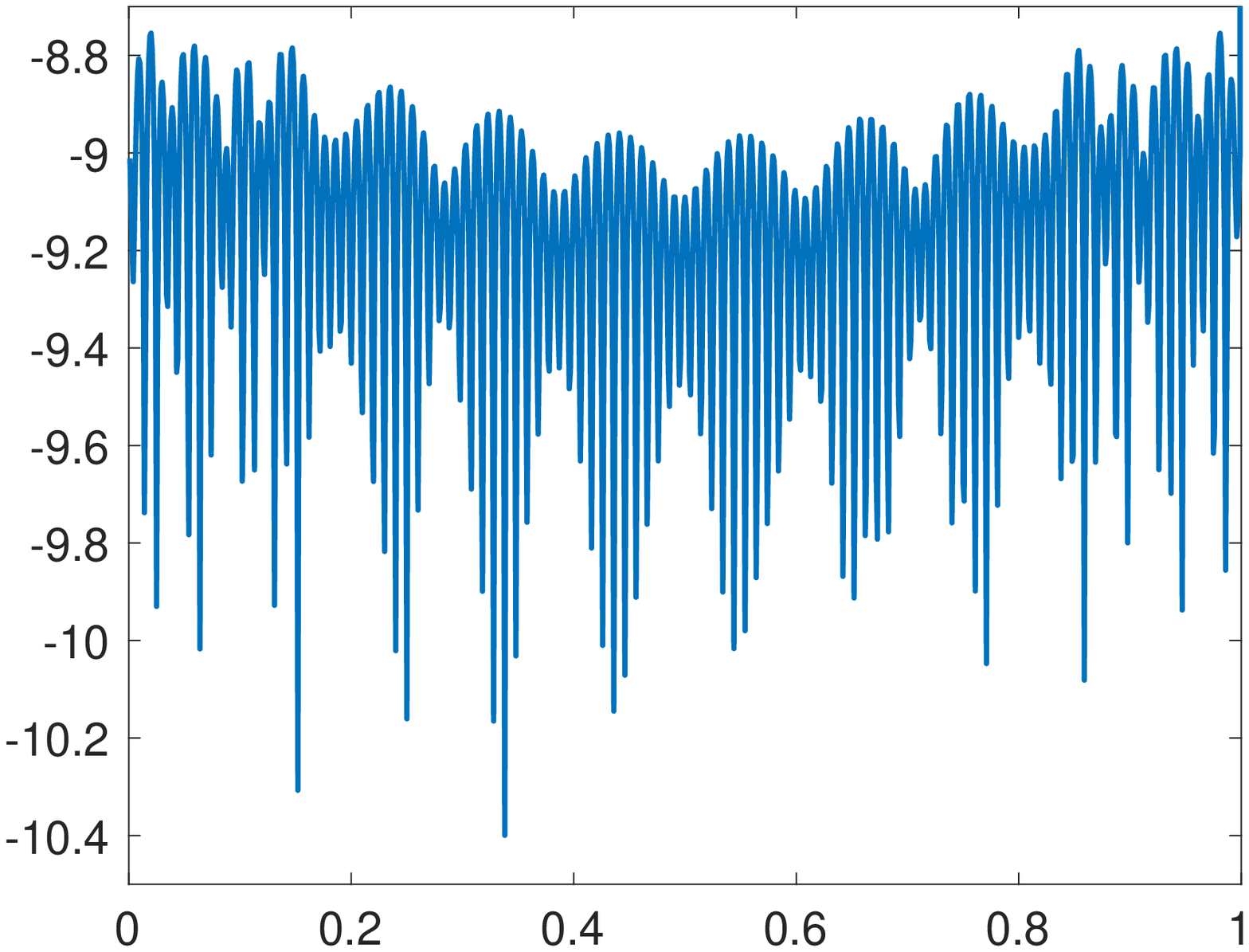} \qquad 
\includegraphics[width=6.5cm]{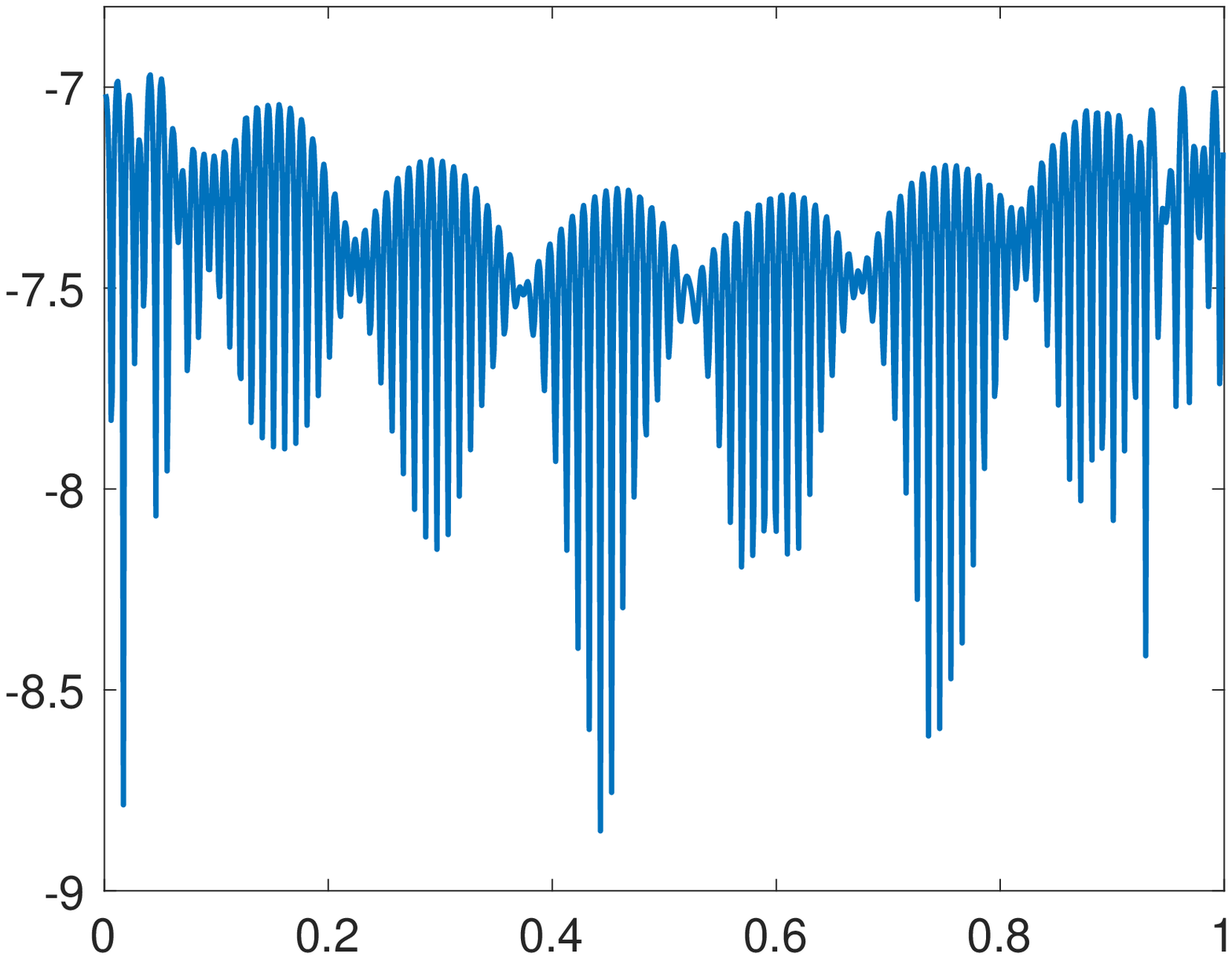}
\centering
\caption{\small The error in (\ref{dappr}) in logarithmic scale for $M=44$ obtained with ESPIRA-I (left) and ESPIRA-II (right). }
\label{diresp}
\end{figure}

\begin{figure}[h]
\includegraphics[width=6.5cm]{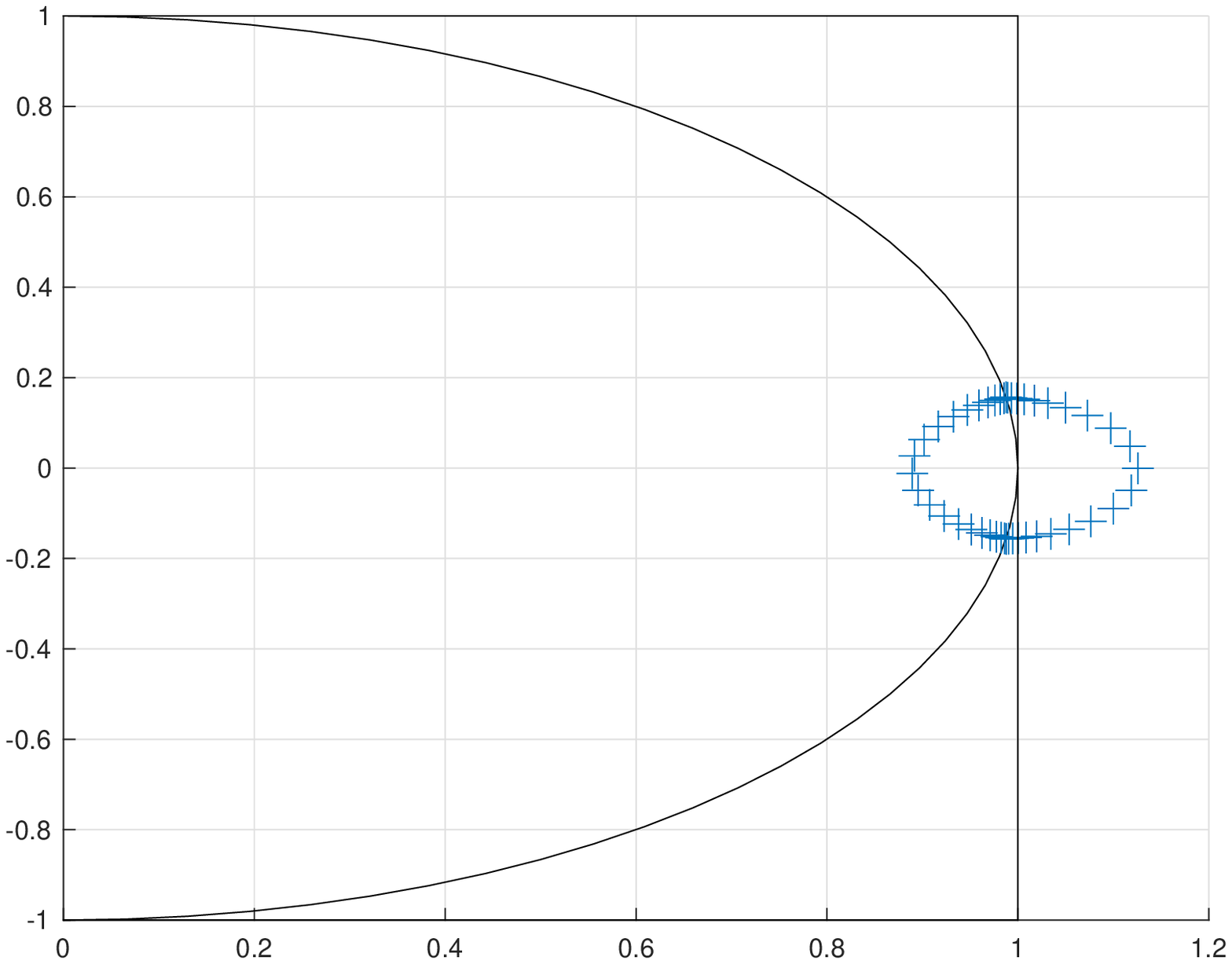} \qquad 
\includegraphics[width=6.5cm]{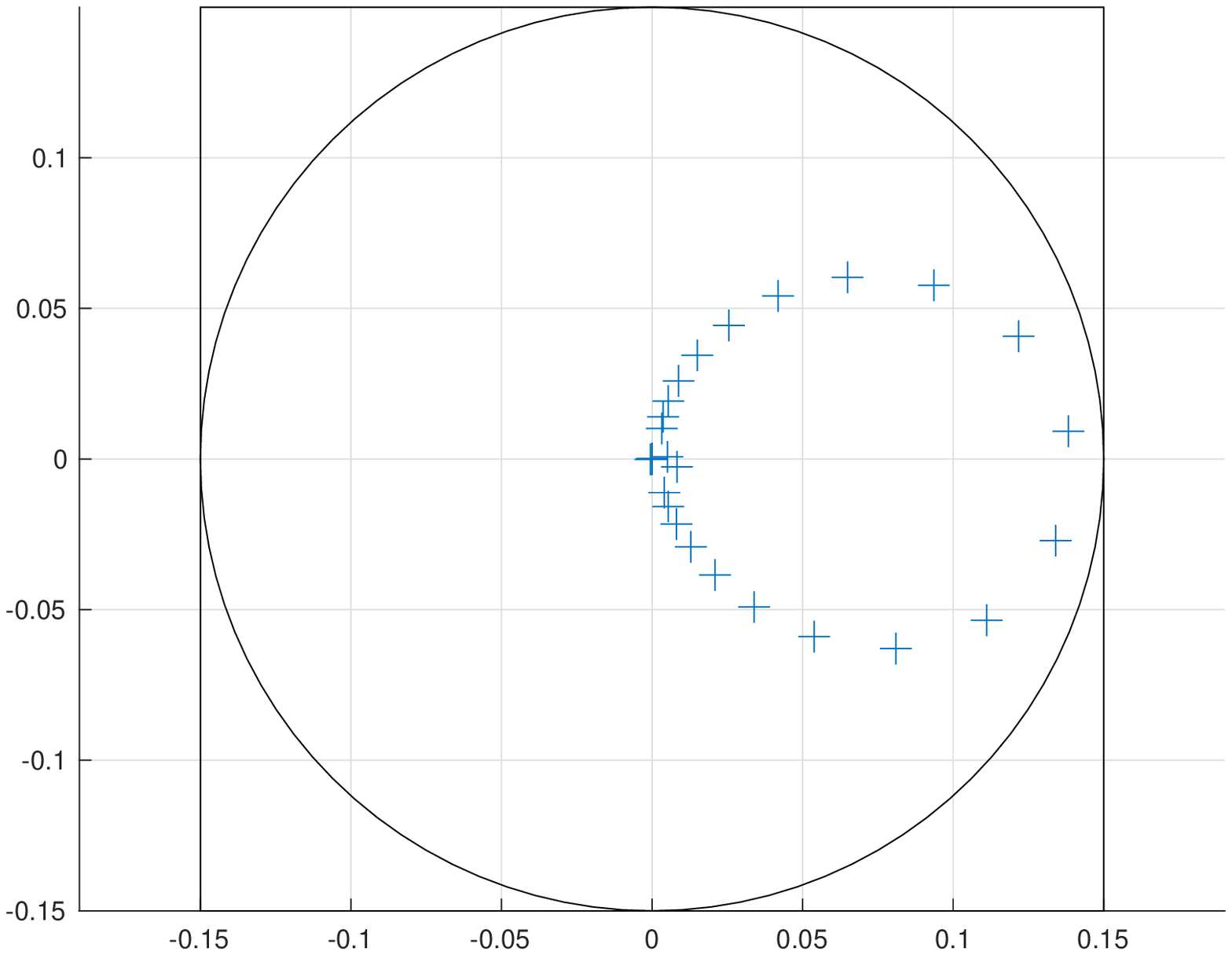}
\centering
\caption{\small Left: complex knots $z_{j}=\mathrm{e}^{ \phi_j}$ and Right: complex weights $\gamma_{j}$ for the approximation of $D_{50}(t)$ by an exponential sum of order $M=44$ on the segment $[0,1]$ obtained with ESPIRA-I. }
\label{noddir}
\end{figure}

\end{example}

\section{Conclusions}

In this paper  we have proposed  two new algorithms, ESPIRA-I and ESPIRA-II, for  stable reconstruction of all parameters of exponential sums $$f(t) = \sum\limits_{j=1}^{M}  \gamma_{j} \, {\mathrm e}^{\phi_{j}t } = \sum\limits_{j=1}^{M}  \gamma_{j} \, z_{j}^{t}$$ from given samples $f_{k} = f(k)$, $k=0, \ldots , 2N-1$ with $N>M$.
Other than the well-known Prony-type methods, such as ESPRIT and MPM, and methods based on (convex) optimization, our algorithms  employ  iterative  rational approximation.
If the knots $z_{j}= {\mathrm e}^{\phi_{j}}$  satisfy $z_{j}^{2N} \neq 1$,  then the discrete Fourier transform vector $\hat{\mathbf f}$ of ${\mathbf f}=(f_{k})_{k=0}^{2N-1}$ possesses a rational structure that is exploited.

ESPIRA-I is directly based on the AAA algorithm, and knots $z_{j}$ with $z_{j}^{2N} = 1$ have to be reconstructed in  a post-processing step.
For ESPIRA-II, we use only the iterative search of index sets to construct tall, well-conditioned Loewner matrices and can show that the parameters $z_{j}$ can be reconstructed from a matrix pencil problem applied to these Loewner matrices. This observation provides the connection to MPM and ESPRIT, where a matrix pencil problem for Hankel matrices is solved.
Therefore ESPIRA-II can be directly applied to reconstruct all knots $z_{j}$ without any case study.

Our numerical experiments show that the reconstruction of parameters from exact data is performed by the ESPIRA algorithms with the same precision as with MPM or ESPRIT, but with less computational effort if $N>M^{2}$.
The ESPIRA algorithms strongly outperform MPM and ESPRIT for reconstruction of  parameters  from noisy input data. We conjecture that these very good reconstruction results are achieved since ESPIRA is (almost) statistically consistent, while MPM and ESPRIT are not. Finally, the ESPIRA algorithms can be applied  for approximation of functions by short exponential sums and deliver good results.
ESPIRA-I provides very accurate results using usual double precision arithmetic for approximation, even where MPM and ESPRIT completely fail  because of ill-conditioned matrices. 
We recommend to use ESPIRA-I if it is known beforehand that the wanted knots $z_{j}$ are not on the unit circle, and to use ESPIRA-II in case of knots on the unit circle.

\section*{Acknowledgement}
The authors would like to thank the reviewers for constructive advices to improve the representation of the paper.
The authors gratefully acknowledge support by the German Research Foundation in the framework of the RTG 2088. The second author acknowledges support by the EU MSCA-RISE-2020 project EXPOWER. 

\small

\end{document}